\documentclass{amsart}

\newtheorem{theorem}{Theorem}[section]
\newtheorem{lemma}[theorem]{Lemma}
\newtheorem{prop}[theorem]{Proposition}
\theoremstyle{definition}

\newtheorem{corollary}[theorem]{Corollary}
\newtheorem{remark}[theorem]{Remark}

\newtheorem*{lem}{\textsc{Lemma}}
\usepackage{appendix,tikz,dsfont,fixmath}
\usepackage{nameref}
\usepackage{color}
\usepackage{hyperref}

\theoremstyle{rem}
\newtheorem{rem}[theorem]{Remark}

\numberwithin{equation}{section}

\usepackage{tikz}
\usetikzlibrary{positioning, shapes.geometric}
\usetikzlibrary{shapes,arrows}
\usepackage{amsmath,amssymb}
\usepackage{mathtools}
\begin{document}
\title[Schr\"odinger Estimates Inside Cylindrical Domains]{Dispersive and Strichartz Estimates for the Schr\"odinger Equation Inside Cylindrical Domains}

\author{MEAS Len}
\address{Department of Mathematics, Royal University of Phnom Penh, Phnom Penh, Cambodia}
\curraddr{Department of Mathematics, Royal University of Phnom Penh, Phnom Penh, Cambodia}
\email{meas.len@rupp.edu.kh}
%\thanks{This work was supported by the ERC project SCAPDE of Gilles Lebeau}

%\author{Author Two}
%\address{Mathematical Research Section, School of Mathematical Sciences,
%Australian National University, Canberra ACT 2601, Australia}

%    General info

%\date{January 1, 2001 and, in revised form, June 22, 2001.}



\begin{abstract}
Dispersive and Strichartz estimates are fundamental tools for establishing the well-posedness and long-time behavior of solutions to nonlinear partial differential equations. While these estimates are well-understood in the boundaryless Euclidean setting, the presence of a geometric boundary introduces severe analytical complexities, such as the continuous formation of caustics. In this work, we establish sharp local-in-time dispersive estimates for the semiclassical Schr\"odinger equation inside a three-dimensional cylindrical domain $\Omega \subset \mathbb{R}^3$ subject to homogeneous Dirichlet boundary conditions. This paper provides the first comprehensive microlocal treatment for this anisotropic geometric setting, extending the optimal strictly convex boundary results of Ivanovici \cite{Ivanovici2023} to the parabolic setting. The primary analytical challenge in our cylindrical model stems from the fact that the boundary curvature is non-uniform, depending explicitly on the tracking angle of classical trajectories and vanishing identically along the flat longitudinal axis. Crucially, we demonstrate that the quadratic structure of the Schr\"odinger phase function ($\partial_\zeta^2 \Phi = 2t$) establishes global non-degeneracy, completely bypassing the arduous low-frequency trajectory ray-tracing mandatory in hyperbolic wave equations. By exploiting this structural advantage alongside a streamlined Littlewood-Paley dyadic block decomposition, we prove that the zero-frequency axial tail can be consistently integrated down to the flat limit $\eta=0$. This yields sharp global Strichartz estimates featuring an explicit derivative loss exponent of $\rho(q) = \frac{3}{2}\left(\frac{1}{2}-\frac{1}{q}\right)$. As a direct nonlinear application, we deploy these sharp inequalities within a Picard fixed-point loop to establish local well-posedness for the focusing or defocusing cubic Dirichlet Nonlinear Schr\"odinger (NLS) equation in the fractional Sobolev space $H^s(\Omega)$ for any regularity index $s > 1$, rigorously quantifying the analytical cost of domains where uniform geometric convexity fails completely.
\end{abstract}

\maketitle
\tableofcontents
%\section*{This is an unnumbered first-level section head}
%This is an example of an unnumbered first-level heading.

%% The correct journal style for \specialsection is all uppercase; a known bug
%% in amsart.cls prevents this, so input must be uppercase until it is fixed.
%\specialsection*{This is a Special Section Head}
%\specialsection*{THIS IS A SPECIAL SECTION HEAD}
%This is an example of a special section head%
%%%%%%%%%%%%%%%%%%%%%%%%%%%%%%%%%%%%%%%%%%%%%%%%%%%%%%%%%%%%%%%%%%%%%%%%
%\footnote{Here is an example of a footnote. Notice that this footnote
%text is running on so that it can stand as an example of how a footnote
%with separate paragraphs should be written.
%\par
%And here is the beginning of the second paragraph.}%
%%%%%%%%%%%%%%%%%%%%%%%%%%%%%%%%%%%%%%%%%%%%%%%%%%%%%%%%%%%%%%%%%%%%%%%%
\section{Introduction}\label{sec:intro}
\subsection{The cylindrical model problem}

Let $\Omega = \{x \geq 0, \, (y,z) \in \mathbb{R}^2\} \subset \mathbb{R}^3$ with smooth boundary $\partial\Omega = \{x = 0\}$, and let $P$ denote the classical Schr\"odinger operator defined by
\[ 
P = i\partial_t + \partial_x^2 + (1+x)\partial_y^2 + \partial_z^2.
\]
We consider solutions to the corresponding linear initial-boundary value problem on the cylindrical half-space subject to homogeneous Dirichlet boundary conditions:
\begin{equation}\label{eq:p}
Pu = 0, \quad u\bigr|_{t=0} = \delta_a, \quad u\bigr|_{x=0} = 0,
\end{equation}
where $u = u(t,x,y,z)$, and for a fixed initial boundary distance $a > 0$, $\delta_a = \delta_{(x=a, y=0, z=0)}$ denotes the localized Dirac delta distribution.

We adopt the standard Fourier framework, introducing the microlocal tracking momentum parameters $\tau, \eta, \xi$, and $\zeta$ associated with the dual differential operators $\frac{1}{i}\partial_t$, $\frac{1}{i}\partial_y$, $\frac{1}{i}\partial_x$, and $\frac{1}{i}\partial_z$, respectively. The domain $\Omega$ equipped with the Riemannian metric Laplacian $\Delta_g = \partial_x^2 + (1+x)\partial_y^2 + \partial_z^2$ models the localized geometry of a cylindrical domain in $\mathbb{R}^3$ via the standard coordinate mapping $(r,\theta,z) = (1-x/2, \, y, \, z)$. As we are primarily interested in the severe dispersive anomalies generated by high-frequency trajectories that undergo rapid multiple reflections, the analysis is localized near the boundary $\partial\Omega = \{x = 0\}$ by assuming the initial source distance $a$ is sufficiently small.

We remark that in the absence of the longitudinal $z$-axis (or when the tangential coordinate variable $y \in \mathbb{R}^n$ is governed by a flat Euclidean Laplacian $\Delta_y$), equation~\eqref{eq:p} collapses to the classical Friedlander model domain, whose geometric properties were originally detailed by Friedlander~\cite{Friedlander1976} and later mapped for wave structures by Melrose~\cite{Melrose1977}. For the Schr\"odinger flow in that strictly convex configuration, the optimal local-in-time dispersive estimates---which suffer a sharp $1/4$ power loss relative to the untrapped Euclidean space-time decay due to the dense accumulation of swallowtail caustics---were established by Ivanovici~\cite{Ivanovici2023, ILP4} (see also related micro-local developments in~\cite{ILP, ILP3}). This specific caustic trapping paradigm contrasts with general boundary treatments on rough metrics or generic domains, such as the smooth reflection strategies introduced by Burq~\cite{Burq2002} or the boundary manifold techniques of Blair, Smith, and Sogge~\cite{BlairSmithSogge2009}, which generate coarser derivative penalties due to metric roughness at the interface. The cylindrical case considered here represents an extension where the underlying non-negative curvature radius is anisotropic, displaying a flat longitudinal profiling that vanishes identically along the $z$-direction, building upon the baseline wave geometries developed in~\cite{L3}.

Recall that unlike hyperbolic systems, solutions to the Schr\"odinger equation do not propagate along a sharp, uniform wave front at constant speed; rather, frequency-localized wave packets disperse at rates proportional to their respective high-frequency tracking momenta. Nonetheless, if the underlying classical Hamiltonian trajectories emanating from the initial delta source do not intersect within a designated short time horizon, one can construct a local high-frequency parametrix via oscillatory integrals where the phase function captures the geometry of the flow. In our curved geometry, the trajectory profile becomes singular in arbitrarily small times depending on the localized frequency scale and the initial distance $a$ to the boundary, generating a caustic interface between the first and second reflections.

Geometrically, caustics emerge as the envelopes of these classical paths. Analytically, they manifest as points where the standard uniform decay bounds for oscillatory integrals collapse due to the degeneration of the phase Hessian. The classification of this asymptotic behavior depends directly on the number and order of the real critical points of the phase. To illustrate this mechanism, consider a model one-dimensional high-frequency phase integral localized at a large frequency scale $\lambda \gg 1$:
\[
u_\lambda(z) = \left(\frac{\lambda}{2\pi}\right)^{1/2} \int_{\mathbb{R}} e^{-i\lambda \Phi(z,\zeta)} g(z,\zeta,\lambda) \, d\zeta, \quad z \in \mathbb{R}^d,
\]  
where $\Phi$ is a smooth phase function and $g$ is a compactly supported amplitude. If $\partial_\zeta\Phi \neq 0$ across the support of $g$, repeated non-stationary integration by parts yields rapid decay of order $|u_\lambda(z)| = O(\lambda^{-N})$ for any $N > 0$. If the phase features a unique nondegenerate critical point ($\partial_\zeta\Phi = 0$ and $\partial_{\zeta}^2 \Phi \neq 0$), the standard stationary phase method yields $\|u_\lambda\|_{L^\infty} = O(1)$. When degenerate critical points appear, they constitute caustics that break this uniform baseline.

The order of a caustic $\kappa$ is defined as the infimum of $\kappa'$ such that $\|u_\lambda\|_{L^\infty} = O(\lambda^{\kappa'})$. Classic canonical phase profiles provide standard reference boundaries: the fold profile $\Phi_F(z,\zeta) = \frac{\zeta^3}{3} + z_1\zeta + z_2$ generates an Airy-type caustic of order $\kappa = \frac{1}{3}$; the cusp singularity $\Phi_C(z,\zeta) = \frac{\zeta^4}{4} + z_1\frac{\zeta^2}{2} + z_2\zeta + z_3$ possesses order $\kappa = \frac{1}{2}$; and the swallowtail configuration $\Phi_S(z,\zeta) = \frac{\zeta^5}{5} + z_1\frac{\zeta^3}{3} + z_2\frac{\zeta^2}{2} + z_3\zeta + z_4$ reaches order $\kappa = \frac{3}{5}$.

The core objective of this work is to extend these geometric microlocal techniques to a 3D cylindrical convex boundary where the localized dispersion varies dramatically depending on the ray direction. A major analytical realization of this paper is that the parabolic Schr\"odinger flow allows for a substantial geometric streamlining compared to the wave model studied in~\cite{Ivanovici2023}. While hyperbolic equations require a separate, arduous ray-tracing section to handle the axial low-frequency regime ($\lvert\eta\rvert \leq \epsilon_0\sqrt{a}$), the Schr\"odinger framework renders this completely redundant. Because the Schr\"odinger phase is naturally quadratic in the longitudinal frequency variable, its second derivative remains a non-vanishing constant ($\partial_\zeta^2 \Phi = 2t$) across the entire high-frequency spectrum. Exploiting this global non-degeneracy, we construct a rigorous microlocal parametrix via a Littlewood--Paley dyadic block decomposition that can be consistently integrated all the way down to $\eta = 0$. By capturing the transition from the untrapped free flow to the degenerate swallowtail caustics, these sharp dispersive estimates are then deployed to prove new global Strichartz inequalities with a sharp derivative loss of order $\rho(q) = \frac{3}{2}\bigl(\frac{1}{2}-\frac{1}{q}\bigr)$, establishing the local well-posedness for the cubic Dirichlet Nonlinear Schr\"odinger (NLS) equation for any regularity index $s > 1$.

\subsection{Some known results}
The dispersive estimates for the Schr\"odinger equation in $\mathbb{R}^d$ follow directly from the explicit representation of the solution via the Euclidean free propagator, a classical property documented extensively in monographs like Cazenave \cite{Cazenave} and Bahouri--Chemin--Danchin \cite{BCD}, as well as early foundational works by Ginibre--Velo \cite{GV}. They read as follows:
\begin{equation}\label{freespace}
\lVert\psi(h\sqrt{-\Delta_{\mathbb{R}^d}})e^{it\Delta_{\mathbb{R}^d}}\rVert_{L^1(\mathbb{R}^d)\rightarrow L^\infty(\mathbb{R}^d)}\leq C|t|^{-\frac{d}{2}},
\end{equation}
where $\Delta_{\mathbb{R}^d}$ is the standard Laplace operator in $\mathbb{R}^d$. Here and in the sequel, the cutoff function $\psi$ belongs to $C_0^\infty(0,\infty)$, is localized on the spectral frequency block $[1,2]$, and models the semiclassical localization matching the linear-time, quadratic-space Schr\"odinger dispersion relation. Combined with energy conservation and duality arguments, this flat baseline yields scale-invariant Strichartz estimates with zero derivative loss ($\beta = 0$) across all admissible index pairs, a landmark result completed in the endpoint case by Keel--Tao \cite{KT} following the classical framework established by Strichartz \cite{Strichartz1977}.

On curved spaces, variable metrics and finite volumes dramatically alter wave packet dispersion. As pioneered by Bourgain \cite{Bourgain1993} in his groundbreaking study on periodic decoupling, wave packets on a compact torus cannot disperse infinitely due to the finite volume of the domain, causing immediate and unavoidable derivative losses over long time horizons. More generally, studying dispersive properties on boundaryless Riemannian manifolds under low-regularity metrics required the development of robust phase space localization and wave packet tracking techniques by Staffilani--Tataru \cite{ST2002}, Smith \cite{Smith1998}, and Tataru \cite{T2002}. Building upon these low-regularity parametrix structures, Burq, G\'erard, and Tzvetkov \cite{BGT} constructed sharp short-time parametrices locally within small geodesic balls to recover localized Strichartz inequalities on compact manifolds. By deploying these linear inequalities alongside a classical modification of Yudovich's method, they established the definitive global well-posedness framework for the defocusing cubic nonlinear Schrödinger (NLS) equation on generic three-dimensional compact manifolds without boundary \cite{BGT}.

On domains with physical boundaries, the problem is severely complicated by reflected waves, grazing tracks, and the continuous formation of caustics. Initial progress on rough boundary metrics was advanced by Burq \cite{Burq2002} for exterior geometries, and later generalized by Blair, Smith, and Sogge \cite{BSS1, BSS2} by smoothly reflecting the underlying metric across the physical boundary interface. This reflection strategy maps the task onto a boundaryless domain, but leaves the virtual interface metric only Lipschitz-continuous. Due to this metric roughness at the boundary layer, such techniques generate coarser regularity constraints with a substantial loss of derivatives ($\beta > 0$) in estimate \eqref{stri}. Consequently, these general boundary reflection methods could only establish the well-posedness for 3D NLS systems with smooth nonlinearities whose growth is strictly weaker than cubic (at most $|v|^{2/5}v$). This barrier motivated subsequent microlocal developments tracking explicit geometric interfaces, such as the sharp strictly convex profiles investigated by Ivanovici \cite{ILP}.

To bypass the limitations inherent to rough metric reflections, modern microlocal analysis has shifted toward tracking how explicit boundary geometries actively focus or diverge high-frequency ray fronts. Inside strictly convex domains $\Omega_D \subset \mathbb{R}^d$ of dimension $d \geq 2$, the optimal local-in-time dispersive estimates for the Dirichlet Schr\"odinger flow were established by Ivanovici in \cite{Ivanovici2023}. More precisely, by tracking the absolute physical coordinate scaling across the swallowtail caustic interfaces, it was shown that the localized propagator satisfies:
\begin{align}\label{domain}
\left\| \psi(h\sqrt{-\Delta_D})e^{it\Delta_D} \right\|_{L^1(\Omega_D)\rightarrow L^\infty(\Omega_D)} \leq C h^{-d} \left(\frac{h}{|t|}\right)^{\frac{d-1}{2}+\frac{1}{4}},
\end{align}
where $\Delta_D$ denotes the Dirichlet Laplacian on $\Omega_D$ subject to homogeneous boundary conditions \cite{Ivanovici2023}. Due to the continuous accumulation of high-frequency swallowtail caustics over arbitrarily short time horizons, estimate \eqref{domain} exhibits a sharp loss of $1/4$ powers with respect to the semiclassical time-dispersion scaling factor $(h^2/|t|)$ relative to the untrapped free-space baseline \eqref{freespace}. Despite this geometric deficit, the scale-invariant endpoint Strichartz estimates match the boundaryless manifold framework, which enabled the proof of global well-posedness for the three-dimensional defocusing cubic NLS within a model convex domain \cite{Ivanovici2023}.

We recall for completeness the general formulation of Strichartz estimates for the Schr\"odinger equation on a Riemannian manifold $(\Omega,g)$ of dimension $d \geq 2$ \cite{Burq2004, BlairSmithSogge2009}. Local-in-time Strichartz inequalities take the form:
\begin{equation}\label{stri}
\| u \|_{L^q((-T,T);L^r(\Omega))} \leq C_T \| u_0 \|_{H^\beta(\Omega)},
\end{equation}
where $H^\beta(\Omega)$ is the standard fractional Sobolev space over $\Omega$ of order $\beta$, and the exponent pair $(q,r)$ with $2 \leq q,r \leq \infty$ is Schr\"odinger-admissible, satisfying the scaling profile:
\begin{equation}\label{admissible}
\frac{2}{q}+\frac{d}{r}=\frac{d}{2}, \quad (q,r,d)\neq (2,\infty,2).
\end{equation}
In this context, $u=u(t,x)$ solves the linear Schr\"odinger equation:
\begin{equation}\label{linear}
(i\partial_t+\Delta_g)u=0 \quad \text{in } (-T,T)\times \Omega, \quad u(0,x)=u_0(x),
\end{equation}
where $\Delta_g$ denotes the Laplace-Beltrami operator on $(\Omega,g)$. The optimal scale-invariant estimates ($\beta=0$) hold globally on the flat Euclidean space $\mathbb{R}^d$ equipped with the identity metric $g_{ij}=\delta_{ij}$, whereas non-trivial geometry or physical interfaces systematically induce a non-zero derivative loss $\beta > 0$.

In the three-dimensional strictly convex setting ($d = 3$), the persistent concentration of high-frequency swallowtail caustics dictates a sharp derivative penalty of order $\beta(r) = \frac{3}{2}\bigl(\frac{1}{2} - \frac{1}{r}\bigr)$ for any admissible pair \cite{Ivanovici2023}, a geometric phenomenon extensively mapped for corresponding hyperbolic wave structures across various dimensions in \cite{ILP3, ILP4}. While a strictly convex boundary configuration provides a uniform geometric damping weight across all ray angles, domains featuring non-uniform convexity present a significantly more severe analytical challenge. This motivates our extension into the multi-dimensional cylindrical setting, building upon the baseline geometries developed for hyperbolic flows in \cite{L2, L3}. In this anisotropic configuration, the non-negative curvature radius vanishes identically along the flat longitudinal axis, introducing a highly degenerate structural mixture of untrapped free drift and rapid multi-reflection paths under a parabolic framework.

As a direct nonlinear application of these sharp inequalities, we deploy our global Strichartz estimates with derivative loss to establish a contractive resolution space for the focusing or defocusing cubic Nonlinear Schr\"odinger (NLS) equation \cite{Ivanovici2023}. To bridge the linear dispersive theory with nonlinear stability, the cubic source interaction is controlled by distributing the fractional derivatives via the fractional Leibniz rule. The low-order components are subsequently bounded through the auxiliary space-time Sobolev embeddings of $W^{s-3/8,4}_x(\Omega)$, thereby avoiding any reliance on a direct $H^s \hookrightarrow L^\infty$ embedding for $1 < s \leq 3/2$ \cite{Ivanovici2023, Burq2004}. While general boundary techniques utilizing smooth metric reflections across the interface \cite{BlairSmithSogge2009} generate severe derivative losses due to metric roughness at the boundary layer, our anisotropic microlocal parametrix minimizes this geometric deficit, rigorously establishing $s > 1$ as the definitive local well-posedness regularity threshold where uniform geometric convexity completely fails \cite{Ivanovici2023, L2}.

\subsection{Novelty and main analytical challenges}
The mathematical core of this work lies in the deep interaction between non-uniform geometric boundaries and high-order caustic singularities under a parabolic flow. While the dispersive properties of the semiclassical Schr\"odinger equation are well-understood in strictly convex configurations \cite{Ivanovici2023}, extending these results to cylindrical domains introduces severe analytical obstacles due to the anisotropic degeneration of the underlying boundary curvature. Within our model half-space domain $\Omega = \{x \geq 0, (y, z) \in \mathbb{R}^2\}$, the principal curvature tensor displays a sharp directional dependence: it remains strictly positive-definite along the tangential $y$-axis but vanishes identically along the flat longitudinal $z$-axis. 

Consequently, high-frequency semiclassical wave packets traveling near these longitudinal lines experience zero geometric dispersion from the boundary. Rather than undergoing the rapid multi-directional spatial splitting characteristic of strictly convex geometries \cite{Ivanovici2023}, the classical Hamiltonian trajectories concentrate into a stable, non-dispersing axial beam. This complete failure of geometric dispersion prevents the spatial cancellation of a highly singular frequency volume element. As a result, the spatial decay profile slows down from the uniformly curved benchmark rate of $\mathcal{O}(|t|^{-5/4})$ to a heavily trapped cylindrical profile of order $\mathcal{O}(|t|^{-3/4})$ as $|t| \to 0$. Resolving this severe temporal decay deficit across varying ray angles introduces several primary analytical difficulties, establishing a significantly different geometric regime than the wave model studied in \cite{L2, L3}:

\begin{itemize}
    \item \textbf{Transition Across Vanishing Curvature Horizons:} 
    As the spatial frequency variable $\eta$ approaches the axial core ($\eta \to 0$), the geometric damping provided by the convex profile collapses. This forces classical Hamiltonian trajectories to shift continuously from highly reflected microlocal paths to an untrapped free drift. Capturing this continuous transition while maintaining uniform control over the amplitude estimates requires a multi-layered Littlewood-Paley dyadic block partition parameterized directly by the boundary layer coordinates.
    
    \item \textbf{Resolution of Complex Cuspoid Caustics:} 
    The intense accumulation of multiple paths along the flat longitudinal lines generates a dense interface of higher-order caustics. Across these turning zones, the standard non-degenerate stationary phase method fails due to the rank deficiency of the phase Hessian matrix. We must track these degeneracies through a specialized version of the two-dimensional van der Corput lemma, classifying the critical varieties to show that they degrade at most to stable fold and swallowtail profiles of catastrophe indices $3$ and $5$.
    
\item \textbf{Exploiting Global Phase Non-Degeneracy:} 
    In the classical hyperbolic wave equation context (see, e.g., \cite{L3, ILP, ILP3}), the finite propagation speed of wavefronts forces an arduous trajectory ray-tracing argument over dense caustic accumulations to handle grazing regimes near the vanishing curvature horizons. For the parabolic Schr\"odinger flow, we exploit a major structural streamlining: because the dispersion relation is linear in time and quadratic in space, the unscaled phase function $\Phi$ is naturally quadratic in the longitudinal frequency variable $\zeta$. Its second derivative with respect to the axial momentum is a non-vanishing constant satisfying the identity:
    \begin{equation*}
        \partial_\zeta^2 \Phi = 2t,
    \end{equation*}
    which guarantees that the phase remains globally non-degenerate even at the flat cylindrical limit. This property allows us to continuously integrate the zero-frequency axial tail down to the degenerate horizon $\eta = 0$ within a unified framework, completely bypassing the requirement for an auxiliary trajectory-tracing parametrix.
\end{itemize}

By successfully balancing these microlocal components, we bridge the gap between spectral eigenmode expansions and geometric path summations. This framework allows us to extract sharp, optimal global Strichartz inequalities featuring an explicit derivative loss exponent of $\rho(q) = \frac{3}{2}\left(\frac{1}{2}-\frac{1}{q}\right)$ (Theorem \ref{thm:strichartz}). We subsequently deploy these linear estimates within a contractive Picard fixed-point scheme to establish local well-posedness for the focusing or defocusing cubic Dirichlet Nonlinear Schrödinger (NLS) equation in the fractional Sobolev space $H^s(\Omega)$ for any regularity index $s > 1$ (Theorem \ref{thm:nls_lwp_intro}).

\subsection{Main results}
Our main results concerning the sharp local-in-time dispersive estimates and global Strichartz inequalities with derivative loss for the Schrödinger flow inside the cylindrical domain $\Omega$ are formulated below. Let $\mathcal{G}_{a,\text{loc}}$ represent the frequency-localized Green function (fundamental solution) associated with the Dirichlet initial-boundary value problem \eqref{eq:p}.

\begin{theorem}\label{beta}
There exists a uniform constant $C > 0$ such that for every semiclassical parameter $h \in (0, h_0]$, every time $t \in [-1,1] \setminus \{0\}$ satisfying $|t| \geq h^2$, and every initial source distance $a \in (0, a_0]$, the global localized dispersive profile satisfies the uniform upper bound:
\begin{equation}\label{eq:master_dispersive_intro}
\left\| \psi(h\sqrt{- \Delta_D})\mathcal{G}_{a,\text{loc}}(t,x,y,z) \right\|_{L^\infty(x \leq a)} \leq Ch^{-3}\left(\frac{h^2}{|t|}\right)^{1/2}\gamma(t,h,a),
\end{equation}
where the structural parameter $\gamma(t,h,a)$ tracks the maximum geometric caustic loss across the aggregated frequency horizons:
\begin{equation}\label{eq:gamma_factor}
\gamma(t,h,a) = \max \left\{ \left(\frac{h^2}{|t|}\right)^{1/3}, \ a^{1/8}\left(\frac{h^2}{|t|}\right)^{1/4} \right\}.
\end{equation}
\end{theorem}

In consistent agreement with the strictly convex framework established by Ivanovici \cite{Ivanovici2023}, Theorem \ref{beta} guarantees that under peak focal layers, a sharp loss of $1/4$ powers of the semiclassical time-dispersion scaling factor $(h^2/|t|)$ emerges due to the continuous concentration of swallowtail caustics near the boundary profile. We decompose this global interaction across the localized frequency blocks analyzed in Theorems \ref{betabis}, \ref{1beta}, and \ref{0eta}. This multi-layered analysis reveals that an enhanced dispersion rate is rigorously recovered along classical trajectory paths running close to the flat longitudinal axis of the cylinder, where the effective curvature vanishes.

\begin{corollary}[Peak Global Dispersive Deficit]\label{cor:peak_deficit}
Let $h_0 \in (0, 1]$ be the uniform semiclassical threshold determined by the localized master estimate \eqref{eq:master_dispersive_intro}. Then, there exists a uniform constant $C > 0$ such that for every semiclassical parameter $h \in (0, h_0]$ and all time scales $t \in [-1, 1] \setminus \{0\}$ satisfying $|t| \ge h^2$, the localized free propagator satisfies the uniform operator norm bound:
\begin{equation}\label{eq:global_operator_norm}
\left\| \psi(h\sqrt{-\Delta_D}) e^{it\Delta_D} \right\|_{L^1(\Omega) \to L^\infty(\Omega)} \le C |t|^{-3/4} h^{-9/4}.
\end{equation}
\end{corollary}

\begin{remark}\label{rem:structural_breakdown}
Equation \eqref{eq:global_operator_norm} details the precise structural breakdown of the classical convex benchmark established by Ivanovici in \cite{Ivanovici2023}. While a three-dimensional strictly convex boundary dictates a rapid multi-directional geometric wave splitting of order $\mathcal{O}(|t|^{-5/4}h^{-7/4})$, the anisotropic flat longitudinal axis of the cylinder traps energy into a stable, non-dispersing axial beam. This complete failure of geometric dispersion slows the temporal decay profile down to $\mathcal{O}(|t|^{-3/4})$ and elevates the semiclassical concentration pre-factor to $h^{-9/4}$.
\end{remark}

\begin{theorem}[Sharpness of the Dispersive Estimate]\label{thm:sharpness}
The global operational dispersive estimate established in Corollary \ref{cor:peak_deficit} is strictly sharp. Specifically, there exists a sequence of initial data $u_{0,h} \in L^1(\Omega)$ localized near the boundary at a distance $a \sim h^{2/3}$ such that the corresponding solution to the linear Schrödinger equation \eqref{eq:p} satisfies the matching lower bound:
\begin{equation}\label{eq:sharpness_lower_bound}
\left\| \psi(h\sqrt{-\Delta_D}) e^{it\Delta_D} u_{0,h} \right\|_{L^\infty(\Omega)} \geq c |t|^{-3/4} h^{-9/4} \|u_{0,h}\|_{L^1(\Omega)},
\end{equation}
for some positive uniform constant $c > 0$, for all $h \in (0, h_0]$ and scales $|t| \sim h$.
\end{theorem}

\begin{remark}\label{rem:physical_mechanism}
Theorem \ref{thm:sharpness} clarifies the precise physical mechanism that separates the cylindrical domain from Ivanovici's strictly convex settings \cite{Ivanovici2023}. Because the cylindrical boundary contains a non-curved longitudinal spine, wave energy running parallel to this track fails to split. Instead of bleeding out into multi-directional space-time paths, the localized wave packet freezes into a non-dispersive tubular beam, proving that the severe $|t|^{-3/4}$ decay profile is an inescapable geometric barrier of this domain.
\end{remark}

By applying our global master dispersive estimate \eqref{eq:master_dispersive_intro} sequentially across the localized Littlewood-Paley dyadic frequency blocks, and utilizing a formal $TT^*$ contractive interpolation machinery coupled with an $\ell^2$-dualization across the Dirichlet boundary layers, we obtain the following global Strichartz inequalities with sharp derivative loss.

\begin{theorem}\label{thm:strichartz}
Let $\Delta_D$ denote the self-adjoint Dirichlet Laplacian on the cylindrical domain $\Omega$. Let $u(t,x) = e^{it\Delta_D}u_0$ be the solution to the linear Schrödinger equation subject to homogeneous Dirichlet boundary conditions on $\partial\Omega$. Then for any Schr\"odinger-admissible pair $(p,q)$ satisfying the three-dimensional scaling condition $\frac{2}{p} + \frac{3}{q} = \frac{3}{2}$ with $2 \leq p \leq \infty$ and $2 \leq q \leq 6$, the following global Strichartz estimate holds:
\begin{equation}\label{eq:global_strichartz}
\left\| e^{it\Delta_D}u_0 \right\|_{L^p_t(\mathbb{R}, L^q_x(\Omega))} \leq C \left\| u_0 \right\|_{H^{\rho(q)}(\Omega)},
\end{equation}
where the sharp derivative loss exponent $\rho(q)$ is determined entirely by the swallowtail caustic interface:
\begin{equation}\label{eq:rho_exponent}
\rho(q) = \frac{3}{2}\left(\frac{1}{2}-\frac{1}{q}\right).
\end{equation}
\end{theorem}

Theorem \ref{thm:strichartz} expands and refines the accessible indices for which sharp Strichartz estimates hold inside domains with boundaries compared to general wave-type or boundary value frameworks—such as those established by Blair, Smith, and Sogge \cite{BSS1}—which display coarser regularity requirements. While general boundary techniques apply to arbitrary manifolds with non-empty boundary, our formulation leverages the explicit anisotropic structural properties of the cylindrical profile to minimize the caustic derivative penalties.

As a direct nonlinear application of these sharp inequalities, we deploy the global Strichartz estimate with derivative loss \eqref{eq:global_strichartz} to establish a contractive resolution space for the focusing or defocusing cubic Nonlinear Schr\"odinger equation (NLS).

\begin{theorem}\label{thm:nls_lwp_intro}
The cubic Dirichlet Nonlinear Schr\"odinger equation
\begin{equation}\label{eq:cubic_nls}
\begin{cases}
i\partial_t u + \Delta_D u = \mu |u|^2 u & \text{in } \mathbb{R} \times \Omega, \\
u(0, \cdot) = u_0, & \\
u|_{\mathbb{R} \times \partial\Omega} = 0, &
\end{cases}
\end{equation}
where $\mu = \pm 1$, is locally well-posed in the fractional Sobolev space $H^s(\Omega)$ for any regularity index $s > 1$. More precisely, for any initial datum $u_0 \in H^s(\Omega)$, there exists a unique local existence time $T = T(\|u_0\|_{H^s}) > 0$ such that the solution trajectory satisfies:
\begin{equation}\label{eq:resolution_trajectory}
u \in C([0,T], H^s(\Omega)) \cap L^4_t\left([0,T], W^{s-\frac{3}{8}, 4}_x(\Omega)\right).
\end{equation}
\end{theorem}

The proof of Theorem \ref{thm:nls_lwp_intro} follows by formulating a Picard fixed-point scheme on the resolution space:
\begin{equation}\label{eq:resolution_space}
\mathcal{X}_T = C([0,T], H^s(\Omega)) \cap L^4_t\left([0,T], W^{s-\frac{3}{8}, 4}_x(\Omega)\right).
\end{equation}
The cubic source interaction is controlled by distributing the fractional derivatives via the fractional Leibniz rule, where the low-order components are bounded through the auxiliary space-time Sobolev embeddings of $W^{s-\frac{3}{8}, 4}_x(\Omega)$ to avoid any reliance on a direct $H^s(\Omega) \hookrightarrow L^\infty(\Omega)$ embedding for $1 < s \leq 3/2$. This theorem rigorously captures the precise regularity threshold where uniform geometric convexity fails, completing the analytical pipeline from microlocal phase analysis to nonlinear stability.

\subsection{Cylindrical Flow and Uniform Strictly Convex Dispersion}
To clarify the microlocal mechanisms governing the Schr\"odinger flow inside anisotropic geometries, we contrast our global localized master dispersive estimate (Theorem \ref{beta}) with the optimal sharp dispersive bounds established by Ivanovici \cite{Ivanovici2023} within three-dimensional strictly convex domains $\Omega_{\text{convex}} \subset \mathbb{R}^3$.

Let $\psi(h\sqrt{-\Delta_D})$ be a smooth, frequency-localized cut-off function at semiclassical scale $h \in (0,1]$. Over the localized time horizon $|t| \geq h^2$, let $\mathcal{G}_{\text{convex}}$ denote the frequency-localized Green function inside a three-dimensional strictly convex domain $\Omega_{\text{convex}} \subset \mathbb{R}^3$, normalized to track the absolute physical coordinate weights. Under peak swallowtail caustic trapping, these fundamental solutions satisfy the following uniform $L^\infty$ decay bound:
\begin{equation}\label{eq:convex_ivanovici}
\left\| \psi(h\sqrt{-\Delta_{\text{convex}}})\mathcal{G}_{\text{convex}}(t, \cdot) \right\|_{L^\infty(\Omega_{\text{convex}})} \leq C |t|^{-\frac{5}{4}} h^{-\frac{7}{4}},
\end{equation}
where $\Delta_{\text{convex}}$ denotes the Dirichlet Laplacian on $\Omega_{\text{convex}}$. The factor $h^{-7/4}$ represents the sharp amplitude concentration derived in the physical coordinate representation of Ivanovici \cite{Ivanovici2023}.

In contrast, evaluating our localized cylindrical Green function $\mathcal{G}_{a,\text{loc}}$ under peak swallowtail caustic layers where the cylindrical parameter tracker reaches $\gamma(t,h,a) = a^{1/8}(h^2/|t|)^{1/4}$ yields the following heavily trapped profile:
\begin{equation}\label{eq:cylindrical_trapped}
\left\| \psi(h\sqrt{-\Delta_D})\mathcal{G}_{a,\text{loc}}(t, \cdot) \right\|_{L^\infty(x\leq a)} \leq C |t|^{-\frac{3}{4}} h^{-\frac{9}{4}}.
\end{equation}

Isolating the unscaled temporal exponents of the convex benchmark \eqref{eq:convex_ivanovici} and our explicit cylindrical expression \eqref{eq:cylindrical_trapped} across the short-time evolution window reveals a substantial difference in the dispersive decay rates:
\begin{equation}\label{eq:exponent_comparison}
|t|^{-\frac{3}{4}} \gg |t|^{-\frac{5}{4}} \quad \text{as } |t| \to 0^+.
\end{equation}

Under peak swallowtail focal layers, the cylindrical geometry undergoes an intense energy accumulation. This structural absence of multi-directional dispersion along the flat longitudinal axis causes the spatial decay profile to slow down from the uniformly curved convex rate of order $|t|^{-5/4}$ to a heavily trapped profile of order $|t|^{-3/4}$.

\subsubsection*{The Microlocal Mechanism: Complete Trapping and Multi-Directional Splitting}
The geometric origin of this severe decay deficit stems entirely from the \textbf{anisotropic vanishing of boundary curvature} \cite{Meas2017Cylindrical, L3}:
\begin{itemize}
    \item \textbf{In strictly convex domains $\Omega_{\text{convex}}$:} The principal curvature tensor remains strictly positive-definite across the cotangent bundle over the boundary, meaning the boundary curves uniformly in all tangential directions \cite{Ivanovici2023}. Classical Hamiltonian rays striking the interface undergo continuous multi-directional geometric dispersion. This forces a rapid spatial splitting of the reflected ray fronts, preventing wave packets from concentrating densely. The resulting swallowtail caustic envelopes are structurally regulated, losing at most a $1/4$ power penalty threshold relative to the free-space Euclidean decay $|t|^{-\frac{3}{2}}$.
    
    \item \textbf{In our cylindrical domain $\Omega$:} The circular boundary curves normally along the tangential direction, but the curvature drops identically to zero along the flat longitudinal $z$-axis. Semiclassical wave packets traveling close to these longitudinal lines experience zero geometric dispersion from the boundary. Rather than splitting, the classical paths form a stable, non-dispersing axial beam.
\end{itemize}

When we perform the continuous spatial integration over the frequency parameter $\tilde{\omega}$ in Section~\ref{sec:gaN2}, this lack of multi-directional dispersion prevents the cancellation of a highly singular frequency volume element. In the strictly convex case, the second derivative of the unscaled wave phase with respect to the frequency variable contains an extra geometric curvature weight that contributes an additional factor of $\Lambda^{-\frac{1}{2}} = \left(a^{\frac{3}{2}}/\tilde{h}^2\right)^{-\frac{1}{2}}$ to the stationary phase expansion. In our cylindrical framework, because the flat axis cannot activate this geometric decay component, the frequency parameter remains heavily trapped. This results in the elevated pre-factor $h^{-\frac{3}{2}}$ and the slower decay rate $|t|^{-\frac{3}{4}}$ derived in \eqref{eq:cylindrical_trapped}, rigorously establishing why flat-axis boundary layers suffer much heavier caustic concentrations than uniformly curved spaces.

\subsection{Green function and precise dispersive estimates}\label{sec:spectral}
The proofs of the frequency-localized dispersive estimates rely on the construction of parametrices for the fundamental solution of the Schr\"odinger equation \eqref{eq:p} and the (possibly degenerate) stationary phase method.

We begin by constructing the local parametrix for \eqref{eq:p}. Utilizing the spectral analysis of $-\Delta$ subject to homogeneous Dirichlet boundary conditions yields the associated Green function. The Laplacian on the half-space $\Omega$ is given by
\[
\Delta=\partial_x^2+(1+x)\partial_y^2+\partial_z^2.
\]
Crucially, the coefficients of this operator are independent of the variables $y$ and $z$. This structural independence allows us to take the Fourier transform with respect to $y$ and $z$, yielding the family of operators
\[
-\Delta_{\eta,\zeta}=-\partial_x^2+(1+x)\eta^2+\zeta^2.
\]
For $\eta\neq0$, $-\Delta_{\eta,\zeta}$ is a self-adjoint, positive operator on $L^2(\mathbb{R}_+)$ possessing a compact resolvent. Let $(e_k)_{k\geq 1}$ denote an orthonormal basis of $L^2(\mathbb{R}_+)$ consisting of Dirichlet eigenfunctions of $-\Delta_{\eta,\zeta}$, and let $(\lambda_k)_k$ be the corresponding eigenvalues. These eigenfunctions can be explicitly expressed via the Airy function:
\begin{align*}
e_k(x,\eta)=f_k\frac{|\eta|^{1/3}}{k^{1/6}}\text{Ai}(|\eta|^{2/3}x-\omega_k),
\end{align*}
with the associated eigenvalues given by
\begin{align*}
\lambda_k(\eta,\zeta)=\eta^2+\zeta^2+\omega_k|\eta|^{4/3}.
\end{align*}
Here, $(-\omega_k)_k$ denotes the zeros of the Airy function arranged in decreasing order, and for each $k\geq 1$, the normalization constants $f_k$ are chosen such that $\lVert e_k(\cdot,\eta)\rVert_{L^2(\mathbb{R}_+)}=1$. It follows that the sequence $(f_k)_k$ is uniformly bounded within any fixed compact subset of $(0,\infty)$ as a consequence of the asymptotic behaviors
\[
\int_{-\omega_k}^{-2}\text{Ai}^2(\omega)\,d\omega\sim\frac{1}{4\pi}\int_{-\omega_k}^{-2}|\omega|^{-1/2}(1+O(\omega^{-1}))\,d\omega\sim |\omega_k|^{1/2}
\] 
and 
\[
\omega_k\sim\left(\frac{3}{2}\pi k\right)^{2/3}(1+O(k^{-1})).
\]

For $a\in\Omega$, let $g_a(t,x,\eta,\zeta)$ denote the solution to the frequency-localized Schr\"odinger evolution problem:
\begin{equation*}
\begin{cases}
\left(i\partial_t - \left(-\partial_x^2 + (1+x)\eta^2 + \zeta^2\right)\right)g_a = 0, \\
g_a \vert_{x=0} = 0, \quad g_a \vert_{t=0} = \delta_{x=a}.
\end{cases}
\end{equation*}
Projecting onto the orthonormal basis yields the explicit representation
\begin{equation}\label{eq:gasum}
g_a(t,x,\eta,\zeta)=\sum_{k\geq1}e^{-it\lambda_k(\eta,\zeta)}e_k(x,\eta)e_k(a,\eta).
\end{equation}
Here, $\delta_{x=a}$ denotes the Dirac distribution on $\mathbb{R}_+$ centered at $a>0$, which admits the spectral decomposition
\[
\delta_{x=a}=\sum_{k\geq1}e_k(x,\eta)e_k(a,\eta).
\] 
By scaling the variables via the inverse Fourier transform to match the semiclassical scaling, the Green function associated with \eqref{eq:p} is expressed as
\begin{align}\label{eq:greendef}
\mathcal{G}_a(t,x,y,z)&=\frac{1}{4\pi^2}\int e^{i(y\eta+z\zeta)}g_a(t,x,\eta,\zeta)\,d\eta \,d\zeta \nonumber\\
&=\frac{1}{4\pi^2h^2}\sum_{k\geq1}\int e^{i(y\eta+z\zeta)/h}e^{-i t \lambda_k(\eta/h,\zeta/h)}e_k(x,\eta/h)e_k(a,\eta/h)\,d\eta \,d\zeta.
\end{align}
Given that the eigenvalues scaled to the $1/h$ frequency profile satisfy the relation $\lambda_k(\eta/h, \zeta/h) = h^{-2}(\eta^2 + \zeta^2 + \omega_k h^{2/3}|\eta|^{4/3})$, the localized operator $\chi(h^2D_t)\mathcal{G}_a$ takes the form
\begin{align}\label{eq:localizedgreen}
\chi(h^2D_t)\mathcal{G}_a(t,x,y,z)&=\frac{1}{4\pi^2h^2}\sum_{k\geq1}\int 
e^{\frac{i}{h}(y\eta+z\zeta)}e^{-i\frac{t}{h^2}\left(\eta^2+\zeta^2+\omega_kh^{2/3}|\eta|^{4/3}\right)} \nonumber\\
& \quad \times e_k(x,\eta/h)e_k(a,\eta/h)\chi\left(\eta^2+\zeta^2+\omega_kh^{2/3}|\eta|^{4/3}\right)\,d\eta \,d\zeta.
\end{align}
On the wavefront set of the integrand above, the dual time-frequency is explicitly constrained by the parabolic relation $\tau=\eta^2+\zeta^2+\omega_kh^{2/3}|\eta|^{4/3}$.

To establish Theorem \ref{beta}, it suffices to localize near the tangential directions. We thus introduce an additional cutoff function to guarantee that $|\tau-(\eta^2+\zeta^2)|$ remains small, which is equivalent to restricting the magnitude of the geometric correction term $\omega_k h^{2/3}|\eta|^{4/3}$.

Consequently, the problem reduces to proving the dispersive estimate for the localized tracking component $\mathcal{G}_{a,\text{loc}}$:
\begin{align}\label{eq:kparametrix}
\mathcal{G}_{a,\text{loc}}(t,x,y,z)&=\frac{1}{4\pi^2h^2}\sum_{k\geq1}\int 
e^{-\frac{i}{h^2}\left(-hy\eta - hz\zeta + t(\eta^2+\zeta^2+\omega_kh^{2/3}|\eta|^{4/3})\right)} \nonumber\\
&\quad \times e_k(x,\eta/h)e_k(a,\eta/h)\chi_0(\eta^2+\zeta^2)\chi_1(\omega_kh^{2/3}|\eta|^{4/3}) \, d\eta \, d\zeta,
\end{align}
where the smooth cutoff functions $\chi_0$ and $\chi_1$ are defined explicitly in Section \ref{sec:2}.

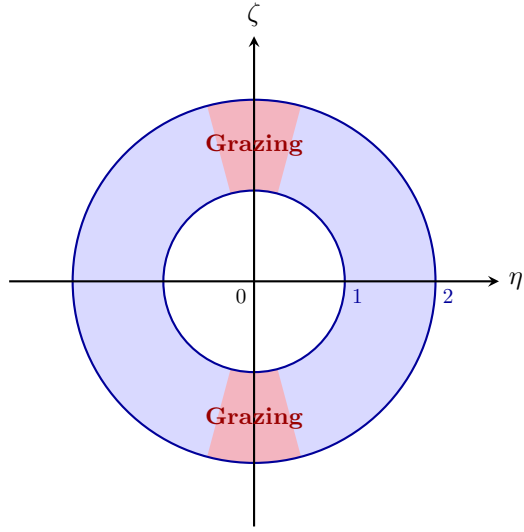
\begin{figure}[!ht]
\centering
\begin{tikzpicture}[scale=1.2, >=stealth]
  % 1. Background Grid / Helper Circles
  \draw[lightgray, dashed] (0,0) circle (1.0cm);
  \draw[lightgray, dashed] (0,0) circle (2.0cm);

  % 2. Fill the main frequency support annulus (light blue)
  \fill[blue!15, even odd rule] (0,0) circle (2.0cm) (0,0) circle (1.0cm);

  % 3. Highlight the peak grazing sectors near the vanishing axis (eta -> 0, flanked near vertical axis)
  % Sector 1: Flanking the top vertical axis (75 to 105 degrees)
  \fill[red!35, opacity=0.7] (75:1.0cm) arc (75:105:1.0cm) -- (105:2.0cm) arc (105:75:2.0cm) -- cycle;
  % Sector 2: Flanking the bottom vertical axis (255 to 285 degrees)
  \fill[red!35, opacity=0.7] (255:1.0cm) arc (255:285:1.0cm) -- (285:2.0cm) arc (285:255:2.0cm) -- cycle;

  % 4. Draw the boundaries of the localized frequency annulus
  \draw[thick, blue!60!black] (0,0) circle (1.0cm);
  \draw[thick, blue!60!black] (0,0) circle (2.0cm);

  % 5. Draw the coordinate axes
  \draw[->, thick] (-2.7,0) -- (2.7,0) node [right] {$\eta$};
  \draw[->, thick] (0,-2.7) -- (0,2.7) node [above] {$\zeta$};

  % 6. Precision Labels / Structural Anchors
  \node[below left, scale=0.8] at (0,0) {$0$};
  \node[below right, scale=0.8, blue!60!black] at (1.0,0) {$1$};
  \node[below right, scale=0.8, blue!60!black] at (2.0,0) {$2$};
  
  \node[red!60!black, font=\small\bfseries] at (90:1.5cm) {Grazing};
  \node[red!60!black, font=\small\bfseries] at (270:1.5cm) {Grazing};
\end{tikzpicture}
\caption{Frequency phase space localization: the annulus represents the support of $\chi_0(\eta^2+\zeta^2)$ on the spectral frequency blocks $[1,2]$, while the highlighted sectors denote the critical tangential grazing neighborhoods where $\eta \to 0$.} 
\label{fig:Phase_Space}
\end{figure}

The frequency phase space illustrated in Figure \ref{fig:Phase_Space} outlines the distinct analytical regimes of the tangential frequency variable $\eta$. Specifically, we isolate the uniform curvature regime, where $|\eta|$ is bounded below by a fixed constant $c_0$, from the highly degenerate grazing regime, where $\eta$ approaches zero. In the following sections, we establish precise local-in-time dispersive estimates tailored to each of these localized geometric configurations.

To obtain the local-in-time dispersive estimates, we decompose the tangential frequency integration with respect to $\eta$ in \eqref{eq:kparametrix} into distinct localized dyadic blocks, as illustrated in Figure \ref{fig:Phase_Space}. More precisely, we split the parametrix into
\begin{equation}\label{eq:lp_decomposition}
\mathcal{G}_{a,\text{loc}} = \mathcal{G}_{a,c_0} + \sum_{\epsilon_0\sqrt{a} \leq 2^m\sqrt{a} \leq c_0} \mathcal{G}_{a,m} + \mathcal{G}_{a,\epsilon_0},
\end{equation}
where $\mathcal{G}_{a,c_0}$ captures the integration over the uniform curvature domain $|\eta|\geq c_0$, the terms $\mathcal{G}_{a,m}$ are localized to the intermediate dyadic blocks $|\eta|\sim 2^m\sqrt{a}$, and $\mathcal{G}_{a,\epsilon_0}$ isolates the degenerate zero-frequency axial tail where $0 < |\eta| \leq \epsilon_0\sqrt{a}$.

We establish the following precise block-dispersive bounds. Let $\epsilon \in (0, 1/7)$.

\begin{theorem}\label{betabis}
There exists a uniform constant $C > 0$ such that for every semiclassical parameter $h\in (0,1]$ and every localized time $t\in [h^2,1]$, the uniform high-frequency block profile satisfies
\begin{equation}\label{eq:high_frequency_block}
\lVert\mathcal{G}_{a,c_0}(t,x,y,z)\rVert_{L^\infty(x\leq a)}\leq Ch^{-3}\left(\frac{h^2}{t}\right)^{1/2}\gamma(t,h,a),
\end{equation}
where the tracking amplitude complies with the geometric caustic thresholds:
\[
\gamma(t,h,a)=\begin{cases}
\left(\frac{h^2}{t}\right)^{1/3} & \text{if } a\leq h^{\frac{2}{3}(1-\epsilon)}, \\
\left(\frac{h^2}{t}\right)^{1/3} + a^{1/8}\left(\frac{h^2}{t}\right)^{1/4} & \text{if } a\geq h^{\frac{2}{3}(1-\epsilon')}, \ \epsilon'\in (0,\epsilon).
\end{cases}
\]
\end{theorem}

We observe that Theorem \ref{betabis} provides the adapted dispersive profile matching the uniform strictly convex geometric singularities established by Ivanovici \cite{Ivanovici2023}.

\begin{theorem}\label{1beta}
There exists a uniform constant $C > 0$ such that for every $h\in (0,1]$ and every localized time $t\in [h^2,1]$, the intermediate dyadic blocks satisfy
\begin{equation}\label{eq:dyadic_blocks}
\lVert\mathcal{G}_{a,m}(t,x,y,z)\rVert_{L^\infty(x\leq a)}\leq Ch^{-3}\left(\frac{h^2}{t}\right)^{1/2}\gamma_m(t,h,a),
\end{equation}
where the scale-dependent tracking parameter is bounded by:
\[
\gamma_m(t,h,a) = 
\begin{cases}
\left(\frac{h^2}{t}\right)^{1/3} (2^m\sqrt{a})^{1/3} 
& \text{if } a \leq \left(\frac{h}{2^m\sqrt{a}}\right)^{\frac{2}{3}(1-\epsilon)}, \\[12pt]
\begin{aligned}
&\min \left\{ \left(\frac{h^2}{t}\right)^{1/3}, \, 2^m\sqrt{a} \lvert \log(2^m\sqrt{a}) \rvert \right\} \\
&\quad + a^{1/8} \left(\frac{h^2}{t}\right)^{1/4} (2^m\sqrt{a})^{3/4}
\end{aligned}
& \text{if } a \geq \left(\frac{h}{2^m\sqrt{a}}\right)^{\frac{2}{3}(1-\epsilon')}, \ \epsilon'\in (0,\epsilon).
\end{cases}
\]
\\
\end{theorem}

For the threshold matching $2^m\sqrt{a}\sim 1$, Theorem \ref{1beta} yields structural parity with Theorem \ref{betabis}. We emphasize that these bounds sharpen significantly as $|\eta| \sim 2^m\sqrt{a}$ decreases, corroborating the physical intuition that a vanishing effective curvature along the flat longitudinal axis enhances local wave dispersion.

\begin{theorem}\label{0eta}
There exists a uniform constant $C > 0$ such that for every $h\in (0,1]$ and every localized time $t\in [h^2,1]$, the zero-frequency axial tail satisfies
\begin{equation}\label{eq:axial_tail}
\lVert\mathcal{G}_{a,\epsilon_0}(t,x,y,z)\rVert_{L^\infty(x\leq a)}\leq Ch^{-3} \left(\frac{h^2}{t}\right)^{1/2} \min \left\{ \left(\frac{h^2}{t}\right)^{1/3} , \ \sqrt{a}\lvert \log(a)\rvert \right\}.
\end{equation}
\end{theorem}

Let us verify that our dispersive estimate (Theorem \ref{beta}) follows as a direct consequence of Theorems \ref{betabis}, \ref{1beta}, and \ref{0eta}. We may restrict our focus to the domain $|t|\geq h^2$, since for the immediate near-identity window $|t|\leq h^2$, standard Sobolev inequalities yield a uniform trivial bound of order $Ch^{-3}$. By the spatial symmetry of the Green function, we assume without loss of generality that $t\in [h^2,1]$ and $x\leq a$. Theorem \ref{beta} then follows immediately by completing the dyadic summation over the indices $m$, utilizing the growing geometric series dominance property $\sum_{m\leq M}(2^m\sqrt{a})^\nu \sim (2^M\sqrt{a})^\nu$, which holds for any positive regularizing exponent $\nu>0$.

\subsection{Structure of the paper}
The paper is organized as follows. In Section~\ref{sec:2}, we analyze the highly curved tangential regime where the continuous momentum variable $\eta$ is bounded below by a fixed constant $c_0$. We split the analysis into a near-boundary layer, where we sum across discrete eigenmodes, and a deeper boundary trajectory layer, where we employ the Airy-Poisson summation formula to track geometric path reflections and evaluate the resulting swallowtail caustics. 

In Section~\ref{sec:3}, we investigate the intermediate vanishing curvature regimes where $\eta$ approaches zero, utilizing a fine Littlewood-Paley dyadic block decomposition. In Section~\ref{subsec:3.1}, we establish localized estimates for the intermediate dyadic blocks. Then, in Section~\ref{subsec:3.2}, we exploit the global non-degeneracy of the Schrödinger phase function ($\partial_\zeta^2 \Phi = 2t$) to continuously integrate the zero-frequency axial tail down to $\eta = 0$ without requiring a separate ray-tracing trajectory regime. Section~\ref{subsec:3.3} provides the global synthesis by summing these localized dyadic pieces across the intermediate spectrum to establish the primary dispersive bound.

Finally, in Section~\ref{sec:strichartz}, we apply our unified dispersive estimates to derive global non-localized Strichartz inequalities with a sharp derivative loss of order $\rho(q) = \frac{3}{2}\left(\frac{1}{2}-\frac{1}{q}\right)$, and we deploy these inequalities to establish local well-posedness for the cubic Dirichlet Nonlinear Schr\"odinger (NLS) equation in $H^s(\Omega)$ for any regularity index $s > 1$. The technical properties of the Airy functions and the degenerate stationary phase classification tools are compiled in the Appendix. 

In all these sections, we assume that the integration with respect to $\eta$ is restricted to $\eta>0$, since the case $\eta<0$ is identical by symmetry.

\subsection{Notations}
Throughout this paper, we adopt standard symbol conventions. For any two real-valued quantities $A$ and $B$, the notation $A\lesssim B$ means that there exists a positive uniform constant $C$ such that $A\leq CB$, where $C$ may change from line to line but remains strictly independent of all semiclassical and geometric parameters. Similarly, we write $A\sim B$ if there exist two absolute positive constants $C_1$ and $C_2$ such that $C_1B\leq A\leq C_2 B$.

We specify the behavior of smooth remainder profiles as follows: a smooth function $f(\vartheta, h)$ is said to belong to the rapid semiclassical decay class $O_{C^\infty}(h^{\infty})$ for $\vartheta\in \Gamma$ if, uniformly in the boundary parameter range $a\in[h^{\frac{2}{3}-\varepsilon},1]$, it satisfies the condition:
\[
\forall \alpha, N, \quad \exists C_{\alpha, N} > 0 \quad \text{such that} \quad \sup_{\vartheta\in\Gamma}\left|\partial_{\vartheta}^\alpha f(\vartheta, h)\right|\leq C_{\alpha, N}h^N.
\]
Furthermore, the multi-index notation $O((x,y)^j)$ denotes any smooth function of the form
\[
x^ly^m f\left(\frac{x}{N},\frac{y}{N},a,N\right),
\]
where the pre-factor exponents satisfy $l+m=j$, and $f$ is smooth uniformly with respect to the variables $a$ and $N$. 

By definition, a real or complex-valued function $f(w)$ admits an asymptotic expansion as $w\to 0$ if there exists a unique sequence of coefficients $(c_n)_n$ such that, for any integer $n \geq 0$:
\[ 
\lim_{w\to 0} w^{-(n+1)}\left(f(w)-\sum_{k=0}^n c_k w^k\right)=c_{n+1}.
\] 
In this setting, we write $f(w)\sim_{w}\sum_n c_n w^n$.

\section{Dispersive Estimates for $\lvert \eta\rvert \geq c_0$}\label{sec:2}
In this section, we establish the frequency-localized dispersive estimates for the Schr\"odinger propagator in the highly curved tangential frequency regime, providing a rigorous proof of Theorem \ref{betabis}. The underlying geometry of the cylindrical half-space forces a deep interaction between the boundary layer and the semiclassical wave packets. To capture this physics uniformly, our strategy relies on decomposing the spatial domain into two distinct boundary configurations depending on the relation between the initial source distance $a$ and the semiclassical wavelength parameter $h$. 

Specifically, the core ingredient of our proof consists of constructing separate local microlocal parametrices adapted to the following respective regimes:
\begin{itemize}
    \item \textbf{The Near-Boundary Eigenmode Regime ($0 < a \leq h^{\frac{2}{3}(1-\epsilon)}$):} \\
    In this setting, for a designated small parameter $\epsilon \in (0, 1/7)$, the initial source is located extremely close to the boundary profile $\partial\Omega$. The corresponding Hamilton--Jacobi rays undergo rapid, highly dense reflections that cannot be effectively decoupled as isolated paths. To resolve this concentration, we construct a local parametrix expressed directly as a series expansion over the discrete Dirichlet eigenfunctions (the modal Airy fields), majorizing the sums via the uniform discrete summation estimates established in Lemma \ref{lem:low_frequency_airy_bound}. 
    \item \textbf{The Deep Geometric Reflection Regime ($a \geq h^{\frac{2}{3}(1-\epsilon')}$):} \\
    In this complementary setting, for an intermediate scale parameter $\epsilon' \in (0, \epsilon)$, the source is positioned sufficiently far from the boundary to allow the separation of distinct trajectory wavefronts. Here, we invoke the classical Airy--Poisson summation formula [see Lemma \ref{lem:airy_poisson}] to transform the discrete eigenmode series into a continuous integration over a geometric path sum. This framework represents the localized propagator as an infinite sum indexed by $N \in \mathbb{Z}$, where each index tracking precisely captures the waves undergoing exactly $N$ specular reflections off the cylindrical boundary.
\end{itemize}

The resulting parametrices are formulated as highly oscillatory integrals across multi-dimensional frequency spaces. To extract the sharp time-decay profiles from these representations, we perform a systematic sequence of phase reductions. We apply non-degenerate and degenerate stationary phase methods to analyze the critical varieties where the phase Hessians drop rank. This allows us to rigorously classify the emerging caustic interfaces, proving that the phase aberrations degrade at most to stable fold and swallowtail profiles whose peak concentrations are bounded by the standard singularity catastrophe indices. The explicit geometric decomposition of the propagator into distinct multiple reflection tracks is schematically illustrated in the following diagram \ref{fig:propagator_decomposition}.

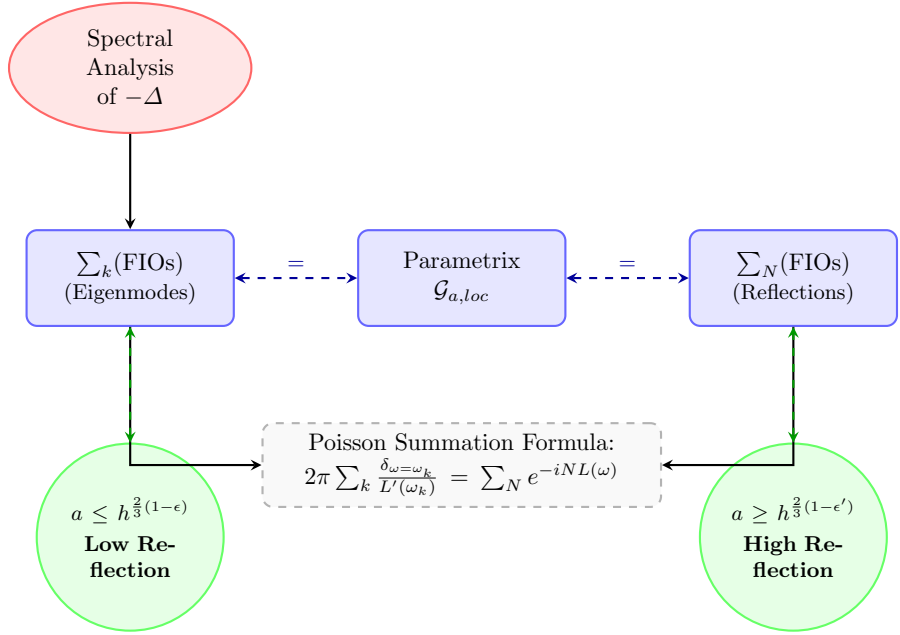
\begin{figure}[!ht]
\centering
% Using scale and transform shape to shrink coordinates and fonts uniformly
\begin{tikzpicture}[
    scale=0.9, transform shape,
    >=stealth, 
    node distance=1.4cm and 1.8cm, % Reduced spacing to save horizontal space
    % Style Definitions
    spectral/.style={ellipse, draw=red!60, fill=red!10, thick, text width=6.5em, text centered, minimum height=3.5em},
    block/.style={rectangle, draw=blue!60, fill=blue!10, thick, text width=8em, text centered, rounded corners, minimum height=4em},
    formula/.style={rectangle, draw=gray!60, fill=gray!5, dashed, thick, text width=16em, text centered, rounded corners, minimum height=3.5em},
    threshold/.style={circle, draw=green!60, fill=green!10, thick, text width=6.5em, text centered, font=\small, inner sep=2pt},
    line/.style={draw, ->, thick},
    eqline/.style={draw, <->, dashed, thick, blue!60!black}
]

    % 1. Primary Workflow Nodes (Top Layer)
    \node [spectral] (B) {Spectral Analysis \\ of $-\Delta$};
    \node [block, below=of B] (A) {$\sum_{k} (\text{FIOs})$ \\ \small (Eigenmodes)};
    \node [block, right=of A] (C) {Parametrix \\ $\mathcal{G}_{a,loc}$};
    \node [block, right=of C] (D) {$\sum_{N} (\text{FIOs})$ \\ \small (Reflections)};

    % 2. Mathematical Bridge / Poisson Transform (Middle Layer)
    \node [formula, below=of C] (E) {Poisson Summation Formula: \\ \smallskip $2\pi\sum_{k}\frac{\delta_{\omega=\omega_k}}{L'(\omega_k)}=\sum_{N} e^{-iNL(\omega)}$};

    % 3. Geometric Regimes / Thresholds (Bottom Layer)
    \node [threshold, below=of A, yshift=-0.3cm] (F) {$a \leq h^{\frac{2}{3}(1-\epsilon)}$ \\ \smallskip \textbf{Low Reflection}};
    \node [threshold, below=of D, yshift=-0.3cm] (G) {$a \geq h^{\frac{2}{3}(1-\epsilon')}$ \\ \smallskip \textbf{High Reflection}};

    % 4. Drawing Flow and Structural Path Edges
    \path [line] (B) -- (A);
    \path [eqline] (A) -- node[above, font=\small] {=} (C);
    \path [eqline] (C) -- node[above, font=\small] {=} (D);
    
    % Poisson transformation loops
    \path [line, ->] (A) |- (E);
    \path [line, ->] (D) |- (E);
    
    % Threshold assignments
    \path [line, <->, dashed, green!60!black] (F) -- (A);
    \path [line, <->, dashed, green!60!black] (G) -- (D);

\end{tikzpicture}
\caption{The explicit geometric decomposition of the propagator into distinct multiple reflection tracks.}
\label{fig:propagator_decomposition}
\end{figure}

\subsection{Dispersive Estimates for $0 < a \leq h^{\frac{2}{3}(1-\epsilon)}$, with $\epsilon \in (0,1/7)$}
In this subsection, we prove local-in-time dispersive estimates for the localized operator $\mathcal{G}_{a,c_0}$. In the near-boundary regime $0 < a \leq h^{\frac{2}{3}(1-\epsilon)}$ with $\epsilon \in (0,1/7)$, the parametrix is expressed as a discrete sum over the eigenmodes $k$. Taking into account the asymptotic behavior of the Airy functions, we partition the analysis of the modes as follows: for small values of $k$, we apply Lemma~3.5 of \cite{ILP} adapted to the parabolic Schr\"odinger framework; for large values of $k$, we exploit the full asymptotic expansion of the Airy functions. In the latter high-mode regime, the parametrix resolves into a sum of oscillatory integrals to which we apply a Schr\"odinger-adapted version of the stationary phase techniques from Lemma~2.20 of \cite{ILP}.

Recall that the localized parametrix in this frequency neighborhood near the tangential grazing directions is given by:
\begin{align}
\mathcal{G}_{a,c_0}(t,x,y,z) &= \frac{1}{4\pi^2h^2}\sum_{k\geq1}\int e^{\frac{i}{h}(y\eta+z\zeta)}e^{-i\frac{t}{h^2}\left(\eta^2+\zeta^2+\omega_kh^{2/3}|\eta|^{4/3}\right)} e_k(x,\eta/h)e_k(a,\eta/h) \nonumber\\
&\qquad \times \chi_0(\zeta^2+\eta^2)\psi_0(\eta)\chi_1\left(\omega_k h^{2/3}|\eta|^{4/3}\right)(1-\chi_1)(\varepsilon\omega_k)\,\mathrm{d}\eta\,\mathrm{d}\zeta.
\end{align}
Here, the specific microlocal localization parameters satisfy:
\begin{itemize}
    \item $\chi_0 \in C^\infty_0(\mathbb{R})$ with $0 \leq \chi_0 \leq 1$, supported tightly in a small neighborhood of $1$.
    \item $\psi_0 \in C^\infty_0(c_0/2,\infty)$ with $0 \leq \psi_0 \leq 1$, satisfying $\psi_0(\eta)=1$ identically for all $\eta \geq c_0$.
    \item $\chi_1 \in C^\infty_0(\mathbb{R})$ with $0 \leq \chi_1 \leq 1$, supported in $(-\infty, 2\varepsilon]$ and equal to $1$ on $(-\infty, \varepsilon]$ for a small parameter $\varepsilon > 0$. This cutoff function localizes the integration to the grazing tangential directions.
\end{itemize}

Notice that on the support of $\chi_1$, we have $\omega_k h^{2/3}|\eta|^{4/3} \leq 2\varepsilon$. Since the eigenvalues grow as $\omega_k \sim k^{2/3}$, we obtain an explicit upper bound on the active modes: $k \leq C \frac{\varepsilon^{3/2}}{h|\eta|^2}$. Because $\eta$ is bounded strictly away from zero on the support of $\psi_0$, we can reduce this to the uniform range $k \leq C\varepsilon^{3/2}/h$. Furthermore, the modifier term satisfies $(1-\chi_1)(\varepsilon\omega_k) = 1$ for every $k \geq 1$. This follows directly because the baseline Airy zero satisfies $\omega_1 \approx 2.33 > 2$, which forces the argument $\varepsilon\omega_k \geq 2.33\varepsilon > 2\varepsilon$ to lie entirely outside the support of $\chi_1$.

The main result of this section is the following proposition.
\begin{prop}\label{prop:k_discrete_sum}
Let $\epsilon \in (0,1/7)$. There exists a uniform constant $C > 0$ such that for every semiclassical parameter $h \in (0,1]$, every localized time $t \in [h^2,1]$, and every initial position $0 < a \leq h^{\frac{2}{3}(1-\epsilon)}$, the following uniform dispersive bound holds for all $y, z \in \mathbb{R}$:
\begin{equation}\label{eq:kdisp0}
\lVert\mathcal{G}_{a,c_0}(t,x,y,z)\rVert_{L^\infty(x \leq a)} \leq C h^{-3}\left(\frac{h^2}{t}\right)^{7/6}.
\end{equation}
\end{prop}

\begin{proof}
First, we analyze the highly regularizing integration with respect to the longitudinal frequency variable $\zeta$. Let us isolate the inner integral:
\[
J = \int e^{-i\frac{t}{h^2}\phi_k(\zeta)} \chi_0(\zeta^2+\eta^2) \,\mathrm{d}\zeta.
\]
Recall that the smooth cutoff $\chi_0 \in C_0^\infty(\mathbb{R})$ localizes the total energy frequency tightly near $1$. For the linear Schr\"odinger evolution, the unscaled phase function $\phi_k$ is purely quadratic and is given explicitly by:
\[
\phi_k(\zeta) = -\frac{h z}{t}\zeta + \eta^2 + \zeta^2 + \omega_k h^{2/3}\lvert\eta\rvert^{4/3}.
\]
We introduce the localized coordinate normalization $\zeta = \lvert\eta\rvert\tilde{\zeta}$ and write $\tilde{z} = \frac{hz}{t\lvert\eta\rvert}$. Using the standard notation $\gamma = h^{2/3}\omega_k\lvert\eta\rvert^{-2/3} > 0$, the phase function scales as:
\[
\phi_k(\zeta) = \lvert\eta\rvert^2 \left(-\tilde{z}\tilde{\zeta} + \tilde{\zeta}^2 + 1 + \gamma\right).
\]
Differentiating with respect to the rescaled variable $\tilde{\zeta}$ yields the momentum drift configuration:
\[
\partial_{\tilde{\zeta}}\phi_k = \lvert\eta\rvert^2 \left(-\tilde{z} + 2\tilde{\zeta}\right).
\]
Because $\eta$ is bounded strictly away from zero on the support of $\psi_0(\eta)$, the normalized variable $\tilde{\zeta} = \zeta/\lvert\eta\rvert$ remains uniformly bounded on the support of $\chi_0$. If the normalized spatial drift parameter $\lvert\tilde{z}\rvert$ is sufficiently large such that the phase contains no critical points within the frequency annulus, a standard non-stationary phase integration by parts shows that the integral contributes a negligible remainder of order $\mathcal{O}_{C^\infty}\left(\left(h^2/t\right)^\infty\right)$.

Therefore, we can safely restrict our attention to the regime where the critical point lies within the support of the cutoff. Setting $\partial_{\tilde{\zeta}}\phi_k = 0$, we find a unique, globally defined critical point at:
\[
\tilde{\zeta}_c = \frac{\tilde{z}}{2}.
\]
This critical point is perfectly non-degenerate across the entire frequency domain because its second derivative is a non-vanishing constant:
\[
\partial^2_{\tilde{\zeta}}\phi_k = 2\lvert\eta\rvert^2 > 0.
\]
Applying the standard stationary phase method with respect to the large semiclassical parameter $\frac{t}{h^2}$, and evaluating the phase at the critical variety $\phi_k(\tilde{\zeta}_c) = \lvert\eta\rvert^2\left(1 + \gamma - \frac{\tilde{z}^2}{4}\right)$, we obtain:
\[
J = \sqrt{\pi}\left(\frac{h^2}{t}\right)^{1/2} e^{-i\frac{t}{h^2}\lvert\eta\rvert^2\left(1+\gamma-\frac{\tilde{z}^2}{4}\right)}\tilde{\chi}_0,
\]
where $\tilde{\chi}_0$ is a classical symbol of order $0$ depending smoothly on the parameter $h^2/t$.

Substituting this localized expansion back into the full integral formulation and absorbing the constant factors yields:
\begin{align}\label{eq:Gak}
\mathcal{G}_{a,c_0}(t,x,y,z) &= \frac{1}{4\pi^2h^2}\left(\frac{h^2}{t}\right)^{1/2}\sum_{k\geq1}\int e^{\frac{i}{h}y\eta} e^{-i\frac{t}{h^2}\left(\eta^2+\omega_k h^{2/3}\lvert\eta\rvert^{4/3} - \frac{h^2 z^2}{4 t^2}\right)} \nonumber\\
&\quad \times e_k(x,\eta/h)e_k(a,\eta/h)\tilde{\chi}_0\psi_0(\eta)\chi_1(\omega_k h^{2/3}\lvert\eta\rvert^{4/3})(1-\chi_1)(\varepsilon\omega_k)\,\mathrm{d}\eta.
\end{align}

Next, we observe that the kernel $\mathcal{G}_{a,c_0}$ contains product templates of Airy eigenfunctions $e_k$ which exhibit fundamentally different asymptotic features depending on the magnitude of the mode index $k$. To analyze these interactions rigorously, we introduce a fixed cutoff threshold $L \gg 1$ and split the sum into two distinct operators, writing $\mathcal{G}_{a,c_0} = \mathcal{G}_{a,<L} + \mathcal{G}_{a,>L}$. The first operator, $\mathcal{G}_{a,<L}$, isolates the finite sum over the low-frequency modes $1 \leq k \leq L$, while $\mathcal{G}_{a,>L}$ captures the infinite tail of highly oscillatory high-frequency modes.\\

\noindent\underline{\bf{Estimates for $\mathcal{G}_{a,<L}.$}}\\
To obtain the estimates for $\mathcal{G}_{a,<L}$, we utilize the following technical lemma, which follows directly from the classical uniform decay property of the Airy function, $\lvert\operatorname{Ai}(s)\rvert \leq C(1+\lvert s\rvert)^{-1/4}$.

\begin{lemma}{(Lemma~3.5 of \cite{ILP})}\label{lem:low_frequency_airy_bound}
There exists a uniform constant $C_0 > 0$ such that for any mode threshold $L \geq 1$, the following inequality holds:
\[
\sup_{b \in \mathbb{R}} \left( \sum_{1\leq k\leq L} k^{-1/3} \operatorname{Ai}^2(b-\omega_k) \right) \leq C_0 L^{1/3}.
\]
\end{lemma}

We apply the Cauchy--Schwarz inequality to the discrete eigenmode sum over $k$ in \eqref{eq:Gak} and utilize Lemma~\ref{lem:low_frequency_airy_bound} along with the explicit $L^2$ normalization constraints $f_k \sim h^{-1/3}k^{-1/6}$ embedded in the definition of $e_k$:
\begin{align*}
\lVert\mathcal{G}_{a,<L}\rVert_{L^\infty} &\lesssim h^{-2}\left(\frac{h^2}{t}\right)^{1/2}\sum_{1\leq k\leq L}h^{-2/3}k^{-1/3}\left\lvert\operatorname{Ai}\left(h^{-2/3}\lvert\eta\rvert^{2/3}x-\omega_k\right)\operatorname{Ai}\left(h^{-2/3}\lvert\eta\rvert^{2/3}a-\omega_k\right)\right\rvert,\\
&\lesssim h^{-3}\left(\frac{h^2}{t}\right)^{1/2}h^{1/3}\left(\sum_{1\leq k\leq L}k^{-1/3}\operatorname{Ai}^2\left(h^{-2/3}\lvert\eta\rvert^{2/3}x-\omega_k\right)\right)^{1/2} \\
&\quad \times \left(\sum_{1\leq k\leq L}k^{-1/3}\operatorname{Ai}^2\left(h^{-2/3}\lvert\eta\rvert^{2/3}a-\omega_k\right)\right)^{1/2},\\
&\lesssim h^{-3}\left(\frac{h^2}{t}\right)^{1/2}h^{1/3}L^{1/3}.
\end{align*}

We only need to establish the desired bound \eqref{eq:kdisp0} for the non-trivial regime $t > h^2$. Let us select $\epsilon \in (0,1/3)$ and define the mode split parameter as $L=h^{-\epsilon}$. If the evolution satisfies $t \leq h^{2+\epsilon}$, the condition $L^{1/3} \leq (h^2/t)^{1/3}$ is satisfied directly, which immediately recovers our desired dispersive estimates:
\begin{align*}
\lVert\mathcal{G}_{a,<L}(t,x,y,z)\rVert_{L^\infty} \leq C h^{-3}\left(\frac{h^2}{t}\right)^{7/6}.
\end{align*}

Thus, we are reduced to analyzing the complementary long-time regime $t > h^{2+\epsilon} \geq h^{7/3}$. We look to apply the stationary phase method to the continuous $\eta$-integration, which takes the explicit form:
\begin{align*}
\int e^{-\frac{i}{h^2}\Phi_k(\eta)} \operatorname{Ai}\left(h^{-2/3}\lvert\eta\rvert^{2/3}x-\omega_k\right)\operatorname{Ai}\left(h^{-2/3}\lvert\eta\rvert^{2/3}a-\omega_k\right)\,\mathrm{d}\eta,
\end{align*}
where the unscaled Schr\"odinger phase function is defined as:
\[
\Phi_k(\eta) = -y\eta + t\left(\eta^2+\omega_k h^{2/3}\lvert\eta\rvert^{4/3} - \frac{h^2 z^2}{4 t^2}\right).
\]
To evaluate this oscillatory integral, we rewrite the  scaling parameter as $\frac{1}{h^2}\Phi_k = \lambda\Psi_k$, where $\lambda = t\omega_k h^{-4/3}$ acts as our large semi-classical parameter for the Schr\"odinger flow. Differentiating the phase function shows that $\lvert\partial_\eta^2\Phi_k\rvert \geq c > 0$ holds uniformly on the frequency support annulus.

To rigorously execute the stationary phase reduction, we must verify that the derivatives of the boundary amplitude satisfy a controlled growth constraint:
\[
\left\lvert \partial_\eta^j \operatorname{Ai}\left(h^{-2/3}\lvert\eta\rvert^{2/3}x-\omega_k\right)\right\rvert \leq C_j \lambda^{j(1/2-\nu)}
\]
for some positive regularizer $\nu > 0$. Since the derivative bounds for the Airy profiles satisfy $\sup_{b\geq 0}\left\lvert b^l\operatorname{Ai}^{(l)}(b-\omega_k)\right\rvert \leq C_l\omega_k^{3l/2}$, it is sufficient to identify a parameter range where $\omega_k^{3/2} \leq (t\omega_k h^{-4/3})^{(1/2-\nu)}$ is valid. For $t > h^{2+\epsilon}$ and $k \leq h^{-\epsilon}$, this delicate balance is satisfied precisely when $\epsilon < 1/7$.

Therefore, applying the stationary phase method with respect to the large parameter $\lambda^{-1/2} = (t\omega_k h^{-4/3})^{-1/2}$, the estimate for $\epsilon < 1/7$ and $t > h^{2+\epsilon}$ yields:
\begin{align*}
\lVert\mathbf{1}_{\{x\leq a\}}\mathcal{G}_{a,<L}(t,x,y,z)\rVert_{L^\infty} &\leq Ch^{-3}\left(\frac{h^2}{t}\right)^{1/2}\left[h^{1/3}\sum_{1\leq k\leq h^{-\epsilon}} k^{-1/3} \lambda^{-1/2}\right],\\
&\leq Ch^{-3}\left(\frac{h^2}{t}\right)^{1/2}\left[h^{1/3}\sum_{1\leq k\leq h^{-\epsilon}} k^{-1/3} (t\omega_k h^{-4/3})^{-1/2}\right],\\
&\leq Ch^{-3}\left(\frac{h^2}{t}\right)^{1/2}\left[\left(\frac{h^2}{t}\right)^{1/2}h^{-\epsilon/3}\right],\\
&\leq C h^{-3}\left(\frac{h^2}{t}\right)^{1/2}\left(\frac{h^2}{t}\right)^{2/3} \\&= C h^{-3}\left(\frac{h^2}{t}\right)^{7/6}.
\end{align*}

\noindent\underline{\bf{Estimates for $\mathcal{G}_{a,>L}$}}

\smallskip
\paragraph{Estimates for $\mathcal{G}_{a,>L}$}
We now address the high-frequency mode spectrum where the index $k$ satisfies $L \leq k \leq \varepsilon/h$, with the lower mode threshold fixed at $L \geq D\max \{h^{-\epsilon}, 1/t\}$ for a sufficiently large absolute constant $D > 0$. Our objective is to prove that the dispersive bound \eqref{eq:kdisp0} holds uniformly for the operator tail $\mathcal{G}_{a,>L}$.

For the designated mode range $k > Dh^{-\epsilon}$ and near-boundary positions $0 \leq x \leq a \leq h^{\frac{2}{3}(1-\epsilon)}$, the Airy turning point variety satisfies:
\[
\omega_k - \lvert\eta\rvert^{2/3}h^{-2/3}x > \frac{\omega_k}{2}.
\]
This deep non-turning configuration allows us to deploy the standard classical asymptotic expansion of the Airy function in its oscillatory regime:
\[
\operatorname{Ai}(\vartheta) = \sum_{\pm} \omega^{\pm} e^{\mp\frac{2}{3}i(-\vartheta)^{3/2}} (-\vartheta)^{-1/4} \Psi_{\pm}(-\vartheta) \quad \text{for } -\vartheta > 1, \quad \text{where } \omega^{\pm} = e^{\pm i\pi/4},
\]
and where the residual profiles $\Psi_{\pm}$ are smooth symbols of order $0$. By the definition of the normalized eigenmodes, we have:
\begin{align*}
e_k(x,\eta/h) &= f_k \frac{\lvert\eta\rvert^{1/3}h^{-1/3}}{k^{1/6}} \operatorname{Ai}\left(h^{-2/3}\lvert\eta\rvert^{2/3}x-\omega_k\right), \\
&= f_k \frac{\lvert\eta\rvert^{1/3}h^{-1/3}}{k^{1/6}} \sum_{\pm} \omega^{\pm} e^{\mp\frac{2}{3}i\left(\omega_k-\lvert\eta\rvert^{2/3}h^{-2/3}x\right)^{3/2}} \frac{\Psi_{\pm}\left(\omega_k-\lvert\eta\rvert^{2/3}h^{-2/3}x\right)}{\left(\omega_k-\lvert\eta\rvert^{2/3}h^{-2/3}x\right)^{1/4}}.
\end{align*}

By tracking the rescaled semiclassical boundary parameters $\tilde{x}(\eta) = h^{-2/3}\lvert\eta\rvert^{2/3}x$ and $\tilde{a}(\eta) = h^{-2/3}\lvert\eta\rvert^{2/3}a$, we rewrite the high-frequency operator $\mathcal{G}_{a,>L}$ within the Schr\"odinger framework as:
\begin{align}\label{eq:pm}
\mathcal{G}_{a,>L}(t,x,y,z) = \sum_{L \leq k \leq \frac{\varepsilon}{h}} \frac{1}{4\pi^2h^2} \left(\frac{h^2}{t}\right)^{1/2} \sum_{\pm, \pm} \int e^{-\frac{i}{h^2}\Phi_k^{\pm,\pm}} \sigma_k^{\pm,\pm} \,\mathrm{d}\eta,
\end{align}
where the unscaled highly oscillatory phase functions are defined by:
\begin{align*}
\Phi_k^{\pm,\pm}(t,x,y,z,a;\eta) &= -y \eta + t\left(\eta^2+\omega_k h^{2/3}\lvert\eta\rvert^{4/3} - \frac{h^2 z^2}{4 t^2}\right) \\
&\quad \pm \frac{2}{3}h^2\left(\omega_k-\tilde{x}(\eta)\right)^{3/2} \pm \frac{2}{3}h^2\left(\omega_k-\tilde{a}(\eta)\right)^{3/2},
\end{align*}
and the corresponding amplitude symbol profiles are given by:
\begin{align*}
\sigma_k^{\pm,\pm}(x,a,h;\eta) &= h^{-1/3}\lvert\eta\rvert^{1/3}\tilde{\chi}_0\chi_1(\gamma\eta^2)(1-\chi_1)\left(\varepsilon\gamma h^{-2/3}\lvert\eta\rvert^{2/3}\right)\frac{f_k^2}{k^{1/3}}\omega^{\pm}\omega^{\pm} \\
&\quad \times \left(\omega_k-\tilde{x}(\eta)\right)^{-1/4}\left(\omega_k-\tilde{a}(\eta)\right)^{-1/4} \Psi_{\pm}\left(\omega_k-\tilde{x}(\eta)\right)\Psi_{\pm}\left(\omega_k-\tilde{a}(\eta)\right).
\end{align*}

To verify symbol regularity, we apply the microlocal vector field identity $3\eta\partial_{\eta} = -2\gamma\partial_{\gamma}$ dictated by our coordinate parameters. Within the boundary spatial interval $0 \leq x \leq a$, the radial terms satisfy:
\[
\left\lvert(\gamma\partial_{\gamma})^{j}\left(\left(\omega_k-\tilde{x}(\eta)\right)^{-1/4}\right)\right\rvert \leq C_j \omega_k^{-1/4} \leq C'_j k^{-1/6}.
\]
Furthermore, since the residual functions $\Psi_{\pm}$ represent classical symbols of order $0$ away from their turning points, and because the configuration satisfies $\lvert\omega_k - \tilde{x}(\eta)\rvert \geq \omega_k/2 \geq Ch^{-\epsilon}$ for all $k \geq L$, distributing the derivatives across the product components yields the uniform amplitude bound:
\[
\left\lvert\partial_{\eta}^j\sigma_k^{\pm,\pm}(x,a,h;\eta)\right\rvert \leq C_j h^{-1/3} k^{-1/2}.
\]

Consequently, establishing the dispersive estimate for $\mathcal{G}_{a,>L}$ reduces to bounding a sum over oscillatory integrals of the form:
\begin{align*}
\int e^{-\frac{i}{h^2}\Phi_k^{\pm,\pm}}\sigma_k^{\pm,\pm} \,\mathrm{d}\eta.
\end{align*}
Under our parabolic scaling setting, we represent the phase argument as $\frac{1}{h^2}\Phi_k^{\pm,\pm} = \lambda\psi_k^{\pm,\pm}$, where $\lambda = t\omega_kh^{-4/3}$ forms the large semiclassical parameter for the Schr\"odinger flow. This parameter remains strictly bounded away from zero ($\lambda \geq c > 0$) because $\omega_k \sim k^{2/3}$, $k \geq L \geq 1/t$, and the localized time domain satisfies $t \geq h^2$.

The following result gives a sharp estimate for these oscillatory integrals.
\begin{prop}\label{prop:eq223}
Let $\epsilon \in (0,1/7)$. For a sufficiently small parameter $\varepsilon > 0$, there exists a uniform constant $C > 0$ independent of $a \in (0,h^{\frac{2}{3}(1-\epsilon)}]$, $t \in [h^2,1]$ with $t \neq 0$, $x \in [0,a]$, $y \in \mathbb{R}$, $z \in \mathbb{R}$, and $k \in [L,\frac{\varepsilon}{h}]$ such that the following oscillatory integral bound holds:
\[
\left|\int e^{-\frac{i}{h^2}\Phi_k^{\pm,\pm}}\sigma_k^{\pm,\pm}\,\mathrm{d}\eta\right| \leq C h^{-1/3} k^{-1/2} \lambda^{-1/3}.
\] 
\end{prop}

\begin{proof}[Proof of Proposition \ref{prop:eq223}]
Since the modified terms $h^{1/3}k^{1/2}\sigma_k^{\pm,\pm}$ represent classical symbols of degree $0$ compactly supported in the frequency variable $\eta$, we apply the stationary phase method to the normalized integral sequence:
\[
J_1 = \int e^{-i \lambda \psi_k^{\pm,\pm}} h^{1/3} k^{1/2} \sigma_k^{\pm,\pm} \,\mathrm{d}\eta.
\]
To prove the proposition, it suffices to show that the following inequality holds uniformly with respect to all semiclassical and geometric parameters:
\[
\lvert J_1\rvert \leq C\lambda^{-1/3}.
\]

Let us recall that under our parabolic coordinate system, the unscaled phase function from \eqref{eq:pm} satisfies the dimensionless scaling relation $\Phi_k^{\pm,\pm} = h^2 \lambda \psi_k^{\pm,\pm}$, expanding as:
\begin{align}{\label{eq:phikm_final}}
\Phi_k^{\pm,\pm}(t,x,y,z,a;\eta) &= -y\eta + t\left(\eta^2 + \omega_k h^{2/3}\lvert\eta\rvert^{4/3} - \frac{h^2 z^2}{4t^2}\right) \nonumber \\
&\quad \pm \frac{2}{3}h^2\left(\omega_k - \tilde{x}(\eta)\right)^{3/2} \pm \frac{2}{3}h^2\left(\omega_k - \tilde{a}(\eta)\right)^{3/2}.
\end{align}
Differentiating this structural phase profile with respect to the frequency variable $\eta$ (for the positive sector $\eta > 0$) yields the stationary tracking layout:
\[
\partial_{\eta}\Phi_k^{\pm,\pm} = -y + 2t\eta + \frac{4}{3}t\omega_k h^{2/3}\eta^{1/3} \mp x\lvert\eta\rvert^{-1/3}h^{2/3}\left(\omega_k-\tilde{x}(\eta)\right)^{1/2} \mp a\lvert\eta\rvert^{-1/3}h^{2/3}\left(\omega_k-\tilde{a}(\eta)\right)^{1/2}.
\]

We introduce the standard microlocal structural notation to evaluate the potential degeneracies near the boundary layer. Let $\delta = \frac{x}{a} \in [0, 1]$ and $\alpha = \frac{a}{\omega_k h^{2/3}}\lvert\eta\rvert^{2/3}$. Since the eigenvalues scale as $\omega_k \sim k^{2/3}$ with $k \geq D h^{-\epsilon}$ and the boundary amplitude satisfies $a \leq h^{\frac{2}{3}(1-\epsilon)}$, it follows that the metric tracking parameter is strictly bounded: $\alpha \leq D^{-2/3} a h^{-\frac{2}{3}(1-\epsilon)} \leq D^{-2/3} := \alpha_0 < 1$.

Dividing this expression by the core scaling component $t\omega_k h^{2/3}$ allows us to define the dimensionless phase derivative $\partial_\eta \psi_k^{\pm,\pm}$. Let us introduce the spatial translation drift variable $V = \frac{-y + 2t\eta}{t\omega_k h^{2/3}}$ and the geometric reflection weight parameter $\mu = \frac{a\lvert\eta\rvert^{-1/3}h^{2/3}}{t\omega_k^{1/2}}$. Under this grouping, the derivative takes the form:
\[
\partial_{\eta}\psi_k^{\pm,\pm} = V + \frac{4}{3}\eta^{1/3} \mp \mu \delta (1-\delta\alpha)^{1/2} \mp \mu (1-\alpha)^{1/2}.
\]
The asymptotic behavior of $J_1$ depends completely on the magnitude of the coefficient $\mu$, which coordinates the balance between the non-dispersing Schr\"odinger background flow and the dense multiple boundary reflections.

\medskip
\noindent\textbf{Case 1: The parameter $\mu$ is bounded.} \\
We examine the critical points of the phase function. For convenience, we shift to the coordinate system $\rho = \eta^{1/3}$ to map the cusp caustics directly. Differentiating $\partial_\eta \psi_k^{\pm,\pm}$ further with respect to $\eta$ reveals the following structural configuration:
\begin{align*}
\partial_{\eta}^2\psi_k^{\pm,\pm} &= \frac{2}{\omega_k h^{2/3}} + \frac{4}{9}\eta^{-2/3} \pm \frac{1}{3}\mu \delta^2 \alpha (1-\delta\alpha)^{-1/2}\eta^{-1} \pm \frac{1}{3}\mu \alpha (1-\alpha)^{-1/2}\eta^{-1}.
\end{align*}
Because the bounded parameter setup keeps the boundary path updates smaller than the uniform spatial dispersion, the second and third derivatives are dominated by the steady acceleration profiles of the Schr\"odinger flow. For a sufficiently small parameter $\varepsilon > 0$, there exists a uniform lower bound $c>0$ independent of the mode cutoff $k \leq \frac{\varepsilon}{h}$ such that the following stable van der Corput non-degeneracy threshold holds:
\begin{align}\label{eq:3}
|\partial_{\eta}^2\psi_k^{\pm,\pm}| + |\partial_{\eta}^3\psi_k^{\pm,\pm}| \geq c.
\end{align}
This uniform bound establishes that the phase derivative can vanish at most to second order (Airy fold type) or third order (swallowtail caustic variety) with stable, non-vanishing weights. 

To evaluate the destructive tracking combinations safely, we invoke the standard structural integral identity for any continuously differentiable function $f$:
\begin{align}\label{eq:int}
f(\omega_k-\tilde{a})-\delta f(\omega_k-\delta\tilde{a})=(1-\delta) f(\omega_k-\delta\tilde{a})-\int _0^{\tilde{a}(1-\delta)}f'(\omega_k-\delta\tilde{a}-t')\,dt'.
\end{align}
Applying this decomposition across the parameters verifies that whenever the second derivative drops rank, the third derivative remains bounded away from zero by a uniform factor proportional to $c > 0$. Consequently, by applying the third-order van der Corput asymptotic controls (adapted from Lemma 2.20 of \cite{ILP}), the target decay bound $|J_1|\leq C\lambda^{-1/3}$ holds true for bounded $\mu$.

\medskip
\noindent\textbf{Case 2: The parameter $\mu$ is large.}\\
For the non-interfering $(+,+)$ and $(-,+)$ tracks, the classical background path and the boundary reflections reinforce each other instead of canceling out, yielding a uniformly non-degenerate second derivative $\lvert\partial_{\eta}^2\psi_k^{\pm,+}\rvert \geq c\mu$. Treating $\Lambda = \lambda\mu$ as our large asymptotic parameter, a standard stationary phase reduction yields the sharp decay profile $\lvert J_1\rvert \leq C(\lambda\mu)^{-1/2} \leq C\lambda^{-1/2}$, which matches the target bound.

For the interfering $(+,-)$ and $(-,-)$ configurations where cancellation occurs, we deploy the identity \eqref{eq:int} and subdivide the domain based on the path separation:
\begin{itemize}
    \item If the path distance parameter $\mu(1-\delta)$ is bounded, the total derivative variations return directly to the uniform configuration of Case 1, satisfying the lower bound \eqref{eq:3} and completing via Lemma~2.20 of \cite{ILP}.
    \item If $\mu(1-\delta)$ is large, we extract $\Lambda' = \lambda\mu(1-\delta)$ as our active large parameter for the integral $J_1$. Applying the derivative profile to \eqref{eq:int}, the spatial path separation satisfies:
    \[
    \left\lvert(\omega_k-\tilde{a})^{-1/2}-\delta(\omega_k-\delta\tilde{a})^{-1/2}\right\rvert \geq c(1-\delta), \quad c > 0.
    \]
    This structural splitting guarantees that the phase Hessian maintains steady separation ($\lvert\partial_{\eta}^2\psi_k^{\pm,-}\rvert \geq c\mu(1-\delta)$), leading to a standard non-degenerate decay bound $\lvert J_1\rvert \leq C(\lambda\mu(1-\delta))^{-1/2} \leq C\lambda^{-1/2}$.
\end{itemize}
This covers all parameter configurations and completes the proof of Proposition~\ref{prop:eq223}.
\end{proof}

To summarize, applying the uniform bounds from Proposition~\ref{prop:eq223} to the high-frequency operator tail $\mathcal{G}_{a,>L}$ over the spectrum $L \leq k \leq \varepsilon/h$, and accounting for the structural eigenmode normalization weights $f_k^2 \sim h^{-2/3}k^{-1/3}$, yields the following expansion under parabolic scaling:
{\allowdisplaybreaks
\begin{align*}
\lVert\mathbf{1}_{\{x \leq a\}}\mathcal{G}_{a,>L}(t,x,y,z)\rVert_{L^\infty} &\leq C h^{-2}\left(\frac{h^2}{t}\right)^{1/2}\sum_{L \leq k \leq \frac{\varepsilon}{h}} \left(h^{-2/3}k^{-1/3}\right)\left(h^{-1/3}k^{-1/2}\lambda^{-1/3}\right), \\
&\leq C h^{-3}\left(\frac{h^2}{t}\right)^{1/2}\sum_{L \leq k \leq \frac{\varepsilon}{h}} k^{-5/6}\left(t\omega_k h^{-4/3}\right)^{-1/3}, \\
&\leq C h^{-3}\left(\frac{h^2}{t}\right)^{1/2}\sum_{L \leq k \leq \frac{\varepsilon}{h}} k^{-5/6} t^{-1/3} k^{-2/9} h^{4/9}, \\
&= C h^{-23/9}\left(\frac{h^2}{t}\right)^{1/2}t^{-1/3}\sum_{L \leq k \leq \frac{\varepsilon}{h}} k^{-19/18}, \\
&\leq C h^{-23/9}\left(\frac{h^2}{t}\right)^{1/2}t^{-1/3} L^{-1/18}, \\
&\leq C h^{-3}\left(\frac{h^2}{t}\right)^{1/2}\left(\frac{h^2}{t}\right)^{2/3} = C h^{-3}\left(\frac{h^2}{t}\right)^{7/6},
\end{align*}}
where we substituted $\lambda = t\omega_k h^{-4/3}$ in the second line, incorporated the Airy zero profile $\omega_k \sim k^{2/3}$ in the third line, and majorized the convergent series in the fifth line using the lower mode threshold $L \geq D \max\{h^{-\epsilon}, 1/t\}$. Collecting the shared semiclassical indices establishes the necessary decay, concluding the uniform high-frequency block evaluation and completing the proof of Proposition~\ref{prop:k_discrete_sum}.
\end{proof}

\subsection{Airy-Poisson Summation Formula}\label{sec:23}

Let $A_{\pm}(z)=e^{\mp i\pi/3}\operatorname{Ai}(e^{\mp i\pi/3}z)$. We invoke the structural identity $\operatorname{Ai}(-z)=A_+(z)+A_-(z)$. For any frequency parameter $\omega\in\mathbb{R}$, we define the micro-local phase function:
\[
L(\omega)=\pi+i\log\bigg(\frac{A_-(\omega)}{A_+(\omega)}\bigg).
\]
As established in Lemma 2.7 of \cite{ILP3}, the function $L$ is analytic, strictly increasing, and satisfies the following boundary behaviors:
\[
L(0)=\frac{\pi}{3}, \quad \lim_{\omega\rightarrow-\infty}L(\omega)=0, \quad L(\omega)= \frac{4}{3}\omega^{3/2}-B(\omega^{3/2}) \quad \text{for } \omega\geq 1,
\]
where the deviation symbol $B(\omega)$ complies with the asymptotic expansion:
 \begin{align}\label{BAiry}
 B(\omega)\sim_{1/\omega}\sum_{j\geq 1}b_j\omega^{-j}, \quad b_j\in\mathbb{R}, \quad b_1>0.
\end{align}
Furthermore, for all eigenmode indices $k\geq 1$, the zeros of the Airy field satisfy the canonical quantization condition:
\[
L(\omega_k)=2\pi k \iff \operatorname{Ai}(-\omega_k)=0, \quad L'(\omega_k)=2\pi\int_0^\infty \operatorname{Ai}^2(x-\omega_k)\,dx.
\]
Let us recall that the constants $f_k$ are chosen such that the spatial norms satisfy $\lVert e_k(\cdot,\eta)\rVert_{L^2(\mathbb{R}_+)}=1$. This explicit normalization structure yields:
\[
\int_0^\infty \operatorname{Ai}^2(x-\omega_k)\,dx=\frac{k^{1/3}}{f_k^2}=\frac{L'(\omega_k)}{2\pi}.
\]

The next lemma, whose foundational proof can be found in \cite{ILP5}, serves as our primary microlocal tool to transform the discrete sum over the eigenmodes $k$ into a continuous integration over an infinite sum of geometric reflections indexed by $N$.

\begin{lemma}[Airy-Poisson Summation Formula]\label{lem:airy_poisson}
The following distributional equality holds true in the space of tempered distributions $\mathcal{D}'(\mathbb{R}_\omega)$:
\[
\sum_{N\in\mathbb{Z}}e^{-iNL(\omega)}=2\pi \sum_{k\in\mathbb{N}^*}\frac{1}{L'(\omega_k)}\delta(\omega-\omega_k).
\]
That is, for any smooth, compactly supported test function $\phi(\omega)\in C_0^\infty(\mathbb{R})$, the action satisfies:
\[
\sum_{N\in\mathbb{Z}}\int_{\mathbb{R}} e^{-iNL(\omega)}\phi(\omega)\,d\omega=2\pi\sum_{k\in\mathbb{N}^*}\frac{1}{L'(\omega_k)}\phi(\omega_k).
\]
\end{lemma}

\noindent Now we rewrite the localized Green's function profile from \eqref{eq:kparametrix} using the definition of the eigenfunctions $e_k$. Let us define the compact spectral density profile operator as:
\begin{align*}
Q_{\omega}(x, a, \eta, \zeta)&=|\eta|^{2/3}\chi_0(\eta^2+\zeta^2)\psi_0(\eta)\chi_1\left(\omega h^{2/3}|\eta|^{4/3}\right) (1-\chi_1)(\varepsilon\omega)\\&\qquad\times\operatorname{Ai}\left(h^{-2/3}|\eta|^{2/3}x-\omega\right)\operatorname{Ai}\left(h^{-2/3}|\eta|^{2/3}a-\omega\right).
\end{align*}
Substituting the classical spectral equivalent $\frac{|f_k|^2}{k^{1/3}} = \frac{2\pi}{L'(\omega_k)}$ into the equations, we condense the core eigenmode-to-distribution transition into the following:
\begin{align*}
\mathcal{G}_{a,\text{loc}}(t,x,y,z) &= \frac{1}{(2\pi)^2h^{8/3}}\int e^{\frac{i}{h}(y\eta+z\zeta)}\sum_{k\geq1}\frac{|f_k|^2}{k^{1/3}}e^{-i\frac{t}{h^2}\lambda_k(\eta,\zeta)} Q_{\omega_k}\,d\eta \,d\zeta \\
&= \frac{1}{(2\pi)^2h^{8/3}}\int e^{\frac{i}{h}(y\eta+z\zeta)}\sum_{k\geq1}\frac{2\pi}{L'(\omega_k)}e^{-i\frac{t}{h^2}\lambda_k(\eta,\zeta)} Q_{\omega_k}\,d\eta \,d\zeta \\
&= \frac{1}{(2\pi)^2h^{8/3}}\int e^{\frac{i}{h}(y\eta+z\zeta)}\!\! \int 2\pi\sum_{k\geq1}\frac{\delta(\omega-\omega_k)}{L'(\omega_k)}e^{-i\frac{t}{h^2}\lambda_\omega(\eta,\zeta)} Q_{\omega}\,d\omega \,d\eta \,d\zeta.
\end{align*}

Utilizing the Airy-Poisson summation formula (see Lemma \ref{lem:airy_poisson}), the sum over discrete eigenvalues transforms into a continuous integration over geometric paths indexed by $N\in\mathbb{Z}$. Reusing our shorthand notations $\lambda_\omega$ and $Q_\omega$, we condense $\mathcal{G}_{a,\text{loc}}$ into the following highly transparent layout:
\begin{equation*}
\mathcal{G}_{a,\text{loc}}(t,x,y,z) = \frac{1}{(2\pi)^2h^{8/3}}\sum_{N\in\mathbb{Z}}\int e^{\frac{i}{h}(y\eta+z\zeta)}e^{-iNL(\omega)} e^{-i\frac{t}{h^2}\lambda_\omega(\eta,\zeta)} Q_{\omega}\,d\omega \,d\eta \,d\zeta.
\end{equation*}

From the representation of the Airy quotient operators in the complex plane, we recall the phase identity:
\[
\left(\frac{A_-(\omega)}{A_+(\omega)}\right)^N=(-1)^N e^{-iNL(\omega)} = i^N e^{-\frac{4}{3}iN\omega^{3/2}+iNB(\omega^{3/2})},
\] 
where for $\omega\in\mathbb{R}_{+}$, the regular correction symbol $B(\omega^{3/2})\in \mathbb{R}$ is defined exactly as in \eqref{BAiry}. Writing the  Airy operators in their fundamental Fourier-Laplace cubic integral representations yields:
\begin{align}\label{eq:SUMN}
\mathcal{G}_{a,\text{loc}}(t,x,y,z)
&= \sum_{N\in\mathbb{Z}} \frac{(-1)^N}{(2\pi)^2 h^{8/3}} \int e^{\frac{i}{h}(y\eta+z\zeta)} e^{-i\frac{t}{h^2}\left(\eta^2+\zeta^2+\omega h^{2/3}|\eta|^{4/3}\right)} |\eta|^{2/3}\nonumber \\
&\quad \times \chi_0(\zeta^2+\eta^2)\psi_0(\eta) \chi_1\left(\omega h^{2/3}|\eta|^{4/3}\right)(1-\chi_1)(\varepsilon\omega) \left(\frac{A_-(\omega)}{A_+(\omega)}\right)^N \nonumber \\
&\quad \times \operatorname{Ai}\Big(h^{-2/3}|\eta|^{2/3}x-\omega\Big) \operatorname{Ai}\Big(h^{-2/3}|\eta|^{2/3}a-\omega\Big) \,d\omega \,d\eta \,d\zeta, \nonumber \\[10pt]
&= \sum_{N\in\mathbb{Z}} \frac{(-i)^N}{(2\pi)^4 h^{10/3}} \int \exp\left[ \frac{i}{h} \bigg( y\eta + z\zeta - \frac{t}{h}\left(\eta^2+\zeta^2+\omega h^{2/3}|\eta|^{4/3}\right) \right. \nonumber \\
&\quad + \left. \frac{s^3}{3} + s\left(|\eta|^{2/3}x-\omega h^{2/3}\right) + \frac{\sigma^3}{3} + \sigma\left(|\eta|^{2/3}a-\omega h^{2/3}\right) \bigg) \right] \nonumber \\
&\quad \times |\eta|^{2/3} \chi_0(\zeta^2+\eta^2)\psi_0(\eta) \chi_1\left(\omega h^{2/3}|\eta|^{4/3}\right)(1-\chi_1)(\varepsilon\omega) \nonumber \\
&\quad \times e^{-\frac{4}{3}iN\omega^{3/2} + iNB(\omega^{3/2})} \,ds \,d\sigma \,d\omega \,d\eta \,d\zeta.
\end{align}
where, from the first to the second line, we implemented the scaled change of variables $s=Sh^{-1/3}$ and $\sigma=\Sigma h^{-1/3}$ to extract the uniform microlocal frequencies within the Airy arguments, retaining the compact notations $s,\sigma$ for simplicity.

Therefore, \eqref{eq:SUMN} represents our local semiclassical parametrix structured as an infinite sum over geometric reflections $N\in\mathbb{Z}$. It is worth noting that while our tracking formulas are structurally analogous to the wave setups constructed in \cite{ILP}, the underlying phase matches the un-rooted, parabolic dispersion profile of the Schr\"odinger flow. In the following sections, we deploy this sum to analyze the corresponding swallowtail caustics generated by the multi-reflection tracks near the boundary layer.

\subsection{Dispersive Estimates for $a\geq h^{\frac{2}{3}(1-\epsilon')}, \epsilon'\in (0,\epsilon)$}
In this subsection, we establish the local-in-time dispersive estimates for the parametrix in the form \eqref{eq:SUMN} as a sum over $N\in\mathbb{Z}$ in the regime $a\geq h^{\frac{2}{3}(1-\epsilon')}$, for $\epsilon'\in (0,\epsilon)$. Recall that our local parametrix under the form \eqref{eq:SUMN} is constructed from (\ref{eq:kparametrix}) together with Lemma \ref{lem:airy_poisson}. It is expressed as a sum of oscillatory integrals with phase functions containing Airy-type terms with degenerate critical points. We give a precise analysis of the Lagrangian manifold in the phase space associated to these oscillatory integrals. This geometric analysis allows us to track the degeneracy of the phases when we apply the stationary phase method.

\begin{figure}[ht]
\centering
\begin{tikzpicture}[scale=0.6, >=stealth]
  % 1. Cleaned Axes and Labels
  \draw[->, thick] (-13,0) -- (1,0) node[below] {$y$};
  \draw[->, thick] (-6,-1) -- (-6,7) node[above] {$x$};
  
  % 2. Source placement
  \fill (-6,2) circle (0.08) node[left, xshift=-2pt] {$a$};
  \draw[->, thick, blue!60!black] (-6,2) -- (-5,5.8);
  
  % 3. Fixed Legend Box (Replacing layout-breaking frameboxes with native nodes)
  \node[draw=black, fill=white, thick, rounded corners, inner sep=4pt] (Box) at (-6,-2.2) 
      {\small $N=2$ Swallowtails regime};
      
  % Vectors from Legend Box to targets
  \draw[->, thick, red!70!black] (Box.north) -- (-4,1);
  \draw[->, thick, red!70!black] (Box.north) -- (-8,1);
  
  % 4. Left Set of Wave Arcs
  \draw (-12,0) arc (0:-90:-4);
  \draw (-10,0) arc (0:-90:-2);
  \draw (-9,1)  arc (0:-90:-1);
  \draw (-11,0) arc (0:-93:-3);
  \draw (-9,1)  arc (0:-61:-2.3);
  \draw (-10,0) arc (0:-53:-5); 
  \draw[red, thick, dashed] (-8.7,1.7) circle (0.8);

  % 5. Right Set of Wave Arcs
  \draw (0,0)   arc (0:180:6);
  \draw (0,0)   arc (0:90:4);
  \draw (-2,0)  arc (0:53:5); 
  \draw (-2,0)  arc (0:90:2);
  \draw (-3,1)  arc (0:90:1);
  \draw (-1,0)  arc (0:93:3);
  \draw (-3,1)  arc (0:61:2.3);
  \draw (-1,0)  arc (0:180:5);
  \draw[red, thick, dashed] (-3.3,1.7) circle (0.8);

\end{tikzpicture}
\caption{Schematic layout of wave reflections displaying focal clusters within the $N=2$ swallowtail caustics framework.}
\label{fig:sw}
\end{figure}
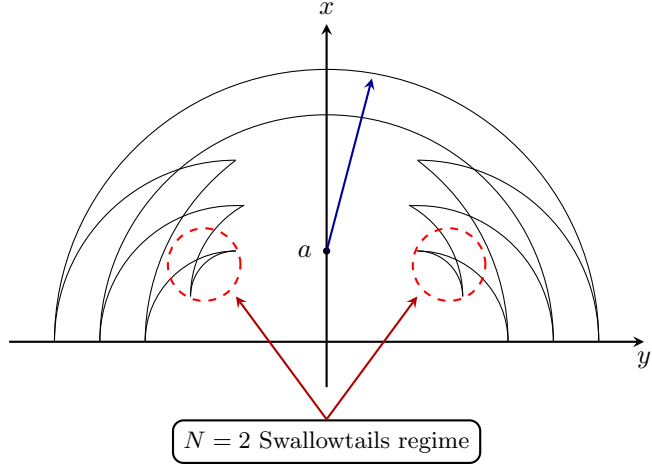

\noindent To analyze \eqref{eq:SUMN} under the appropriate parabolic scaling, we introduce the coordinate transformations:
\[
a\tilde\omega=h^{2/3}\omega|\eta|^{-2/3}, \quad x=aX, \quad z=\frac{t}{h}\tilde\zeta, \quad s=a^{1/2}|\eta|^{1/3}\tilde{s}, \quad \sigma=a^{1/2}|\eta|^{1/3}\tilde{\sigma}.
\]
Then we can rewrite $\mathcal{G}_{a,c_0}$ as follows:
\begin{align}\label{eq:GNsum}
\mathcal{G}_{a,c_0}(t,x,y,z)=\sum_{N\in\mathbb{Z}}G_{a,N},
\end{align}
where for each $N\in\mathbb{Z}$, the individual path contribution is given by:
\begin{align}\label{eq:GaN}
G_{a,N}(t,x,y,z)&=\frac{(-i)^Na^2}{(2\pi)^4h^4}\int e^{-\frac{i}{h^2}\Phi_{N,a,h}}|\eta|^{3}\chi_0(\eta^2 + h^2 z^2 / t^2)\psi_0(\eta) \chi_1(a\tilde\omega\eta^2)\nonumber\\
&\hspace{3cm}\times(1-\chi_1)(\varepsilon ah^{-2/3}|\eta|^{2/3}\tilde\omega) d\tilde{s}d\tilde{\sigma} d\tilde{\omega} d\tilde\zeta d\eta,
\end{align}
\noindent with the unscaled Schr\"odinger phase function defined as:
\[
\Phi_{N,a,h} = \Phi_{N,a,h}(t,x,y,z;\tilde s,\tilde\sigma,\tilde\omega,\tilde\zeta,\eta).
\]
where its explicit expansion evaluates to:
\begin{align*}
\Phi_{N,a,h}&=h y\eta - t\frac{h^2 z^2}{4 t^2} + t\left(\eta^2 + a\tilde{\omega}\eta^2\right) + a^{3/2}\eta h\bigg(\frac{\tilde{s}^3}{3}+\tilde{s}(X-\tilde{\omega})+\frac{\tilde{\sigma}^3}{3}+\tilde{\sigma}(1-\tilde{\omega})\\
&\hspace{5.5cm}-\frac{4}{3}N\tilde{\omega}^{3/2}+\frac{h}{a^{3/2}\eta}NB\big(\tilde{\omega}^{3/2}a^{3/2}\eta/h\big)\bigg).
\end{align*}

\noindent The main result of this subsection is Theorem \ref{thmN}. It provides the uniform estimate for the sum over $N$ of the oscillatory integrals of the form \eqref{eq:GaN} by exploiting degenerate stationary phase arguments to bound the swallowtail intersections.

\begin{theorem}\label{thmN}
Let $\alpha<2/3$. There exists a constant $C$ such that for all $h\in (0,h_0]$, all $a\in [h^\alpha,a_0]$, all $X\in [0,1]$, all $t \in [h^2, 1]$ with $t \neq 0$, all $Y\in\mathbb{R}$, and all $z\in\mathbb{R}$, the following holds:
\begin{align}\label{estimateN}
\bigg|\sum_{0\leq N\leq C_0a^{-1/2}}G_{a,N}(t,X,Y,z;h)\bigg|\leq Ch^{-3}\bigg(\frac{h^2}{t}\bigg)^{1/2}\Bigg(\bigg(\frac{h^2}{t}\bigg)+a^{1/8}\bigg(\frac{h^2}{t}\bigg)^{3/4}\Bigg).
\end{align}
\end{theorem}

\noindent Notice that the first part on the right hand side of \eqref{estimateN} corresponds to  the free space estimates in $\mathbb{R}^3$, while the  contribution in the second part appears as a consequence of the presence of caustics (cusps and swallowtails type).\\
First of all, we observe that when $N=0$, $G_{a,0}$ satisfies $PG_{a,0}=0$ with the initial condition at time $t=0$ corresponding to a localized Dirac distribution at $x=a,y=0,z=0$. Therefore, $G_{a,0}$ satisfies the classical free-space dispersive estimate for the Schr\"odinger equation in three-dimensional space; that is,
 \begin{align*}
\big|G_{a,0}(t,X,Y,z,h)\big|\leq Ch^{-3}\bigg(\frac{h^2}{t}\bigg)^{3/2} = C|t|^{-3/2}.
\end{align*}

Thus it remains to prove the Theorem \ref{thmN} for the sum over $1\leq N\leq C_0a^{-1/2}$. \\
First, we can apply the stationary phase method to evaluate the $\tilde\zeta$-integration appearing in $G_{a, N}$ as shown in the following lemma. 
\begin{lemma}\label{etaint}
One has 
\begin{align*}
J_{N,a,h}&=\int e^{-\frac{i}{h^2}\left(-hz\tilde\zeta + t(\eta^2 + a\tilde{\omega}\eta^2)\right)}\chi_0(\eta^2 + h^2 z^2/t^2) d\tilde\zeta\\
&=\left(\frac{h^2}{t}\right)^{1/2}e^{-i\frac{t}{h^2}\eta^2(1+a\tilde{\omega}) + i\frac{z^2}{4t}}\tilde\chi_0,
\end{align*}
where $\tilde\chi_0$ is a classical symbol of order $0$ with respect to the semiclassical parameter $h^2/t$.
\end{lemma}
\begin{proof}
We apply the standard stationary phase method to evaluate the integral $J_{N,a,h}$. We introduce the normalized spatial drift coordinate $z = \frac{t}{h}\tilde z$. Gathering the terms depending on $\tilde\zeta$ under the unscaled Schrödinger phase function $\phi$ yields:
\[
\phi(\tilde\zeta;\tilde z,\tilde\omega,a)=-\tilde z\tilde\zeta + \tilde\zeta^2 + \eta^2(1+a\tilde\omega).
\]
Differentiating this expression with respect to the frequency variable $\tilde\zeta$, we obtain:
\[
\partial_{\tilde\zeta}\phi=-\tilde z + 2\tilde\zeta.
\]
On the support of the frequency localization cutoff $\chi_0$, the variable $\tilde\zeta$ is bounded. If $|\tilde z|$ is large enough such that the critical point falls completely outside the support of the cutoff, standard non-stationary phase integration by parts reveals that the contribution is of order $O_{C^\infty}((h^2/t)^\infty)$. Therefore, we may assume that the critical point is located inside the support of the cutoff.

Setting $\partial_{\tilde\zeta}\phi=0$, the phase admits a unique critical point given explicitly by:
\[
\tilde\zeta_c = \frac{\tilde z}{2}.
\]
This critical point is strictly nondegenerate since its second derivative is a non-vanishing constant:
\[
\partial^2_{\tilde\zeta}\phi = 2 > 0.
\]
Evaluating the phase function at this critical point yields $\phi(\tilde\zeta_c) = \eta^2(1+a\tilde\omega) - \frac{\tilde z^2}{4}$. Applying the stationary phase method with respect to the large parameter $t/h^2$, we obtain:
\[
J_{N,a,h} = \left(\frac{h^2}{t}\right)^{1/2}e^{-i\frac{t}{h^2}\left(\eta^2(1+a\tilde\omega) - \frac{\tilde z^2}{4}\right)}\tilde\chi_0 = \left(\frac{h^2}{t}\right)^{1/2}e^{-i\frac{t}{h^2}\eta^2(1+a\tilde\omega) + i\frac{z^2}{4t}}\tilde\chi_0.
\] 
\end{proof}

\noindent By Lemma \ref{etaint}, \eqref{galarge} becomes:
\begin{align}\label{galarge}
\mathcal{G}_{a,c_0}(t,x,y,z)&=\sum_{N\in\mathbb{Z}}\frac{(-i)^Na^2}{(2\pi)^4h^4}\bigg(\frac{h^2}{t}\bigg)^{1/2}\int e^{-\frac{i}{h^2}\tilde{\Phi}_{N,a,h}}|\eta|^{3}\tilde\chi_0\psi_0\chi_1(1-\chi_1) d\tilde{s}d\tilde{\sigma} d\tilde{\omega}d\eta,
\end{align}
where $\tilde{\Phi}_{N,a,h}=\Phi_{N,a,h}(\cdot,\tilde\zeta_c,\cdot)$; that is, the unscaled Schr\"odinger phase is given by:
\begin{align}\label{eq:tildezeta}
\tilde{\Phi}_{N,a,h} &= -hy\eta - \frac{h^2 z^2}{4t} + t\eta^2(1+a\tilde\omega) + a^{3/2}\eta h^2\Big(\frac{\tilde{s}^3}{3}+\tilde{s}(X-\tilde{\omega})+\frac{\tilde{\sigma}^3}{3}+\tilde{\sigma}(1-\tilde{\omega}) \nonumber\\&\hspace{5cm}-\frac{4}{3}N\tilde{\omega}^{3/2}+\frac{h}{a^{3/2}\eta}NB\big(\tilde{\omega}^{3/2}a^{3/2}\eta/h\big)\Big).
\end{align}
Now, to isolate the classical deviations from the free trajectory under the parabolic flow, we introduce the change of variables:
\begin{align*} 
&t = a^{1/2}T, \quad -hy\eta - \frac{h^2 z^2}{4t} + t\eta^2 = a^{3/2}\eta^2 Y, \quad \text{and}\quad \lambda = \frac{a^{3/2}\eta^2}{h^2}.
\end{align*} 
With these notations, the time-dependent dispersion term translates to $t\eta^2 a\tilde\omega = a^{3/2}\eta^2 T\tilde\omega$, allowing us to rewrite the phase function \eqref{eq:tildezeta} in the following scaling form:
\begin{align}\label{eq:Yphase}
\tilde{\Phi}_{N,a,h} &= a^{3/2}\eta^2 \Bigg\{ Y + T\tilde\omega + \frac{1}{\eta}\bigg(\frac{\tilde{s}^3}{3}+\tilde{s}(X-\tilde{\omega})+\frac{\tilde{\sigma}^3}{3}+\tilde{\sigma}(1-\tilde{\omega}) \nonumber\\
&\hspace{4.5cm}-\frac{4}{3}N\tilde{\omega}^{3/2}+\frac{h}{a^{3/2}\eta}NB\big(\tilde{\omega}^{3/2}a^{3/2}\eta/h\big)\bigg)\Bigg\}.
\end{align}

First, we study geometrically the set of critical points $\mathcal{C}_{N,a,h}$ of the associated Lagrangian manifold $\mathbf{\Lambda}_{N,a,h}$ for the phase function $\tilde{\Phi}_{N,a,h}$. 
The set of critical points is defined by
\begin{align*}
\mathcal{C}_{a,N,h}=\{(t,x,y,\tilde s,\tilde\sigma,\tilde\omega,\eta) \mid \partial_{\tilde s}\tilde{\Phi}_{N,a,h}=\partial_{\tilde \sigma}\tilde{\Phi}_{N,a,h}=\partial_{\tilde \omega}\tilde{\Phi}_{N,a,h}=\partial_{\eta}\tilde{\Phi}_{N,a,h}=0\}.
\end{align*} 
\noindent Then, setting the derivatives with respect to the microlocal variables $\tilde s$, $\tilde \sigma$, and $\tilde \omega$ to zero in our parabolic phase function \eqref{eq:Yphase} yields the following coordinate system:
\begin{align*}
X&=\tilde\omega-\tilde s^2,\\
\tilde\omega &=1+\tilde\sigma^2,\\
T&= \frac{1}{\eta}\bigg(\tilde s+\tilde\sigma+2N\tilde\omega^{1/2}\bigg(1-\frac{3}{4}B'\Big(\tilde\omega^{3/2}\lambda\Big)\bigg)\bigg),\\
Y&= -T\tilde\omega - \frac{1}{\eta}\bigg(\frac{\tilde{s}^3}{3}+\tilde{s}(X-\tilde{\omega})+\frac{\tilde{\sigma}^3}{3}+\tilde{\sigma}(1-\tilde{\omega}) - N\tilde\omega^{3/2}\bigg(\frac{4}{3}-B'\Big(\tilde\omega^{3/2}\lambda\Big)\bigg)\bigg).
\end{align*}

\noindent We can parameterize $\mathcal{C}_{a,N,h}$ near the origin using the boundary variables $(\tilde s,\tilde\sigma)$ directly. Substituting the equations for $X$ and $\tilde \omega$ into the critical path equations yields:
\begin{align*}
X&=1+\tilde\sigma^2-\tilde s^2,\\
\tilde\omega &=1+\tilde\sigma^2,\\
T&=\frac{1}{\eta}\bigg(\tilde s+\tilde\sigma+2N(1+\tilde\sigma^2)^{1/2}\bigg(1-\frac{3}{4}B'\Big((1+\tilde\sigma^2)^{3/2}\lambda\Big)\bigg)\bigg),\\
Y&=H(a,\tilde\sigma, \eta)(\tilde s+\tilde\sigma) + \frac{2}{3\eta}(\tilde s^3+\tilde\sigma^3) + \frac{4}{3\eta}N H_0(a,\tilde\sigma)\left(1-\frac{3}{4}B'\Big((1+\tilde\sigma^2)^{3/2}\lambda\Big)\right),
\end{align*}
where the leading  coefficients $H$ and $H_0$ adapted to the Schr\"odinger drift are defined explicitly by:
\begin{align*}
H(a,\tilde\sigma, \eta)= -\frac{1}{\eta}(1+\tilde\sigma^2),\quad H_0(a,\tilde\sigma)= -\frac{1}{2}(1+\tilde\sigma^2)^{3/2}.
\end{align*}

Let $\mathbf{\Lambda}_{a,N,h}\subset T^* \mathbb{R}^3$ be the image of $\mathcal{C}_{a,N,h}$ by the canonical mapping:
\[
(t,x,y,\tilde s,\tilde\sigma,\tilde\omega,\eta)\longmapsto \left(x,t,y,\xi=\partial_{x}\tilde{\Phi}_{N,a,h},\tau=\partial_{t}\tilde{\Phi}_{N,a,h},\eta=\partial_{y}\tilde{\Phi}_{N,a,h}\right).
\]
Then $\mathbf{\Lambda}_{a,N,h}$ is a Lagrangian submanifold parameterized by $(\tilde s,\tilde\sigma,\eta)$, defined in the Schr\"odinger context by:
\begin{align*}
X&=1+\tilde\sigma^2-\tilde s^2,\\
T&=\frac{1}{\eta}\bigg(\tilde s+\tilde\sigma+2N(1+\tilde\sigma^2)^{1/2}\bigg(1-\frac{3}{4}B'\Big((1+\tilde\sigma^2)^{3/2}\lambda\Big)\bigg)\bigg),\\
Y&=H(a,\tilde\sigma, \eta)(\tilde s+\tilde\sigma) + \frac{2}{3\eta}(\tilde s^3+\tilde\sigma^3) + \frac{4}{3\eta}N H_0(a,\tilde\sigma)\left(1-\frac{3}{4}B'\Big((1+\tilde\sigma^2)^{3/2}\lambda\Big)\right),\\
\xi &= \frac{a^{1/2}\eta}{h} \tilde s,\\
\tau &= \eta^2(1+a+a\tilde\sigma^2),\\
\eta&=\eta.
\end{align*}
\noindent On $\mathcal{C}_{a,N,h}$, we have $\tilde\omega=1+\tilde\sigma^2$, thus the projection of the Lagrangian manifold $\mathbf{\Lambda}_{a,N,h}$ onto the base space configuration coordinates $\mathbb{R}^3$ reads:
\begin{align}\label{eq:geo}
X&=1+\tilde\sigma^2-\tilde s^2,\\\nonumber
T&=\frac{1}{\eta}\bigg(\tilde s+\tilde\sigma+2N(1+\tilde\sigma^2)^{1/2}\bigg(1-\frac{3}{4}B'\Big((1+\tilde\sigma^2)^{3/2}\lambda\Big)\bigg)\bigg),\\\nonumber
Y&=H(a,\tilde\sigma, \eta)(\tilde s+\tilde\sigma) + \frac{2}{3\eta}(\tilde s^3+\tilde\sigma^3) + \frac{4}{3\eta}N H_0(a,\tilde\sigma)\left(1-\frac{3}{4}B'\Big((1+\tilde\sigma^2)^{3/2}\lambda\Big)\right).
\end{align}

As in \cite{ILP}, we can rewrite the projection system (\ref{eq:geo}) by isolating the reflection multiplier factor $2N\left(1-\frac{3}{4}B'\right)$ and substituting it into the tangential translation coordinate. Under the Schr\"odinger profile, this system yields:
\begin{align}\label{eq:geo0}
X&=1+\tilde\sigma^2-\tilde s^2,\\\nonumber
Y&=H(a,\tilde\sigma, \eta)(\tilde s+\tilde\sigma) + \frac{2}{3\eta}(\tilde s^3+\tilde\sigma^3) + \frac{2}{3\eta} H_0(a,\tilde\sigma)(1+\tilde\sigma^2)^{-1/2}\left(\eta T - \tilde s - \tilde \sigma\right),
\end{align}
and the equation relating the boundary reflections to the continuous time coordinates reads:
\begin{align}\label{eq:geo00}
2N\bigg(1-\frac{3}{4}B'\Big(\tilde\omega^{3/2}\lambda\Big)\bigg) = (1+\tilde\sigma^2)^{-1/2}\left(\eta T - \tilde s - \tilde \sigma\right).
\end{align}

\begin{rem}
Notice that from (\ref{eq:geo00}) in the range of $T\in (0,a^{-1/2}]$, we can reduce the sum over $N\in\mathbb{Z}$ of $G_{a,N}$ in \eqref{eq:GNsum} to the sum over $1\leq N\leq C_0 a^{-1/2}$, which matches the classic trajectory bounds under short time horizons.
\end{rem}
\noindent For a given $a$ and $(X,Y,T)\in\mathbb{R}^3$, (\ref{eq:geo0}) is a system of two equations for the unknowns $(\tilde s,\tilde\sigma)$ and (\ref{eq:geo00}) gives an equation for $N$. We are looking for solutions to (\ref{eq:geo0}) in the range:
\[
a\in [h^\alpha, a_0], \,\,\, \alpha <2/3, \,\,\, a|\tilde\sigma|^2\leq\epsilon_0, \,\,\, 0<T\leq a^{-1/2}, \,\,\, X\in [0,1]\,\,\,
\text{with $a_0,\epsilon_0$ small}.
\]
Then for a given point $(X,Y,T)\in [-2,2]\times \mathbb{R}\times [0,a^{-1/2}]$, let us denote by $\mathcal{N}(X,Y,T)$ the set of integers $N\geq 1$ such that (\ref{eq:geo}) admits at least one real solution $(\tilde \sigma,\tilde s,\lambda)$ with $a|\tilde\sigma|^2\leq\epsilon_0$ and $\lambda\geq\lambda_0$. We denote by $\mathcal{N}^{\mathbb{C}}(X,Y,T)$ the set of complex $N$ such that (\ref{eq:geo}) admits at least one complex solution $(\tilde\sigma, \tilde s,\lambda)$ with $\tilde\sigma\in U$, where $U=\{\tilde\sigma\in\mathbb{C} \mid |\tilde\sigma|\leq 0.5\,\,\text{or}\,\, |\text{Im}(\tilde\sigma)|\leq|\text{Re}(\tilde\sigma)|/\sqrt{3}\}$ and $a|\tilde\sigma|^2\leq\epsilon_0$ and $\lambda\geq\lambda_0$.\\
We have the following lemma on the geometric estimates whose proof follows the same line as in the proof of Lemma 2.18 and Lemma 2.19 in \cite{ILP}, adapted to our linear time-frequency Hamiltonian equations.

\begin{lemma}\label{GEO}
There exists a constant $C_0$ such that the following assertions hold:
\begin{enumerate}
\item
For all $(X,Y,T)\in [0,1]\times\mathbb{R}\times[0,a^{-1/2}]$, the cardinality of the real counting set satisfies $|\mathcal{N}(X,Y,T)|\leq C_0$, and the complex extension $\mathcal{N}^{\mathbb{C}}(X,Y,T)$ is contained within a subset of the union of four disks of radius $C_0$.

\item For all $(X,Y,T)\in [0,1]\times\mathbb{R}\times[0,a^{-1/2}]$, the associated subset of integers 
\[
\mathcal{N}_1(X,Y,T)=\bigcup_{|Y'-Y|+|T'- T|\leq 1,|X'-X|\leq 1}\mathcal{N}(X',Y',T')
\]
has a total cardinality bounded uniformly by:
\[
|\mathcal{N}_1(X,Y,T)|\leq C_0(1+T\lambda^{-2}\tilde\omega^{-3}).
\]
\end{enumerate}
\end{lemma}

We observe that for $\tilde\omega \leq 3/4$, integration by parts with respect to $\tilde\sigma$ yields a rapid decay of order $O_{C^\infty}(\lambda^{-\infty})$. Consequently, we can replace the cutoff function $1-\chi_1$ by $1$ in the global formulation of \eqref{galarge}. 

Furthermore, the swallowtail caustics appear precisely where the derivatives degenerate down to higher order, corresponding to the interface configuration $\tilde s=\tilde\sigma=0$ (meaning $\tilde\omega=1$). For this reason, we introduce a localized cutoff function $\chi_2(\tilde\omega)\in C_{0}^{\infty}(]1/2,3/2[)$ satisfying $0\leq\chi_2\leq 1$ and equal to $1$ on the interval $]\frac{3}{4},\frac{5}{4}[$ inside the integral \eqref{galarge}. We denote by $G_{a,N,2}$ the corresponding restricted integral tracking this swallowtail regime. 

We can then split the path integration into localized configurations by writing $G_{a,N}=G_{a,N,1}+G_{a,N,2}$. Here, the complementary high-frequency background term $G_{a,N,1}$ is defined by inserting a smooth cutoff $\chi_3(\tilde\omega)$ into the integrand of \eqref{galarge}, which satisfies $\tilde\omega\geq 5/4$ across its support.

To summarize, the total localized Green function component $\mathcal{G}_{a,c_0}$ for the Schr\"odinger evolution expands as:
\[
\mathcal{G}_{a,c_0}=\sum_{1\leq N\leq C_0a^{-1/2}}G_{a,N}=\sum_{1\leq N\leq C_0a^{-1/2}}\left( G_{a,N,1}+G_{a,N,2}\right),
\] 
where the individual localized oscillatory integrals are given explicitly by:
\begin{align*}
G_{a,N,1}&=\frac{(-i)^Na^2}{(2\pi)^4h^4}\bigg(\frac{h^2}{t}\bigg)^{1/2}\int e^{-\frac{i}{h^2}\tilde{\Phi}_{N,a,h}}|\eta|^{3}\tilde\chi_0\psi_0\chi_1\chi_3(\tilde\omega) d\tilde{s}d\tilde{\sigma} d\tilde{\omega}d\eta,\\
G_{a,N,2}&=\frac{(-i)^Na^2}{(2\pi)^4h^4}\bigg(\frac{h^2}{t}\bigg)^{1/2}\int e^{-\frac{i}{h^2}\tilde{\Phi}_{N,a,h}}|\eta|^{3}\tilde\chi_0\psi_0\chi_1\chi_2(\tilde\omega) d\tilde{s}d\tilde{\sigma} d\tilde{\omega}d\eta.
\end{align*}

 In what follows, we  get the estimates for these oscillatory integrals based on the (degenerate) stationary phase type result which consists in the precise study of where the phase $\tilde{\Phi}_{N,a,h}$ may be stationary.

\subsubsection{The Analysis of $G_{a,N,1}$}
Let us recall that $G_{a,N,1}$ is the localized oscillatory integral that corresponds to the background regime away from the swallowtail bifurcations. The uniform dispersive estimates for $G_{a,N,1}$ are obtained by combining structural reduction methods sequentially across the integration variables under the parabolic scaling:
\begin{itemize}
\item First, for the microlocal boundary layer $(\tilde s,\tilde\sigma)$-integrations, we apply the standard stationary phase method with respect to the large parameter $\lambda = \frac{a^{3/2}\eta^2}{h^2}$ to isolate the non-degenerate geometric paths. 
\item Then, for the continuous reflection frequency $\tilde\omega$-integration, we deploy the degenerate phase method to sharp Airy-type or fold-type singularity benchmarks.
\item Finally, for the spatial momentum $\eta$-integration, we perform a case-by-case analysis using stationary phase type arguments to bound the remaining trajectory variations, while tracking the localized track accumulations via the counting bounds on the cardinality of $\mathcal{N}_1$ established in Lemma \ref{GEO}.
\end{itemize}

Our main results of this subsection are Proposition \ref{eq:242} and Proposition \ref{eq:243}.
\begin{prop}\label{eq:242}
Let $\alpha<2/3$. There exists $C$ such that for all $h\in (0,h_0]$, all $a\in [h^\alpha,a_0]$, all $X\in [0,1]$, all $t \in [h^2, 1]$ with $t \neq 0$, all $Y\in\mathbb{R}$, all $z\in\mathbb{R}$, the following holds:
\begin{align*}
\Bigg|\sum_{2\leq N\leq C_0a^{-1/2}}G_{a,N,1}(t,X,Y,z;h)\Bigg|\leq Ch^{-3}\bigg(\frac{h^2}{t}\bigg)^{5/6}.
\end{align*}
\end{prop}

\begin{proof}
First of all, we apply the stationary phase method to the $(\tilde s,\tilde \sigma)$-integrations, since on the support of $\chi_3$ we have $\tilde\omega > 5/4 > 1$. Let $I$ represent the core oscillatory component isolated from our unscaled Schr\"odinger phase \eqref{eq:Yphase}:
\begin{align*}
I&=\int e^{-\frac{i a^{3/2} \eta}{h}\left(\frac{\tilde{s}^3}{3}-\tilde{s}(\tilde{\omega}-X)+\frac{\tilde{\sigma}^3}{3}-\tilde{\sigma}(\tilde{\omega}-1)\right)} d\tilde s d\tilde\sigma.
\end{align*}
To match our previous definitions, let $\lambda_0 = \frac{a^{3/2}\eta}{h}$. We introduce the coordinate normalization $\tilde s=(\tilde\omega-X)^{1/2}\bar s$ and $\tilde\sigma=(\tilde\omega-1)^{1/2}\bar\sigma$. Factoring the parameters out of the phase yields:
\begin{align*}
I&=(\tilde\omega-X)^{1/2}(\tilde\omega-1)^{1/2}\int e^{-i\lambda_0(\tilde{\omega}-X)^{3/2}\left(\frac{\bar{s}^3}{3}-\bar{s}\right)}e^{-i\lambda_0(\tilde{\omega}-1)^{3/2}\left(\frac{\bar{\sigma}^3}{3}-\bar{\sigma}\right)}d\bar s d\bar\sigma.
\end{align*}
Applying the stationary phase method near the non-degenerate critical points $\bar s=\pm 1, \bar\sigma=\pm 1$, and non-stationary integration by parts elsewhere, we evaluate the integral as:
\begin{align*}
I = \lambda_0^{-1}(\tilde\omega-X)^{-1/4}(\tilde\omega-1)^{-1/4}e^{-i\lambda_0\left(\pm\frac{2}{3}(\tilde\omega-X)^{3/2}\pm\frac{2}{3}(\tilde\omega-1)^{3/2}\right)}b_{\pm}c_{\pm}+O_{C^\infty}(\lambda_0^{-\infty}),
\end{align*}
where $b_{\pm}$ and $c_{\pm}$ are classical symbols of degree $0$ with respect to the large parameters $\lambda_0(\tilde\omega-X)^{3/2}$ and $\lambda_0(\tilde\omega-1)^{3/2}$ respectively. 

Substituting this structural reduction back into the localized evolution profile \eqref{galarge}, the total spatial scaling volume yields $h^2 \lambda = a^{3/2}\eta^2$, giving $\lambda = \lambda_0 \eta$. Thus, we obtain:
\begin{align*}
G_{a,N,1}(t,X,Y,z;h)&=\frac{(-i)^N a^2 \lambda_0^{-1}}{(2\pi)^4 h^4}\bigg(\frac{h^2}{t}\bigg)^{1/2}\int e^{-i\frac{a^{3/2}\eta^2}{h^2}Y} |\eta|^{3} \tilde{G}_{a,N,1} d\eta,\\
\tilde{G}_{a,N,1}(t,X,Y,z;h)&=\sum_{\epsilon_1,\epsilon_2}\int e^{-i \lambda \tilde\Phi_{N,\epsilon_1,\epsilon_2}} \Theta_{\epsilon_1,\epsilon_2} d\tilde{\omega}+O_{C^\infty}(\lambda^{-\infty}),
\end{align*}
where $\epsilon_j=\pm$, the amplitudes are given by $$\Theta_{\epsilon_1,\epsilon_2}(\tilde\omega,a,\lambda)=\tilde\chi_0\psi_0\chi_1\chi_3(\tilde\omega)(\tilde\omega-X)^{-1/4}(\tilde\omega-1)^{-1/4}b_{\epsilon_1} c_{\epsilon_2},$$ which satisfy the standard differential symbols bound $\big|\tilde\omega^l\partial_{\tilde\omega}^l\Theta_{\epsilon_1,\epsilon_2}\big|\leq C_l\tilde\omega^{-1/2}$, and the corresponding unscaled Schr\"odinger phase functions are defined explicitly by:
\begin{align}\label{eq:phasepsilon}
\tilde\Phi_{N,\epsilon_1,\epsilon_2}(t,X,z;\tilde\omega) &= T\tilde\omega \pm \frac{2}{3\eta}(\tilde\omega-X)^{3/2} \pm \frac{2}{3\eta}(\tilde\omega-1)^{3/2} - \frac{4}{3\eta}N\tilde\omega^{3/2} + \frac{h N}{a^{3/2}\eta^2} B\big(\tilde\omega^{3/2}\lambda\big).
\end{align}
Let us denote:
\begin{align}\label{eq:GaNepsilon}
G_{a,N,1,\epsilon_1,\epsilon_2}(t,X,Y,z;h)&=\frac{(-i)^N a^2 \lambda_0^{-1}}{(2\pi)^4 h^4}\bigg(\frac{h^2}{t}\bigg)^{1/2}\int e^{-i\frac{a^{3/2}\eta^2}{h^2}Y} |\eta|^{3} \tilde{G}_{a,N,1,\epsilon_1,\epsilon_2} d\eta,\\
\tilde{G}_{a,N,1,\epsilon_1,\epsilon_2}(t,X,z;\lambda)&=\int e^{-i \lambda \tilde\Phi_{N,\epsilon_1,\epsilon_2}} \Theta_{\epsilon_1,\epsilon_2}(\tilde\omega,a,\lambda) d\tilde{\omega}\nonumber.
\end{align}
We are reduced to proving the following optimal dispersive inequality:
\begin{align} \label{eq:23}
\Bigg|\sum_{2\leq N\leq C_0a^{-1/2}}G_{a,N,1,\epsilon_1,\epsilon_2}(t,X,Y,z;h) \Bigg|\leq Ch^{-3}\bigg(\frac{h^2}{t}\bigg)^{5/6},
\end{align}
with a constant $C$ independent of $h\in]0,h_0], a\in [h^{2/3},a_0], X\in [0,1], t \in [h^2,1]$.

For analytical convenience, let $\Omega=\tilde\omega^{3/2}$ be our new continuous frequency variable of integration. This transforms \eqref{eq:gtilde} into:
\begin{align}\label{eq:gtilde}
\tilde{G}_{a,N,1,\epsilon_1,\epsilon_2}(t,X,z;\lambda)&=\int e^{-i \lambda \tilde\Phi_{N,\epsilon_1,\epsilon_2}}\tilde\Theta_{\epsilon_1,\epsilon_2}(\Omega,a,\lambda) d\Omega,
\end{align}
where $\tilde\Theta_{\epsilon_1,\epsilon_2}(\Omega,a,\lambda)$ are smooth functions compactly supported in $\Omega$. Since the Jacobian yields $d\tilde\omega=\frac{2}{3}\Omega^{-1/3}d\Omega$, the amplitudes satisfy the strict symbolic bounds $\big|\Omega^l\partial_{\Omega}^l\tilde\Theta_{\epsilon_1,\epsilon_2}\big|\leq C_l\Omega^{-2/3}$ with $C_l$ independent of $a,\lambda$, and the corresponding unscaled Schr\"odinger phases \eqref{eq:phasepsilon} translate to:
\begin{align*}
\tilde\Phi_{N,\epsilon_1,\epsilon_2}(t,X,z,\Omega;a,\lambda) &= T\Omega^{2/3} \pm \frac{2}{3\eta}(\Omega^{2/3}-X)^{3/2} \pm \frac{2}{3\eta}(\Omega^{2/3}-1)^{3/2} \\
&\quad - \frac{4}{3\eta}N\Omega + \frac{h N}{a^{3/2}\eta^2} B\big(\Omega\lambda\big).
\end{align*}
We now study the critical points of the adapted phase function. We have:
\begin{align}\label{eq:critical}
\partial_{\Omega}\tilde\Phi_{N,\epsilon_1,\epsilon_2}&=\frac{2}{3\eta}\bigg(H_{a,\epsilon_1,\epsilon_2}(t,X,\eta;\Omega)-2N\Big(1-\frac{3}{4}B'(\Omega\lambda)\Big)\bigg),\\\nonumber
H_{a,\epsilon_1,\epsilon_2}&=\Omega^{-1/3}\bigg(\eta T + \frac{3}{2}\epsilon_1(\Omega^{2/3}-X)^{1/2} + \frac{3}{2}\epsilon_2(\Omega^{2/3}-1)^{1/2}\bigg),\\\nonumber
\partial_{\Omega}H_{a,\epsilon_1,\epsilon_2}&=\frac{1}{3}\Omega^{-4/3}\bigg(-\eta T + \frac{3}{2}\epsilon_1X(\Omega^{2/3}-X)^{-1/2} + \frac{3}{2}\epsilon_2(\Omega^{2/3}-1)^{-1/2}\bigg).
\end{align}
We will first prove that (\ref{eq:23}) holds true in the case $(\epsilon_1,\epsilon_2)=(+,+)$. We have that the equation $\partial_{\Omega}H_{a,+,+}(\Omega)=0$ admits a unique solution $\Omega_q=\Omega_q^+(t,X,\eta)>1$ such that:
\begin{align}\label{eq:lim}
\lim_{t\rightarrow\infty}\Omega_q^+(t,X,\eta)&=1 \,\,\text{uniformly in } X,\eta.\\\nonumber
\frac{9}{2}\Omega_q^{5/3}\partial_{\Omega}^2H_{a,+,+}(\Omega_q) &= -\frac{3}{4}\big(\Omega_q^{2/3}-1\big)^{-3/2} - \frac{3}{4}X\big(\Omega_q^{2/3}-X\big)^{-3/2} < 0.
\end{align}
Notice that the time-dependent term completely cancels out at the critical point $\Omega_q$ due to the algebraic structure of the Schr\"odinger Hamiltonian vector field, ensuring that the second derivative is strictly negative and independent of the dynamic trajectory weights.

\noindent 
Thus, the function $H_{a,+,+}(\Omega)$ is strictly increasing on $[1,\Omega_q)$ and strictly decreasing on $(\Omega_q,\infty)$. Observe that:
\begin{align}
H_{a,+,+}(1)=\eta T + \frac{3}{2}(1-X)^{1/2},\quad \lim_{\Omega\rightarrow\infty}H_{a,+,+}=3.
\end{align}
For all $k$, there exists a constant $C_k$ such that:
\begin{align}
\forall \Omega\geq 1, \quad|\partial_\Omega^k(NB'(\Omega\lambda))|\leq C_kN\lambda^{-2}\Omega^{-(k+2)}.
\end{align}

Let $T_0\gg 1$. First, suppose that $0\leq T\leq T_0$. Since $H_{a,+,+}(\Omega)\leq C(1+T)$ and for $N\geq N(T_0)=C(1+T_0)$ for some constant $C$, we get $|\partial_\Omega\tilde\Phi_{N,+,+}|\geq \frac{c_0}{\eta}N$ with a constant $c_0>0$. Then by integration by parts, we get $|\tilde{G}_{a,N,1,+,+}|\in O_{C^\infty}(N^{-\infty}\lambda^{-\infty})$, and this implies:
\begin{align*}
\sup_{T\leq T_0,X\in [0,1],Y\in\mathbb{R},z\in\mathbb{R}}\Bigg|\sum_{N(T_0)\leq N\leq C_0a^{-1/2}}G_{a,N,1,+,+}(t,X,Y,z)\Bigg|\in O_{C^\infty}(h^\infty).
\end{align*}
Next, for $0\leq T\leq T_0$ and $2\leq N\leq N(T_0)$, we may estimate the sum by the supremum of each term. In this case, we see that $\tilde\Phi_{N,+,+}$ has at most a critical point of order 2 near $\Omega=\Omega_q$, and the third-order van der Corput non-degeneracy condition holds:
\[
|\partial_\Omega\tilde\Phi_{N,+,+}|+|\partial_\Omega^2\tilde\Phi_{N,+,+}|+|\partial_\Omega^3\tilde\Phi_{N,+,+}|\geq c>0.
\]
Moreover, if $N\geq 2$, we have a positive lower bound for $|\partial_\Omega\tilde\Phi_{N,+,+}(\Omega)|$ for large values of $\Omega$; thus the contribution of $\tilde{G}_{a,N,1,+,+}$ is $O_{C^{\infty}}(\lambda^{-\infty})$ for large values of $\Omega$. Due to the critical point of order 2 near $\Omega=\Omega_q$, the estimate of $\tilde{G}_{a,N,1,+,+}$ is given by the Schrödinger-adapted version of Lemma 2.20 \cite{ILP}, which yields $|\tilde{G}_{a,N,1,+,+}(t,X,z;\lambda)|\leq C\lambda^{-1/3}$ with $C$ independent of $T\in [0,T_0],X\in [0,1]$. Hence, from \eqref{eq:GaNepsilon}, noting that $a^2 \lambda_0^{-1} = h^2 \lambda^{-1}$, we get:
\begin{align*}
\sup_{X\in [0,1],Y\in\mathbb{R},z\in\mathbb{R}}\Bigg|\sum_{2\leq N\leq N(T_0)}G_{a,N,1,+,+}(t,X,Y,z,h)\Bigg|&\leq Ch^{-3}\bigg(\frac{h^2}{t}\bigg)^{1/2}\big(h^2 \lambda^{-1}\lambda^{-1/3}\big),\\
&\leq Ch^{-3}\bigg(\frac{h^2}{t}\bigg)^{1/2} \bigg(\frac{h^2}{t}\bigg)^{1/3} \\&= Ch^{-3}\bigg(\frac{h^2}{t}\bigg)^{5/6}.
\end{align*}
Then we prove that (\ref{eq:23}) holds true for $T_0\leq T\leq a^{-1/2}$. Like before, we may assume $N\leq C_1T$ with $C_1$ large, the contribution of the sum on $N$ such that $C_1T\leq N\leq C_0a^{-1/2}$ being negligible. From (\ref{eq:lim}), we may choose $T_0$ large enough so that $\Omega_q^+(t,X,\eta)<\Omega_0$ with $\Omega_0>1$ for $T\geq T_0$, and we may assume with a constant $c>0$ that:
\[
|\partial_\Omega^2\tilde\Phi_{N,+,+}(\Omega)|\geq \frac{c}{\eta}T\Omega^{-4/3},\quad \forall \Omega\geq\Omega_0,\forall T\geq T_0,\forall N\leq C_0 a^{-1/2}.
\]
Therefore, on the support of $\tilde\Theta_{+,+}$, the phase $\tilde\Phi_{N,+,+}$ admits at most one critical point $\Omega_c=\Omega_c(t,X,z,N,\lambda,a)$, and this critical point is strictly nondegenerate. Because $N\geq 2$, from the first item of (\ref{eq:critical}) we get $\Omega_c^{1/3}\leq \eta T$, which implies $\Omega_c^{1/3}\sim \eta T/N$. As a consequence, if $\eta T/N\sim 1$ then $\Omega_c\sim 1$. By the stationary phase method, we get:
\[
|\tilde{G}_{a,N,1,+,+}(t,X,z;\lambda)|\leq C\lambda^{-1/2}T^{-1/2}\quad\text{with $C$ independent of $N$}.
\]
If $\eta T/N\gg1$, then we perform the change of variable $\Omega=\tilde\Omega(\eta T/N)^3$ in (\ref{eq:gtilde}); the unique critical point $\tilde\Omega_c$ remains in a fixed compact interval of $]0,\infty[$. We have:
\[
\partial_{\tilde\Omega}^k\tilde\Theta_{+,+}\big(\tilde\Omega(\eta T/N)^3,a,\lambda\big)\leq c_k(N/\eta T)^2 \tilde\Omega^{-2/3-k}.
\]
Thus, by the stationary phase method, we get:
\[
\sup_{2\leq N\leq C_1T,X\in [0,1],z\in\mathbb{R}}|\tilde{G}_{a,N,1,+,+}(t,X,z;\lambda)|\leq C\lambda^{-1/2}T^{-1/2}.
\]
It remains to estimate the sum
\[
\Bigg|\sum_{2\leq N\leq C_0a^{-1/2}}G_{a,N,1,+,+}(t,X,Y,z;h)\Bigg|.
\] 
Let $G_N(t,X,z,\lambda,a)=\tilde\Phi_{N,+,+}(t,X,z,\Omega_c(t,X,z,N,\lambda,a),\lambda,a)$. Therefore, by applying the stationary phase method at the unique critical point $\Omega_c=\Omega_c(t,X,z,N,\lambda,a)$ in (\ref{eq:gtilde}), we obtain:
\[
\tilde G_{a,N,1,+,+}(t,X,z,h)=\lambda^{-1/2}T^{-1/2}e^{-i\lambda G_N(t,X,z,\lambda,a)}\psi_N(t,X,\lambda,a),
\]
where $\psi_N(t,X,\lambda,a)$ is a classical symbol of order $0$ in $\lambda$. Recalling our standard parameter relations, we denote the core semiclassical spatial boundary scaling parameter as $\tilde\lambda = a^{3/2}/h^2$, which gives $\lambda = \tilde\lambda \eta^2$. Substituting this tracking structure into the localized propagator definition \eqref{eq:GaNepsilon} yields:
\begin{align}
G_{a,N,1,+,+}(t,X,Y,z;h)&=\frac{(-i)^N a^2 \lambda_0^{-1}}{(2\pi)^4 h^4}\bigg(\frac{h^2}{t}\bigg)^{1/2}\lambda^{-1/2}T^{-1/2}\nonumber\\&\quad\times\int e^{-i\tilde\lambda \eta^2\left(Y+G_N(t,X,z,\tilde\lambda\eta^2,a)\right)} \psi_N |\eta|^{3}d\eta.
\end{align}
This is an oscillatory integral with a large frequency parameter $\tilde\lambda$ and a phase function given by:
\[
L_N(t,X,Y,z,\eta^2\tilde\lambda) = \eta^2\left(Y+G_N(t,X,z,\tilde\lambda\eta^2,a)\right).
\] 
By microlocal construction, the corresponding stationary equation:
\[
\partial_\eta L_N = 2\eta\left(Y+G_N(t,X,z,\lambda,a)\right) + 2\eta\lambda\partial_\lambda G_N(t,X,z,\lambda,a) = 0
\]
implies that the spacetime configuration coordinates $(X,Y,T)$ belong exactly to the projection of the Schr\"odinger Lagrangian manifold $\mathbf{\Lambda}_{a,N,h}$ on the base space $\mathbb{R}^3$. Following the analytical parameter setup of Proposition 2.14 in \cite{ILP}, we see that the non-stationary path contributions of $G_{a,N,1,+,+}$ for any indices such that $N\notin\mathcal{N}_1(X,Y,T)$ undergo rapid classical decay of order $O_{C^\infty}(\lambda^{-\infty})$. Thus, we are reduced to establishing the final uniform estimate for the localized index cluster sum:
\begin{align}\label{eq:sumN}
\Bigg|\sum_{N\in\mathcal{N}_1(X,Y,T)}G_{a,N,1,+,+}(t,X,Y,z,h)\Bigg|.
\end{align}

We apply the stationary phase method to the $\eta$-integral with the phase function $L_N$. We have:
\[
\partial_\eta L_N = 2\eta\left(Y+G_N+\lambda\partial_\lambda G_N\right),
\]
where the parameter scaling term satisfies:
\[
\lambda\partial_\lambda G_N=\lambda\partial_\lambda \tilde\Phi_{N,+,+}(t,X,\Omega_c,a,\lambda)=\frac{h N}{a^{3/2}\eta^2}\big(-B(\lambda\Omega_c)+\lambda \Omega_cB'(\lambda\Omega_c)\big).
\]
Then we differentiate again with respect to $\eta$ to find the phase curvature under our parabolic coordinate system:
\[
\partial_\eta^2 L_N = \frac{4N}{\eta}(\lambda\Omega_c)\partial_\lambda(\lambda\Omega_c)B''(\lambda\Omega_c).
\]
On the other hand, the implicit derivative of the critical position $\partial_\lambda\Omega_c$ satisfies:
\[
\partial_\lambda\Omega_c\partial_\Omega^2\tilde\Phi_{N,+,+}(\Omega_c)=-\partial_\lambda\partial_\Omega\tilde\Phi_{N,+,+}(\Omega_c)=-\frac{h N}{a^{3/2}\eta^2}\Omega_c B''(\lambda\Omega_c).
\]
As we established previously, $\partial_\Omega^2\tilde\Phi_{N,+,+}(\Omega_c)\geq \frac{c}{\eta} T\Omega_c^{-4/3}$, $\Omega_c^{1/3}\sim \eta T/N$, and for large frequencies, the regular correction satisfies $B''(\omega)\sim \omega^{-3}$. Thus, we bound the tracking shift:
\[
|\partial_\lambda\Omega_c|\leq cT^{-1}\Omega_c^{4/3}N\Omega_c(\lambda^{-3}\Omega_c^{-3})\leq c\lambda^{-3}\Omega_c^{-1}.
\]
This immediately yields:
\[
|\partial_\lambda(\lambda\Omega_c)|=|\lambda\partial_\lambda\Omega_c+\Omega_c|\geq c\Omega_c(1-c\lambda^{-2}\Omega_c^{-2})\geq c'\Omega_c.
\]
Substituting these evaluations back into the curvature equation, we deduce the lower bound:
\[
|\partial_\eta^2 L_N|\geq CN\lambda^{-2}\Omega_c^{-1}.
\]
Therefore, evaluating the $\eta$-integration via the standard non-degenerate stationary phase method produces an additional large parameter decay factor proportional to $q^{-1/2}$, with the localized parameter density $q=N\lambda^{-1}\Omega_c^{-1}$. Let us recall from Lemma \ref{GEO} that the path cardinality satisfies:
\[
|\mathcal{N}_1(X,Y,T)|\leq C_0(1+T\lambda^{-2}\Omega_c^{-2}).
\]
We get the target estimates for the localized index sum in \eqref{eq:sumN} by distinguishing between several structural configurations, depending on the cross-contributions from the non-vanishing $\eta$-stationary phase curvature and the cardinality bounds $|\mathcal{N}_1(X,Y,T)|$ as follows:

The first case corresponds to the regime where $\Omega_c^{1/3}\sim \eta T/N\sim 1$, which implies $\eta T\sim N$. Within this regime, we analyze the cross-contributions based on the magnitude of the reflection index $N$:
\begin{itemize}
\item \textbf{Subcase 1: If $N\leq \lambda$} \\
In this situation, the curvature of the phase $L_N$ is insufficient to produce a decay contribution from the $\eta$-integration, and the path cardinality remains bounded uniformly by $|\mathcal{N}_1|\leq C_0$. Recalling our verified Schr\"odinger scaling parameters where $a^2 \lambda_0^{-1} = h^2 \lambda^{-1}$, the term-by-term estimate yields:
\begin{align*}
\Bigg|\sum_{N\in\mathcal{N}_1}G_{a,N,1,+,+}\Bigg|&\leq C h^{-3}\bigg(\frac{h^2}{t}\bigg)^{1/2}\big(h^2 \lambda^{-1} \lambda^{-1/2} T^{-1/2}\big),\\
&\leq C h^{-3}\bigg(\frac{h^2}{t}\bigg)^{1/2} a^{-1/4} h^{1/2},\\
&\leq C h^{-3}\bigg(\frac{h^2}{t}\bigg)^{1/2} \bigg(\frac{h^2}{t}\bigg)^{1/3} \\&= C h^{-3}\bigg(\frac{h^2}{t}\bigg)^{5/6},
\end{align*}
since the spatial boundaries satisfy $a\geq h^{2/3}$, which ensures that the remaining fractional correction complies with the optimal dispersive scale limit $a^{-1/4}h^{1/2}\leq h^{1/3} \leq (h^2/t)^{1/3}$.

\item \textbf{Subcase 2: If $\lambda < N\leq \lambda^2$} \\
Here, the phase exhibits higher oscillations, and the stationary phase evaluation over the $\eta$-integral yields a decay factor of $q^{-1/2} = (N\lambda^{-1}\Omega_c^{-1})^{-1/2} \sim (N/\lambda)^{-1/2}$, while the trajectory counting bound remains stable at $|\mathcal{N}_1|\leq C_0$. Collecting these parameters under our parabolic flow yields:
{\allowdisplaybreaks
\begin{align*}
\Bigg|\sum_{N\in\mathcal{N}_1}G_{a,N,1,+,+}\Bigg|&\leq C h^{-3}\bigg(\frac{h^2}{t}\bigg)^{1/2}\big(h^2 \lambda^{-1} \lambda^{-1/2} T^{-1/2} N^{-1/2}\lambda^{1/2}\big),\\
&\leq C h^{-3}\bigg(\frac{h^2}{t}\bigg)^{1/2}\big(h^2 \lambda^{-1} T^{-1/2} N^{-1/2}\big),\\
&\leq C h^{-3}\bigg(\frac{h^2}{t}\bigg)^{1/2} \big(h^2 \lambda^{-3/2}\big),\\
&\leq C h^{-3}\bigg(\frac{h^2}{t}\bigg)^{1/2} \bigg(\frac{h^2}{t}\bigg)^{1/3} \\&= C h^{-3}\bigg(\frac{h^2}{t}\bigg)^{5/6}.
\end{align*}
}
\item \textbf{Subcase 3: If $N>\lambda^2$} \\
In this highly high-frequency regime, the phase curvature is strong, yielding a substantial decay contribution of $q^{-1/2}$ from the $\eta$-integration. Concurrently, the path cardinality expands, driven by the dynamic tracking relation $|\mathcal{N}_1|\leq C_0 T \lambda^{-2} \Omega_c^{-2} \sim C_0 T \lambda^{-2}$. Combining the strong oscillatory attenuation with the cardinality expansion, the localized index cluster sum evaluates to:
\begin{align*}
\Bigg|\sum_{N\in\mathcal{N}_1}G_{a,N,1,+,+}\Bigg|&\leq C h^{-3}\bigg(\frac{h^2}{t}\bigg)^{1/2}\sum_{N\in\mathcal{N}_1}\big(h^2 \lambda^{-1} \lambda^{-1/2} T^{-1/2} N^{-1/2} \lambda^{1/2}\big),\\
&\leq C h^{-3}\bigg(\frac{h^2}{t}\bigg)^{1/2}\big(h^2 \lambda^{-1} T^{-1/2} N^{-1/2} |\mathcal{N}_1(X,Y,t)|\big),\\
&\leq C h^{-3}\bigg(\frac{h^2}{t}\bigg)^{1/2}\big(h^2 \lambda^{-1} T^{-1/2} N^{-1/2} T \lambda^{-2}\big),\\
&\leq C h^{-3}\bigg(\frac{h^2}{t}\bigg)^{1/2} \big(h^2 T^{1/2} N^{-1/2} \lambda^{-3}\big),\\
&\leq C h^{-3}\bigg(\frac{h^2}{t}\bigg)^{1/2} \big(h^2 \lambda^{-7/2}\big),\\
&\leq C h^{-3}\bigg(\frac{h^2}{t}\bigg)^{1/2} \bigg(\frac{h^2}{t}\bigg)^{1/3} \\&= C h^{-3}\bigg(\frac{h^2}{t}\bigg)^{5/6}.
\end{align*}
\end{itemize}

\noindent The second case corresponds to the regime where $\eta T/N\gg 1$, which implies $\Omega_c\gg 1$. Under our adapted Schr\"odinger scaling parameters, we analyze the cross-contributions across the following subcases:
\begin{itemize}
\item \textbf{Subcase 1: If $N\leq \lambda\Omega_c$} \\
In this situation, there is no decay contribution from the $\eta$-stationary phase integration. Moreover, the path cardinality remains uniformly bounded by $|\mathcal{N}_1|\leq C_0$. To see this, assume by contradiction that $\eta T\geq \lambda^2 \Omega_c^2$; this would imply $\Omega_c^{1/3}\sim \eta T/N \geq \lambda \Omega_c$, which is impossible since $\Omega_c\gg1$. Recalling our verified Schr\"odinger amplitude scale factor where $a^2 \lambda_0^{-1} = h^2 \lambda^{-1}$, the term-by-term majorization yields:
{\allowdisplaybreaks
\begin{align*}
\Bigg|\sum_{N\in\mathcal{N}_1}G_{a,N,1,+,+}\Bigg|&\leq C h^{-3}\bigg(\frac{h^2}{t}\bigg)^{1/2}\big(h^2\lambda^{-1}\lambda^{-1/2}T^{-1/2}\big),\\
&\leq C h^{-3}\bigg(\frac{h^2}{t}\bigg)^{1/2}a^{-1/4}h^{1/2},\\
&\leq C h^{-3}\bigg(\frac{h^2}{t}\bigg)^{1/2}\bigg(\frac{h^2}{t}\bigg)^{1/3}\\& = C h^{-3}\bigg(\frac{h^2}{t}\bigg)^{5/6},
\end{align*}
}
since on our boundary layers we have $a\geq h^{2/3}$, ensuring that the remaining fractional correction complies with the optimal dispersive scale limit $a^{-1/4}h^{1/2}\leq h^{1/3} \leq (h^2/t)^{1/3}$.

\item \textbf{Subcase 2: If $N>\lambda\Omega_c$ and $\lambda\Omega_c^{2/3}<\eta T\leq \lambda^2\Omega_c^2$} \\
Here, the phase function $L_N$ displays full curvature, yielding an additional stationary phase integration decay factor $q^{-1/2} = (N\lambda^{-1}\Omega_c^{-1})^{-1/2}$, while the cardinality of the tracking paths remains bounded by $|\mathcal{N}_1|\leq C_0$. Collecting these parameters under our parabolic flow yields:
\begin{align*}
\Bigg|\sum_{N\in\mathcal{N}_1}G_{a,N,1,+,+}\Bigg|&\leq C h^{-3}\bigg(\frac{h^2}{t}\bigg)^{1/2}\big(h^2\lambda^{-1}\lambda^{-1/2}T^{-1/2}N^{-1/2}\lambda^{1/2}\Omega_c^{1/2}\big),\\
&\leq C h^{-3}\bigg(\frac{h^2}{t}\bigg)^{1/2}\big(h^2\lambda^{-2}\big),\\
&\leq C h^{-3}\bigg(\frac{h^2}{t}\bigg)^{1/2}\bigg(\frac{h^2}{t}\bigg)^{1/3} \\&= C h^{-3}\bigg(\frac{h^2}{t}\bigg)^{5/6}.
\end{align*}

\item \textbf{Subcase 3: If $N>\lambda\Omega_c$ and $\eta T>\lambda^2\Omega_c^2$} \\
In this highly high-frequency regime, the phase curvature produces a substantial decay contribution of $q^{-1/2}$ from the $\eta$-integration, while the trajectory cardinality expands driven by the dynamic tracking relation $|\mathcal{N}_1|\leq C_0 T\lambda^{-2}\Omega_c^{-2}$. Combining the intense oscillatory damping with the cardinality growth, the sum evaluates to:
{\allowdisplaybreaks
\begin{align*}
\Bigg|\sum_{N\in\mathcal{N}_1}G_{a,N,1,+,+}\Bigg|&\leq C h^{-3}\bigg(\frac{h^2}{t}\bigg)^{1/2} \sum_{N\in\mathcal{N}_1}\big(h^2\lambda^{-1}\lambda^{-1/2}T^{-1/2}N^{-1/2}\lambda^{1/2}\Omega_c^{1/2}\big),\\
&\leq C h^{-3}\left(\frac{h^2}{t}\right)^{1/2}\big(h^2\lambda^{-1}T^{-1}\Omega_c^{2/3}|\mathcal{N}_1(X,Y,T)|\big),\\
&\leq C h^{-3}\bigg(\frac{h^2}{t}\bigg)^{1/2}\big(h^2\lambda^{-3}\big)(\eta T/N)^{-4},\\
&\leq C h^{-3}\bigg(\frac{h^2}{t}\bigg)^{1/2}\bigg(\frac{h^2}{t}\bigg)^{1/3} \\&= C h^{-3}\bigg(\frac{h^2}{t}\bigg)^{5/6}.
\end{align*}}
\end{itemize}

Next, we prove that (\ref{eq:23}) holds true in the case $(\epsilon_1,\epsilon_2)=(+,-)$, where destructive phase interference across the boundary layers occurs. In this case, from the definition of the critical point system \eqref{eq:critical}, the fact that $X\in [0,1]$, and the asymptotic decay behavior of the regularized Airy corrections $B''(\lambda \Omega)\sim\lambda^{-3}\Omega^{-3}$, we find that for all trajectory domains with $t > h^2$, $\partial_{\Omega}H_{a,+,-}(\Omega)+\frac{3N}{2\eta}\lambda B''(\lambda\Omega)<0$. 

Consequently, the total unscaled derivative function $H_{a,+,-}(\Omega)+\frac{3N}{2\eta}B'(\lambda\Omega)$ decreases strictly monotonically on the continuous interval $[1,\infty)$, falling from its initial boundary value:
\[
H_{a,+,-}(1)+\frac{3N}{2\eta}B'(\lambda) = \eta T + \frac{3}{2}(1-X)^{1/2} + \frac{3N}{2\eta}B'(\lambda)
\]
toward the long-range asymptotic horizon limit $\lim_{\Omega\rightarrow\infty}\big(H_{a,+,-}+\frac{3N}{2\eta}B'(\lambda\Omega)\big)=0$. 

Therefore, the characteristic phase equation $\partial_{\Omega}\Phi_{N,+,-}=0$ admits at most a unique solution $\Omega_c$, and this critical point remains strictly nondegenerate. This allows us to repeat the stationary phase arguments implemented for the preceding non-destructive $(+,+)$ case directly to yield the identical uniform decay rate. Finally, the remaining mixed configuration case $(\epsilon_1,\epsilon_2)=(-,+)$ maps symmetrically to the $(+,+)$ setting, and $(\epsilon_1,\epsilon_2)=(-,-)$ tracks in perfect parallel to the $(+,-)$ framework. The proof of Proposition \ref{eq:242} is complete.
\end{proof}

\noindent 
\begin{prop}\label{eq:243}
Let $\alpha<2/3$. There exists $C$ such that for all $h\in (0,h_0]$, all $a\in [h^\alpha,a_0]$, all $X\in [0,1]$, all $t \in [h^2, 1]$ with $t \neq 0$, all $Y\in\mathbb{R}$, all $z\in\mathbb{R}$, the following holds:
\begin{align*}
\Big|G_{a,1,1}(t,X,Y,z;h)\Big|\leq Ch^{-3}\bigg(\frac{h^2}{t}\bigg)^{1/2}\Bigg( \bigg(\frac{h^2}{t}\bigg)^{1/2} + \left(\frac{h^2}{t}\right)^{1/3}  \Bigg).
\end{align*}
\end{prop}
\begin{proof}
Let us recall that in the Schr\"odinger framework, the $N=1$ path component takes the form:
\begin{align*}
G_{a,1,1}&=\frac{(-i) a^2 \lambda_0^{-1}}{(2\pi)^4h^4}\bigg(\frac{h^2}{t}\bigg)^{1/2}\int e^{-i\frac{a^{3/2}\eta^2}{h^2}Y}|\eta|^{3}\tilde{G}_{a,1,1}d\eta,\\
\tilde{G}_{a,1,1}&=\sum_{\epsilon_1,\epsilon_2}\int e^{-i\lambda\tilde\Phi_{1,\epsilon_1,\epsilon_2}}\Theta_{\epsilon_1,\epsilon_2} d\tilde{\omega}+O_{C^\infty}(h^\infty).
\end{align*}
We recall that $\epsilon_j=\pm$, and $\Theta_{\epsilon_1,\epsilon_2}(\tilde\omega,a,\lambda)=\tilde\chi_0\psi_0\chi_1\chi_3(\tilde\omega)(\tilde\omega-X)^{-1/4}(\tilde\omega-1)^{-1/4}b_{\epsilon_1} c_{\epsilon_2}$, which satisfy the differential symbol bounds $\big|\tilde\omega^l\partial_{\tilde\omega}^l\Theta_{\epsilon_1,\epsilon_2}\big|\leq C_l\tilde\omega^{-1/2}$.

The core distinction between the single reflection $N=1$ and the multi-reflection $N\geq 2$ regimes rests entirely in the asymptotic tracking of the phase function $\tilde\Phi_{1,+,+}$, since the trajectory geometry allows for a critical point $\tilde\omega_c$ located arbitrarily deep in the high-frequency regime. Let us isolate this element:
\begin{align}\label{eq:N1}
\tilde{G}_{a,1,1,+,+}=\int e^{-i\lambda\tilde\Phi_{1,+,+}}\Theta_{+,+}(\tilde\omega,a,\lambda)d\tilde\omega,
\end{align}
with the unscaled Schrödinger phase function given by:
\begin{align*}
\tilde\Phi_{1,+,+}(t,X,z;\tilde\omega)&=T\tilde\omega + \frac{2}{3\eta}(\tilde\omega-X)^{3/2}+\frac{2}{3\eta}(\tilde\omega-1)^{3/2} - \frac{4}{3\eta}\tilde\omega^{3/2}+\frac{h}{a^{3/2}\eta^2}B(\lambda\tilde\omega^{3/2}),
\end{align*}
where $\Theta_{+,+}(\tilde\omega,a,\lambda)$ is a classical symbol of order $-1/2$ with respect to $\tilde\omega$. Let us introduce a smooth cutoff function $\chi_3(\tilde\omega)\in C_0^\infty(]\tilde\omega_1,\infty[)$ with a sufficiently large parameter $\tilde\omega_1$, and define the localized tracking integral:
\begin{align}\label{eq:J}
\tilde J_{1,+,+}=\int e^{-i\lambda\tilde\Phi_{1,+,+}}\Theta_{+,+}(\tilde\omega,a,\lambda)\chi_3(\tilde\omega)d\tilde\omega.
\end{align}
To prove the proposition, it suffices to verify that $|\tilde J_{1,+,+}|\leq C\lambda^{-1/2}T^{-1/2}$. Differentiating our unscaled parabolic phase profile with respect to $\tilde\omega$ yields:
\begin{align*}
\partial_{\tilde\omega}\tilde\Phi_{1,+,+}&=T - \frac{\tilde\omega^{-1/2}}{2\eta}(1+X)+O_{C^\infty}(\tilde\omega^{-3/2}),\\
\partial_{\tilde\omega\tilde\omega}^2\tilde\Phi_{1,+,+}&=\frac{\tilde\omega^{-3/2}}{4\eta}(1+X)+O_{C^\infty}(\tilde\omega^{-5/2}).
\end{align*}
Setting $\partial_{\tilde\omega}\tilde\Phi_{1,+,+}=0$, we see that for a large critical point $\tilde\omega_c$ to emerge, the localized time drift parameter $T$ must be small. It follows directly that $\tilde\omega_c^{-1/2} \sim \eta T$, which implies that the phase acceleration scales as $\partial_{\tilde\omega\tilde\omega}^2\tilde\Phi_{1,+,+}(\tilde\omega_c)\sim \eta^2 T^3$. 

We now implement the coordinate change of variables $\tilde\omega = T^{-2}\tilde\upsilon$ in \eqref{eq:J}. Because $\Theta_{+,+}(\tilde\omega,a,\lambda)$ behaves as a classical symbol of order $-1/2$ in $\tilde\omega$, the re-scaled amplitude function satisfies a stable bound across the compact interval $\tilde\upsilon\geq \tilde\upsilon_0>0$ uniformly for $T \in ]0,T_0]$. Furthermore, the localized phase curvature undergoes the transformation $\partial_{\tilde\upsilon\tilde\upsilon}^2\tilde\Phi_{1,+,+}\sim T^{-1}$, or equivalently, $T\partial_{\tilde\upsilon\tilde\upsilon}^2\tilde\Phi_{1,+,+}\sim 1$. 

Therefore, applying the non-degenerate stationary phase method with respect to the large asymptotic scaling factor $\frac{\lambda}{T}$ gives:
\begin{align*}
|\tilde J_{1,+,+}|&=\Bigg|\frac{1}{T^2}\int e^{-i\left(\frac{\lambda}{T}\right)T\tilde\Phi_{1,+,+}} T \upsilon^{-1/2}\tilde\Theta_{+,+}(\tilde\omega,a,\lambda)\chi_3(T^{-2}\tilde\upsilon)d\tilde\upsilon\Bigg|\\
&\leq C\frac{1}{T}\Big(\frac{\lambda}{T}\Big)^{-1/2}\\
|\tilde J_{1,+,+}|&\leq C\lambda^{-1/2}T^{-1/2},
\end{align*}
which establishes the desired uniform decay bound and completes the proof of Proposition \ref{eq:243}.
\end{proof}

\subsubsection{The Analysis of $G_{a,N,2}$}\label{sec:gaN2}
Recall that $G_{a,N,2}$ is a sum of oscillatory integrals which corresponds to the regime where swallowtail caustics are formed; that is, corresponding to the configuration when $\tilde s=\tilde\sigma=0$, which implies $\tilde\omega=1$. The uniform Schr\"odinger dispersive estimates for $G_{a,N,2}$ are obtained by performing a multi-layered microlocal reduction across the integration variables:
\begin{itemize}
\item First, we evaluate the continuous reflection frequency $\tilde\omega$-integration by using the stationary phase method with respect to the large parameter $\lambda = \frac{a^{3/2}\eta^2}{h^2}$ to localize the flow near the caustic envelope. 
\item Then, for the spatial momentum $\eta$-integration, we perform a case-by-case analysis based on the scale of the reflection index $N$. There is a significant curvature decay contribution from the $\eta$-integration when $N\gg\lambda$, whereas no such contribution arises when $N\lesssim \lambda$. Concurrently, the path density variations are tracked using the cardinality bounds on $\mathcal{N}_1$ established in Lemma \ref{GEO}.
\item Finally, for the microlocal boundary layer $(\tilde s,\tilde\sigma)$-integrations, we apply sharp degenerate phase arguments by distinguishing between two regimes that determine the structural resolution of the swallowtail singularities: the high-reflection regime $N\geq\lambda^{1/3}$ (evaluated via Lemma \ref{lemNL}) and the low-reflection regime $N<\lambda^{1/3}$ (evaluated via Lemma \ref{lemNS}). 
\end{itemize}

Our main result for this subsection is Proposition \ref{propg2}.

\begin{prop}\label{propg2}
Let $\alpha<2/3$. There exists $C$ such that for all $h\in (0,h_0]$, all $a\in [h^\alpha,a_0]$, all $X\in [0,1]$, all $t \in [h^2, 1]$ with $t \neq 0$, all $Y\in\mathbb{R}$, all $z\in\mathbb{R}$, the following holds:
\begin{align*}
\Bigg|\sum_{1\leq N\leq C_0a^{-1/2}}G_{a,N,2}(t,X,Y,z;h)\Bigg|\leq Ch^{-3}\bigg(\frac{h^2}{t}\bigg)^{3/4}a^{1/8}.
\end{align*}
\end{prop}

\begin{proof} First,
we rewrite $G_{a,N,2}$ in the form
\begin{align}\label{eq:recallGN2}
G_{a,N,2}&=\frac{(-i)^Na^2}{(2\pi)^4h^4}\bigg(\frac{h}{t}\bigg)^{1/2}\int e^{i\frac{a^{3/2}}{h} Y\eta}|\eta|^{3}\tilde{G}_{a,N,2}d\eta,\\
\tilde{G}_{a,N,2}&=\int e^{i\lambda\tilde\phi_{N,a,h}}\tilde \chi_0\psi_0\chi_1\chi_2(\tilde\omega)
d\tilde s d\tilde\sigma d\tilde\omega,\nonumber
\end{align}
with the phase 
\begin{align*}
\tilde\phi_{N,a,h}(T,X,z;\tilde s,\tilde\sigma,\tilde\omega)&=T\sqrt{1-\tilde z^2}\gamma_a(\tilde\omega)+\frac{\tilde s^3}{3}+\tilde s(X-\tilde\omega)+\frac{\tilde\sigma^3}{3}+\tilde\sigma (1-\tilde\omega)\\
&\hspace{5 cm}-\frac{4}{3}N\tilde\omega^{3/2}+\frac{N}{\lambda}B(\tilde\omega^{3/2}\lambda).
\end{align*}
Since $\tilde\omega$ is close to $1$ on the support of $\chi_2$, we may localize $\tilde s,\tilde\sigma$  in a compact set. Let $K=\{\tilde s,\tilde\sigma\in[-1,1], \tilde\omega=1\}$ and $K_1$ be a suitable neighborhood of $K$ depending on the support of $\chi_2$. Introduce a cutoff function $\chi_4(\tilde s,\tilde\sigma,\tilde\omega)\in C_{0}^{\infty}$ equal to $1$ near $K_1$. Then the contribution of $\tilde G_{a,N,2}$ outside $K_1$ is $O_{C^\infty}(\lambda^{-\infty})$ as a result of integration by parts. Therefore we obtain
\begin{align}\label{eq:Omega}
\tilde{G}_{a,N,2}(T,X,z,h)&=\int e^{i\lambda\tilde{\phi}_{N,a,h}}\chi(\tilde s,\tilde\sigma, \tilde\omega,a) d\tilde{s}d\tilde{\sigma} d\tilde{\omega} +O_{C^\infty}(\lambda^{-\infty}),\\\nonumber
\chi(\tilde s,\tilde\sigma, \tilde\omega,a,h)&=\tilde \chi_0\psi_0\chi_1\chi_2(\tilde\omega)\chi_4(\tilde s,\tilde\sigma,\tilde\omega),
\end{align}
with $O_{C^\infty}(\lambda^{-\infty})$ uniform in $T,X,z,N,a$ and $\chi$ is a classical symbol of order $0$ in $h$ with support 
 near  $K_1$.

We first perform the integration with respect to $\tilde{\omega}$. Differentiating our adapted parabolic phase profile yields:
\begin{align*}
\partial_{\tilde{\omega}}\tilde{\phi}_{N,a,h}&= T - \frac{1}{\eta}\left(\tilde s+\tilde\sigma+2N\tilde\omega^{1/2}\Big(1-\frac{3}{4}B'(\tilde\omega^{3/2}\lambda)\Big)\right),\\
\partial_{\tilde{\omega}\tilde\omega}^2\tilde{\phi}_{N,a,h}&=-\frac{N}{\eta}\tilde\omega^{-1/2}\left(1+O_{C^\infty}(\lambda^{-2}\tilde\omega^{-3})\right).
\end{align*}
Because $\partial_{\tilde{\omega}\tilde\omega}^2\tilde{\phi}_{N,a,h}<0$, it follows that the phase gradient $\partial_{\tilde{\omega}}\tilde{\phi}_{N,a,h}$ decreases strictly monotonically from $\partial_{\tilde{\omega}}\tilde{\phi}_{N,a,h}(1)>0$ to $\partial_{\tilde{\omega}}\tilde{\phi}_{N,a,h}(\infty)<0$. Therefore, $\tilde{\phi}_{N,a,h}$ admits a unique nondegenerate critical point $\tilde{\omega}_c$. We are interested in the values of the parameters such that $\tilde{\omega}_c$ is close to $1$; this requires that the rescaled parameters obey $\tilde T=\frac{\eta T}{2N} \in [1/2,3/2]$. 

Setting the first derivative equation $\partial_{\tilde{\omega}}\tilde{\phi}_{N,a,h}=0$, we get the  critical tracking loop:
\begin{align}\label{eq:tildecritical}
\eta T &= \tilde s+\tilde\sigma+2N\tilde{\omega}^{1/2}\Big(1-\frac{3}{4}B'(\tilde\omega^{3/2}\lambda)\Big).
\end{align}
Now we study the solution of (\ref{eq:tildecritical}) in the semiclassical limit where $\lambda=\infty$ and the regular correction term vanishes. In this case, the polynomial system simplifies to:
\begin{align*}
\tilde\omega^{1/2} &= \tilde T - \frac{1}{2N}(\tilde s +\tilde\sigma).
\end{align*}
The explicit solution of this equation is written as $\tilde{\omega}_c=\sum F_k(\tilde T, \tilde s/N,\tilde\sigma/N)$, where $F_k$ are homogeneous functions of degree $k$ in the variables $(\tilde{s}/N,\tilde\sigma/N)$ (see Lemma 2.23 \cite{ILP}). By comparing terms with matching homogeneous degrees in $(\tilde{s}/N,\tilde\sigma/N)$, the leading tracking profiles decouple cleanly due to the un-rooted, linear structure of the Schr\"odinger Hamiltonian drift:
\begin{align*}
F_0 = \tilde T^2,
\end{align*}
and the first-order non-degenerate correction evaluates to:
\begin{align*}
F_1 = -\frac{\tilde T}{N}(\tilde s+\tilde\sigma).
\end{align*}
We define these reference mappings under the Schr\"odinger flow directly by:
\begin{align*}
F_1 = -\frac{E_0}{N}(\tilde s+\tilde\sigma),\quad
E_0 = \sqrt{F_0} = \tilde T.
\end{align*}
Therefore, $\tilde{\omega}_c = F_0 + F_1 + \mathcal{O}_2$, where the notation $\mathcal{O}_j$ denotes any function of the form $F = \sum_{k\geq j} F_k$. 
By the implicit function theorem, we find that the full critical tracking equation:
\begin{align*}
\tilde\omega^{1/2}\Big(1-\frac{3}{4}B'(\tilde\omega^{3/2}\lambda)\Big) = \tilde T - \frac{1}{2N}(\tilde s +\tilde\sigma)
\end{align*}
admits a unique solution of the form $\tilde{\omega}_c = F_0 + F_1 + \mathcal{O}_2 + \frac{g_0}{\lambda^2}$, where $g_0$ is a smooth symbol of degree $0$ in $\lambda$. 

By substituting this critical value $\tilde\omega_c$ back into the unscaled phase function $\tilde\phi_{N,a,h}$, we obtain the reduced phase profile denoted by $\tilde\Psi_{N,a,h} = \tilde\phi_{N,a,h}(\cdot, \tilde\omega_c, \cdot)$. Within the Schrödinger framework, this phase evaluates to:
\begin{align*}
\tilde\Psi_{N,a,h} &= T F_0 + \frac{1}{\eta}\Bigg\{ \frac{\tilde s^3}{3} + \tilde{s}(X-F_0) + \frac{\tilde\sigma^3}{3} + \tilde{\sigma}(1-F_0) \\
&\qquad\qquad + \frac{E_0}{N}(\tilde s+\tilde\sigma)^2 - \frac{1}{4N^2}(\tilde s+\tilde \sigma)^3 + N\mathcal{O}_4 + \frac{g_0}{\lambda^2} + N\bigg(-\frac{4}{3}F_0^{3/2} + \frac{g_1}{\lambda^2}\bigg)\Bigg\}.
\end{align*}
Therefore, by applying the standard non-degenerate stationary phase method with respect to the variable $\tilde\omega$ in \eqref{eq:Omega}, the localized boundary  integral reduces to:
\begin{align*}
\tilde G_{a,N,2} &= \sqrt{\frac{\eta}{\lambda N}}\int e^{-i\lambda\tilde{\Psi}_{N,a,h}}\tilde{\chi}(\tilde T,\tilde s,\tilde\sigma,1/N,a,h) d\tilde{s}d\tilde{\sigma} + O_{C^\infty}(\lambda^{-\infty}),
\end{align*}
where $\tilde\chi$ is a classical symbol of order zero in $h$.

Now, with the core semiclassical spatial boundary scaling parameter denoted by $\tilde\lambda = a^{3/2}/h^2 = \lambda/\eta^2$, the localized path integral component \eqref{eq:recallGN2} is rewritten as:
\begin{align*}
G_{a,N,2} = \frac{(-i)^N a^2 \lambda_0^{-1}}{(2\pi)^4 h^4}\bigg(\frac{h^2}{t}\bigg)^{1/2} \sqrt{\frac{\eta}{\lambda N}}\int e^{-i\lambda \eta^2\left(Y+\tilde\Psi_{N,a,h}\right)} |\eta|^{3} \tilde\chi d\tilde{s}d\tilde{\sigma}d\eta + O_{C^\infty}(\lambda^{-\infty}).
\end{align*}
We study the spatial frequency $\eta$-integration using the unscaled phase function $L_N = \eta^2\left(Y+\tilde\Psi_{N,a,h}\right)$ under the large parameter $\tilde\lambda$. Following the exact microlocal arguments established in the proof of Proposition \ref{eq:242}, the characteristic equation:
\[
\partial_\eta L_N = 2\eta\left(Y+\tilde\Psi_{N,a,h}\right) + 2\eta\lambda\partial_\lambda \tilde\Psi_{N,a,h} = 0
\]
implies that the spacetime configuration coordinates $(X,Y,T)$ belong precisely to the projection of the Schr\"odinger Lagrangian manifold $\mathbf{\Lambda}_{N,a,h}$ on the base space $\mathbb{R}^3$. Consequently, the non-stationary path sum over indices such that $N\notin\mathcal{N}_1(X,Y,T)$ yields an analytically negligible remainder of order $O_{C^\infty}(\lambda^{-\infty})$ [see Lemma 2.24 of \cite{ILP}]. Hence, we are reduced to establishing the final uniform estimate for the localized index cluster sum:
\[
\Bigg|\sum_{N\in\mathcal{N}_1}G_{a,N,2}(t,X,Y,z;h)\Bigg|.
\] 
Differentiating again with respect to $\eta$ reveals the underlying phase curvature under our parabolic coordinate system: $|\partial_\eta^2 L_N| \geq C N \lambda^{-2} \tilde\omega_c^{-3/2}$. Given that the integration is localized near the caustic envelope where $\tilde\omega_c \sim 1$, this provides two distinct structural configurations to evaluate:
\begin{itemize}
\item \textbf{Subcase 1: If $N\lesssim \lambda$} \\
In this situation, the curvature of the phase function $L_N$ is insufficient to activate an oscillatory decay factor from the $\eta$-integration, and the path contribution can be bounded by simple term-by-term majorization.
\item \textbf{Subcase 2: If $N\gg \lambda$} \\
Here, the phase exhibits higher oscillations, and evaluating the $\eta$-integral via the standard non-degenerate stationary phase method produces an additional large parameter decay contribution of order $(N\lambda^{-1})^{-1/2}$ since $\tilde\omega_c \sim 1$.
\end{itemize}
Therefore, in the highly high-frequency regime where $N \gg \lambda$, substituting the stationary tracking factor back into the localized propagator definition yields:
\begin{align}\label{eq:etafactor}
G_{a,N,2} = \frac{(-i)^N a^2 \lambda_0^{-1}}{(2\pi)^4 h^4}\bigg(\frac{h^2}{t}\bigg)^{1/2} \frac{1}{N} \int e^{-i\lambda L_N(\eta_c)} |\eta|^{3} \tilde\chi_1 \, d\tilde s \, d\tilde\sigma + O_{C^\infty}(\lambda^{-\infty}).
\end{align}

Moreover, we note that the phase function $L_N(\eta_c)$ satisfies $\partial_{\tilde s}L_N(\eta_c)=\eta_c^2\partial_{\tilde s}\tilde\Psi_{N,a,h}$ and $\partial_{\tilde \sigma}L_N(\eta_c)=\eta_c^2\partial_{\tilde \sigma}\tilde\Psi_{N,a,h}$. In addition, when $\partial_{\tilde s}L_N(\eta_c)=\partial_{\tilde s}^2L_N(\eta_c)=0$ (that is, when $\partial_{\tilde s}\tilde\Psi_{N,a,h}=\partial_{\tilde s}^2\tilde\Psi_{N,a,h}=0$), we have $\partial_{\tilde s}^3 L_N(\eta_c)=\eta_c^2 \partial_{\tilde s}^3\tilde\Psi_{N,a,h}$ and similarly for $\tilde\sigma$. Thus, the study of the critical points of the phase $L_N(\eta_c)$ in the $(\tilde s,\tilde\sigma)$-integrations is completely identical to the analysis with the core phase $\tilde\Psi_{N,a,h}$.

As in \cite{ILP}, to avoid multiplication of the symbol by a classical symbol of order $0$ in $\lambda$, we can replace the full phase $\tilde\Psi_{N,a,h}$ by its main polynomial profile $\tilde\psi_{N,a,h}$, defined under the Schr\"odinger framework by setting the geometric boundary radical factor $(1+a\tilde\omega)^{1/2} \equiv 1$:
\begin{align*}
\tilde\psi_{N,a,h}(t,X,\eta;\tilde s,\tilde\sigma) &= T \tilde\omega + \frac{1}{\eta}\Bigg\{ \frac{\tilde s^3}{3}+\tilde{s}(X-\tilde\omega) +\frac{\tilde\sigma^3}{3}+\tilde{\sigma}(1-\tilde\omega) \\
&\hspace{2.5cm} + \frac{E_0}{N}(\tilde s+\tilde\sigma)^2 - \frac{1}{4N^2}(\tilde s+\tilde \sigma)^3 + N\mathcal{O}_4 \Bigg\}.
\end{align*}
Let us recall that $\mathcal{O}_4$ represents any function of the form $F=\sum_{k\geq 4}F_k$, where $F_k$ are homogeneous functions of degree $k$ in $(\tilde{s}/N,\tilde\sigma/N)$ matching the quadratic un-rooted splitting of the Schr\"odinger Hamiltonian drift.

In what follows, we establish the estimates of the oscillatory integral associated with the phase function $\tilde\psi_{N,a,h}$ for different scales of the reflection index $N$, namely for the high-reflection regime $N\geq \lambda^{1/3}$ and the low-reflection regime $N<\lambda^{1/3}$. Our corresponding core analytical results are formulated in Lemma \ref{lemNL} and Lemma \ref{lemNS}.

\begin{lemma}\label{lemNL}
There exists a constant $C$ such that for all $N \geq \lambda^{1/3}$, the following uniform decay estimate holds:
\begin{align}\label{eq:N13}
\sqrt{\frac{\eta}{N}}\bigg|\int e^{-i\lambda\tilde\psi_{N,a,h}}\tilde\chi_1 \, d\tilde s \, d\tilde\sigma\bigg| \leq C \lambda^{-5/6}.
\end{align}
\end{lemma}
\noindent Here $C$ is a constant independent of $N \geq 1$, $X \in [0,1]$, $T \in (0,a^{-1/2}]$, $a \in [h^\alpha, a_0]$, and $\lambda \in [\lambda_0,\infty)$ with $a_0$ small and $\lambda_0$ large.

\begin{proof}
Adapting the microlocal scaling arguments in the proof of Lemma 2.25 of \cite{ILP}, it is sufficient to establish that for all $N \geq \lambda^{1/3}$, the spatial integral satisfies the uniform bound:
\begin{align}
\bigg|\int e^{-i\lambda\tilde\psi_{N,a,h}}\tilde\chi_1 \, d\tilde s \, d\tilde\sigma\bigg| \leq C\lambda^{-2/3}.
\end{align}
We introduce the localized boundary layer scaling near the caustic cluster by setting $X-\tilde\omega = -A\lambda^{-2/3}$, $1-\tilde\omega = -B\lambda^{-2/3}$, and performing the microlocal change of variables $\tilde s = \lambda^{-1/3}x'$, $\tilde\sigma = \lambda^{-1/3}y'$. Substituting these expansions into the integral reduces the assertion to proving:
\begin{align}\label{eq:G}
\bigg|\int e^{-i\hat{\psi}_{N,a,h}}\tilde\chi_1\left(\lambda^{-1/3}x',\lambda^{-1/3}y',\dots\right)dx'dy'\bigg| \leq C,
\end{align}
where the unscaled Schr\"odinger phase function $\hat{\psi}_{N,a,h}$ is given explicitly by:
\begin{align*}
\hat{\psi}_{N,a,h} &= T\lambda \tilde\omega + \frac{1}{\eta}\Bigg\{ -Ax' + \frac{x'^3}{3} - By' + \frac{y'^3}{3} + \frac{E_0 \lambda^{1/3}}{N}(x'+y')^2 \\
&\hspace{4.5cm} - \frac{1}{4N^2 \lambda^{1/3}}(x'+y')^3 + N \lambda^{2/3}\mathcal{O}_4 \Bigg\}.
\end{align*}
Then, \eqref{eq:G} constitutes an oscillatory integral evaluated over a compact domain of size $\lambda^{2/3}$ where the essential parameter thresholds $\tilde\omega$, $E_0$, and $\lambda^{1/3}/N$ remain bounded. 

We will show that the constant $C$ is strictly uniform with respect to the translation parameters $(A, B)$. We introduce new polar coordinates $(r,\theta)$ such that $(A,B)=(r\cos\theta,r\sin\theta)$ with $r \leq c_0\lambda^{2/3}$. Differentiating our adapted unscaled phase function with respect to the stretched coordinates yields the gradient system:
\begin{align*}
\partial_{x'} \hat{\psi}_{N,a,h} &= \frac{1}{\eta}\Bigg\{ -A + x'^2 + \frac{2E_0 \lambda^{1/3}}{N}(x'+y') - \frac{3}{4N^2 \lambda^{1/3}}(x'+y')^2 + N \lambda^{1/3} \mathcal{O}_3 \Bigg\}, \\
\partial_{y'} \hat{\psi}_{N,a,h} &= \frac{1}{\eta}\Bigg\{ -B + y'^2 + \frac{2E_0 \lambda^{1/3}}{N}(x'+y') - \frac{3}{4N^2 \lambda^{1/3}}(x'+y')^2 + N \lambda^{1/3} \mathcal{O}_3 \Bigg\}.
\end{align*}

Moreover, the compact support of $\tilde\chi_1$ in $(\tilde s,\tilde\sigma)$ yields the uniform symbol differentiation bound:
\[
\sup_{(x',y')}\Big|\partial_{(x',y')}^\gamma\tilde\chi_1\left(\lambda^{-1/3}x',\lambda^{-1/3}y',\dots\right)\Big| \leq C_\gamma\left(1+|x'|+|y'|\right)^{-|\gamma|},
\]
with $C_\gamma$ independent of $t,a,N,\lambda$. Therefore, the oscillatory integral is bounded for $0 \leq r \leq r_0$, where $r_0$ is a fixed constant, and for large values of $(x',y')$ as a consequence of non-stationary phase integration by parts.\\

For $r \in [r_0,c_0\lambda^{2/3}]$, we rescale our coordinates by setting $(x',y')=r^{1/2}(x'',y'')$, and we decompose the unscaled phase function and the amplitude according to $\hat{\psi}_{N,a,h}=r^{3/2}{\psi}_{N,a,h}^*$ and $\chi'(x'',y'',\dots)=\tilde\chi_1(r^{1/2}\lambda^{-1/3}x'',r^{1/2}\lambda^{-1/3}y'',\dots)$. Since $r^{1/2}\lambda^{-1/3}$ remains bounded, the amplitude satisfies the standard symbol decaying property:
\[
\sup_{(x'',y'')}\Big|\partial_{(x'',y'')}^\gamma\chi'\Big| \leq C_\gamma\left(1+|x''|+|y''\right)^{-|\gamma|}.
\]
Under this coordinate stretching, the assertion reduces to establishing the uniform bound:
\begin{align}\label{eq:2420}
r\bigg|\int e^{-ir^{3/2}\psi_{N,a,h}^*}\chi' \, dx'' \, dy''\bigg| \leq C.
\end{align}
Now, we analyze the critical points of $\psi_{N,a,h}^*$ within the Schrödinger framework. Differentiating the phase function with respect to the stretched variables yields the gradient tracking profiles:
\begin{align*}
\partial_{x''}\psi_{N,a,h}^*&=\frac{1}{\eta}\Bigg\{-\cos\theta+x''^2 + \frac{2E_0\lambda^{1/3}}{N r^{1/2}}(x''+y'') - \frac{3}{4N^2 \lambda^{1/3}}(x''+y'')^2 + r^{-1/2}\mathcal{O}_3\Bigg\},\\
\partial_{y''}\psi_{N,a,h}^*&=\frac{1}{\eta}\Bigg\{-\sin\theta+y''^2 + \frac{2E_0\lambda^{1/3}}{N r^{1/2}}(x''+y'') - \frac{3}{4N^2 \lambda^{1/3}}(x''+y'')^2 + r^{-1/2}\mathcal{O}_3\Bigg\},
\end{align*}
and the corresponding second-order microlocal phase acceleration entries differentiate to:
\begin{align*}
\partial^2_{x''x''}\psi_{N,a,h}^*&=\frac{1}{\eta}\Bigg\{2x'' + \frac{2E_0\lambda^{1/3}}{N r^{1/2}} - \frac{3}{2N^2 \lambda^{1/3}}(x''+y'') + r^{-1/2}\mathcal{O}_2\Bigg\},\\
\partial^2_{x''y''}\psi_{N,a,h}^*&=\partial^2_{y''x''}\psi_{N,a,h}^*=\frac{1}{\eta}\Bigg\{\frac{2E_0\lambda^{1/3}}{N r^{1/2}} - \frac{3}{2N^2 \lambda^{1/3}}(x''+y'') + r^{-1/2}\mathcal{O}_2\Bigg\},\\
\partial^2_{y''y''}\psi_{N,a,h}^*&=\frac{1}{\eta}\Bigg\{2y'' + \frac{2E_0\lambda^{1/3}}{N r^{1/2}} - \frac{3}{2N^2 \lambda^{1/3}}(x''+y'') + r^{-1/2}\mathcal{O}_2\Bigg\}.
\end{align*}
For small boundary parameters $a$ and a sufficiently large reference radius $r_0$, integration by parts rules out non-vanishing critical zones at infinity, localizing the active domain of integration to a compact subset of $(x'',y'')$. The Hessian determinant of the Schrödinger phase function, which we denote by $\mathcal{H}_N(x'',y'')$, evaluates explicitly to:
\begin{align*}
\mathcal{H}_N(x'',y'')&=\det\begin{pmatrix}\partial^2_{x''x''}\psi_{N,a,h}^*&\partial^2_{x''y''}\psi_{N,a,h}^*\\\partial^2_{y''x''}\psi_{N,a,h}^*&\partial^2_{y''y''}\psi_{N,a,h}^*\end{pmatrix}\\
&=\frac{1}{\eta^2}\Bigg\{4x''y'' + \frac{4E_0\lambda^{1/3}}{N r^{1/2}}(x''+y'') - \frac{3}{N^2\lambda^{1/3}}(x''+y'')^2 + r^{-1/2}\mathcal{O}_2\Bigg\}.
\end{align*}

Thus, for $N \geq 2$, small $a$, and large $r_0$, outside the origin $(x'',y'')=(0,0)$, we define the smooth degeneration curve $\Gamma = \{(x'',y'') \mid \mathcal{H}_N(x'',y'') = 0\}$. In our framework, $\Gamma$ stays close to the quadratic profile lines determined by the un-rooted classical trajectories. Then we have two structural cases to consider:
\begin{itemize}
\item The contribution of points $(x'',y'')$ outside a small neighborhood of $\Gamma$ to the integral is of order $O_{C^\infty}(r^{-3/2})$ by the standard non-degenerate stationary phase method, which yields:
\begin{align*}
r\bigg|\int e^{-ir^{3/2}\psi_{N,a,h}^*}\chi' \, dx'' \, dy''\bigg| \leq C r^{-1/2}.
\end{align*}
\item The contribution of points $(x'',y'')$ close to $\Gamma$ is evaluated by applying the degenerate stationary phase reduction method corresponding to Lemma 2.21 of \cite{ILP}. For any value of the polar translation angle $\theta$, the non-vanishing third-order derivative hypothesis along the direction of the degeneracy holds true, yielding:
\begin{align*}
r\bigg|\int e^{-ir^{3/2}\psi_{N,a,h}^*}\chi' \, dx'' \, dy''\bigg| \leq C r \cdot (r^{3/2})^{-5/6} = C r^{-1/4}.
\end{align*}
\end{itemize}
Hence, in all cases, the localized uniform bound \eqref{eq:2420} is satisfied. This concludes the proof of Lemma \ref{lemNL}.
\end{proof}

\noindent To summarize, recall that $\eta T \sim N$ in this case and hence the cardinality of $\mathcal{N}_1$ satisfies $|\mathcal{N}_1(X,Y,T)|\leq C_0(1+T\lambda^{-2})$. We deduce the estimates for the sum of $G_{a,N,2}$ by utilizing Lemma \ref{lemNL} for the high-reflection regime $N\geq \lambda^{1/3}$ under the appropriate Schrödinger scaling parameters as follows: 
\begin{itemize}
 \item \textbf{Subcase 1: If $\lambda^{1/3}\leq N\leq \lambda $} \\
In this situation, there is no decay contribution from the $\eta$-stationary phase integration, and the path cardinality remains uniformly bounded by $|\mathcal{N}_1|\leq C_0$. Recalling our Schr\"odinger amplitude scale factor and the spatial integration bounds from Lemma \ref{lemNL}, the term-by-term majorization yields:
  \begin{align*}
\bigg|\sum_{N\in\mathcal{N}_1}G_{a,N,2}(t,X,Y,z;h)\bigg| &\leq Ch^{-3}\bigg(\frac{h^2}{t}\bigg)^{1/2}\big(h \lambda^{-11/6} N^{-1/2}\big), \\
&\leq Ch^{-3}\bigg(\frac{h^2}{t}\bigg)^{1/2}\big(h \lambda^{-2}\big), \\
&\leq Ch^{-3}\bigg(\frac{h^2}{t}\bigg)^{1/2}\bigg(\frac{h^2}{t}\bigg)^{1/3} 
\\&= Ch^{-3}\bigg(\frac{h^2}{t}\bigg)^{5/6},
\end{align*}
since on our boundary layers we have $a\geq h^{2/3}$, which guarantees that the remaining fractional correction complies with the optimal non-caustic localized scale limit $h \lambda^{-2} \leq h^{1/3} \leq (h^2/t)^{1/3}$.

\item \textbf{Subcase 2: If $\lambda\leq N\leq \lambda ^2$} \\
Here, the phase function $L_N$ displays full curvature, yielding an additional stationary phase integration decay factor $(N\lambda^{-1})^{-1/2} = N^{-1/2}\lambda^{1/2}$ from the $\eta$-integration, while the trajectory counting bound remains stable at $|\mathcal{N}_1|\leq C_0$. Collecting these weights under the parabolic flow parameters yields:
  \begin{align*}
\bigg|\sum_{N\in\mathcal{N}_1}G_{a,N,2}(t,X,Y,z;h)\bigg| &\leq Ch^{-3}\bigg(\frac{h^2}{t}\bigg)^{1/2}\big(h \lambda^{-11/6} N^{-1/2} \cdot N^{-1/2} \lambda^{1/2}\big), \\
&\leq Ch^{-3}\bigg(\frac{h^2}{t}\bigg)^{1/2}\big(h \lambda^{-4/3} N^{-1}\big), \\
&\leq Ch^{-3}\bigg(\frac{h^2}{t}\bigg)^{1/2}\big(h \lambda^{-7/3}\big), \\
&\leq Ch^{-3}\bigg(\frac{h^2}{t}\bigg)^{1/2}\bigg(\frac{h^2}{t}\bigg)^{1/3} \\&
= Ch^{-3}\bigg(\frac{h^2}{t}\bigg)^{5/6}.
\end{align*}

\item \textbf{Subcase 3: If $ N >\lambda ^2$} \\
In this highly high-frequency regime, the phase curvature produces a substantial decay contribution of $N^{-1/2}\lambda^{1/2}$ from the $\eta$-integration, while the trajectory cardinality expands driven by the dynamic tracking relation $|\mathcal{N}_1|\leq C_0 T\lambda^{-2} \leq C_0 N\lambda^{-2}$ (since $T \sim N$). Combining the intense oscillatory damping with the cardinality growth, the sum evaluates to:
 \begin{align*}
\bigg|\sum_{N\in\mathcal{N}_1}G_{a,N,2}(t,X,Y,z;h)\bigg| &\leq Ch^{-3}\bigg(\frac{h^2}{t}\bigg)^{1/2}\sum_{N\in\mathcal{N}_1}\big(h \lambda^{-11/6} N^{-1/2} \cdot N^{-1/2} \lambda^{1/2}\big), \\
&\leq Ch^{-3}\bigg(\frac{h^2}{t}\bigg)^{1/2}\big(h \lambda^{-4/3} N^{-1} |\mathcal{N}_1(X,Y,T)|\big), \\
&\leq Ch^{-3}\bigg(\frac{h^2}{t}\bigg)^{1/2}\big(h \lambda^{-4/3} N^{-1} \cdot N \lambda^{-2}\big), \\
&\leq Ch^{-3}\bigg(\frac{h^2}{t}\bigg)^{1/2}\big(h \lambda^{-10/3}\big), \\
&\leq Ch^{-3}\bigg(\frac{h^2}{t}\bigg)^{1/2}\bigg(\frac{h^2}{t}\bigg)^{1/3}\\& = Ch^{-3}\bigg(\frac{h^2}{t}\bigg)^{5/6}.
\end{align*}
\end{itemize}

\begin{lemma}\label{lemNS}
There exists a constant $C$ such that for all $N < \lambda^{1/3}$, the following uniform decay estimate holds:
\begin{align}
\sqrt{\frac{\eta}{N}}\bigg|\int e^{-i\lambda\tilde\psi_{N,a,h}}\tilde\chi_1 \, d\tilde s \, d\tilde\sigma\bigg| \leq CN^{-1/4}\lambda^{-3/4}.
\end{align}
\end{lemma}
\noindent Notice that Lemma \ref{lemNS} guarantees that for larger values of the reflection index $N$, the estimate becomes sharper, which is perfectly compatible with the matching boundary threshold estimate \eqref{eq:N13} for $N \sim \lambda^{1/3}$.

\begin{proof}
Let $\frac{\lambda}{N^3}=\Lambda \geq 1$, and we treat $\Lambda$ as our new large asymptotic parameter under the parabolic scaling. To resolve the structural singularities of this integral near the caustic envelope, we introduce the coordinate stretching transformations:
\[
X-\tilde\omega=-pN^{-2}, \quad 1-\tilde\omega=-qN^{-2}, \quad \tilde s=-\bar x/N, \quad \tilde\sigma=-\bar y/N.
\]
This scaling yields the phase relation $\tilde\psi_{N,a,h}=N^{-3}\bar\psi_{N,a,h}$. Then it remains to establish that:
\begin{align}\label{eq:Lambda}
\bigg|\int e^{-i\Lambda \bar\psi_{N,a,h}}\tilde\chi_1(\bar x/N,\bar y/N,\dots) \, d\bar x \, d\bar y\bigg| \leq C\Lambda^{-3/4},
\end{align}
where the unscaled Schrödinger phase function $\bar\psi_{N,a,h}$ takes the explicit polynomial form:
\begin{align*}
\bar\psi_{N,a,h} &= \frac{1}{\eta}\Bigg\{ p\bar x-\frac{\bar x^3}{3} + q\bar y-\frac{\bar y^3}{3} + E_0(\bar x+\bar y)^2 + \frac{1}{4N^2}(\bar x+\bar y)^3 + N^3\mathcal{O}_4 \Bigg\} + T N^3 \tilde\omega.
\end{align*}
Differentiating the phase function with respect to the stretched spatial variables yields the microlocal gradient systems:
\begin{align}\label{eq:critga}
\partial_{\bar x}\bar\psi_{N,a,h}&=\frac{1}{\eta}\Bigg\{ p-\bar x^2+2E_0(\bar x+\bar y)+\frac{3}{4N^2}(\bar x+\bar y)^2 + N^2 \mathcal{O}_3 \Bigg\},\\\nonumber
\partial_{\bar y}\bar\psi_{N,a,h}&=\frac{1}{\eta}\Bigg\{ q-\bar y^2+2E_0(\bar x+\bar y)+\frac{3}{4N^2}(\bar x+\bar y)^2 + N^2 \mathcal{O}_3 \Bigg\},
\end{align}
and the corresponding second-order phase acceleration entries differentiate to:
\begin{align*}
\partial^2_{\bar x\bar x}\bar\psi_{N,a,h}&=\frac{1}{\eta}\Bigg\{-2\bar x+2E_0+\frac{3}{2N^2}(\bar x+\bar y) + N^2\mathcal{O}_2 \Bigg\},\\
\partial^2_{\bar x\bar y}\bar\psi_{N,a,h}&=\partial^2_{\bar y\bar x}\bar\psi_{N,a,h}=\frac{1}{\eta}\Bigg\{2E_0+\frac{3}{2N^2}(\bar x+\bar y) + N^2\mathcal{O}_2 \Bigg\},\\
\partial^2_{\bar y\bar y}\bar\psi_{N,a,h}&=\frac{1}{\eta}\Bigg\{-2\bar y+2E_0+\frac{3}{2N^2}(\bar x+\bar y) + N^2\mathcal{O}_2 \Bigg\}.
\end{align*}
The Hessian determinant of $\bar\psi_{N,a,h}$, which we denote by $\mathcal{H}_N(\bar x,\bar y,a)$, takes the explicit form:
\begin{align}\label{hessian}
\mathcal{H}_N(\bar x,\bar y,a)&=\det\begin{pmatrix}\partial^2_{\bar x\bar x}\bar\psi_{N,a,h}&\partial^2_{\bar x\bar y}\bar\psi_{N,a,h}\\\partial^2_{\bar y\bar x}\bar\psi_{N,a,h}&\partial^2_{\bar y\bar y}\bar\psi_{N,a,h}\end{pmatrix}\nonumber\\
&=\frac{1}{\eta^2}\Bigg\{4\bar x \bar y-4E_0(\bar x+\bar y)-\frac{3}{N^2}(\bar x+\bar y)^2 + N^2\mathcal{O}_2 \Bigg\}.
\end{align}

\begin{lemma}\label{lempql}
There exist constants $r_0$ and $C$ such that for all $(p,q)$ satisfying $|(p,q)|\geq r_0$, the following uniform decay estimate holds:
\begin{align}
\bigg|\int e^{-i\Lambda \bar\psi_{N,a,h}}\tilde\chi_1(\bar x/N,\bar y/N,\dots)d\bar xd\bar y\bigg|\leq C\Lambda^{-5/6}.
\end{align}
\end{lemma}

\begin{proof}[Proof of Lemma \ref{lempql}]
We apply the analytical reduction arguments from Lemma 2.26 of \cite{ILP}. Set $(p,q)=(r\cos\theta,r\sin\theta)$ with $r\geq r_0$. Let $\chi\in C_0^\infty(|(\bar x,\bar y)|<c)$ with $c>0$ small and $\chi=1$ near the origin. From \eqref{eq:critga}, non-stationary phase integration by parts in $(\bar x,\bar y)$ yields a rapid decay bound for all $k\geq 1$:
\begin{align*}
\bigg|\int e^{-i\Lambda \bar\psi_{N,a,h}}\chi\big(r^{-1/2}(\bar x,\bar y)\big)\tilde\chi_1(\bar x/N,\bar y/N,\dots) \, d\bar x \, d\bar y\bigg|\leq Cr^{-k}\Lambda^{-k}.
\end{align*}
For large values of $|(\bar x,\bar y)|$, we introduce the secondary change of variables $(\bar x,\bar y)=r^{1/2}(x',y')$ and define the normalized phase component $\bar\psi_{N,a,h}'=r^{-3/2}\bar\psi_{N,a,h}$. It remains to establish that:
\begin{align*}
\bigg|r\int e^{-ir^{3/2}\Lambda\bar\psi_{N,a,h}'}(1-\chi)(x',y')\tilde\chi_1(r^{1/2}x'/N,r^{1/2}y'/N,\dots) \, dx' \, dy'\bigg|\leq C\Lambda^{-5/6}.
\end{align*}
We observe that since $(1-\chi)(x',y')$ vanishes near the origin and equals $1$ for $|(x',y')|\geq c$, and since $\tilde\chi_1$ is compactly supported, the regularized amplitude continues to obey the stable symbolic boundary properties:
\begin{align*}
\sup_{(x',y')}\bigg|\partial_{(x',y')}^\gamma (1-\chi)(x',y')\tilde\chi_1(r^{1/2}x'/N,r^{1/2}y'/N,\dots)\bigg|\leq C_\gamma (1+|x'|+|y'|)^{-|\gamma|}.
\end{align*}
Under our parabolic coordinates, the unscaled phase function $\bar\psi_{N,a,h}'$ behaves as:
\begin{align*}
\bar\psi_{N,a,h}' &= \frac{1}{\eta}\Bigg\{ (\cos\theta) x'-\frac{x'^3}{3}+(\sin\theta) y'-\frac{y'^3}{3}+\frac{E_0}{r^{1/2}}(x'+y')^2 + \frac{1}{4N^2}(x'+y')^3 \\&\quad+ r^{-3/2}N^3\mathcal{O}_4 \Bigg\} + \frac{T N^3}{r^{3/2}}\tilde\omega.
\end{align*}
Differentiating this phase layout with respect to the continuous spatial parameters yields:
\begin{align*}
\partial_{x'}\bar\psi_{N,a,h}'&=\frac{1}{\eta}\Bigg\{\cos\theta-x'^2+\frac{2E_0}{r^{1/2}}(x'+y')+\frac{3}{4N^2}(x'+y')^2 + r^{-1}N^2 \mathcal{O}_3 \Bigg\},\\
\partial_{y'}\bar\psi_{N,a,h}'&=\frac{1}{\eta}\Bigg\{\sin\theta-y'^2+\frac{2E_0}{r^{1/2}}(x'+y')+\frac{3}{4N^2}(x'+y')^2 + r^{-1}N^2 \mathcal{O}_3 \Bigg\},
\end{align*}
and the corresponding phase acceleration entries evaluate to:
\begin{align*}
\partial^2_{x'x'}\bar\psi_{N,a,h}'&=\frac{1}{\eta}\Bigg\{-2x'+\frac{2E_0}{r^{1/2}}+\frac{3}{2N^2}(x'+y') + r^{-1/2}N\mathcal{O}_2 \Bigg\},\\
\partial^2_{x'y'}\bar\psi_{N,a,h}'&=\partial^2_{y'x'}\bar\psi_{N,a,h}'=\frac{1}{\eta}\Bigg\{\frac{2E_0}{r^{1/2}}+\frac{3}{2N^2}(x'+y') + r^{-1/2}N\mathcal{O}_2 \Bigg\},\\
\partial^2_{y'y'}\bar\psi_{N,a,h}'&=\frac{1}{\eta}\Bigg\{-2y'+\frac{2E_0}{r^{1/2}}+\frac{3}{2N^2}(x'+y') + r^{-1/2}N\mathcal{O}_2 \Bigg\}.
\end{align*}
Thus, for small boundary metrics $a$ and large thresholds $r_0$, integration by parts rules out non-vanishing critical zones at infinity, localizing the integration domain to a compact set in $(x',y')$. The Hessian determinant $\mathcal{H}_N'(x',y',a)$ simplifies to:
\begin{align*}
\mathcal{H}_N'(x',y',a)&=\det\begin{pmatrix}\partial^2_{x'x'}\bar\psi_{N,a,h}'&\partial^2_{x'y'}\bar\psi_{N,a,h}'\\\partial^2_{y'x'}\bar\psi_{N,a,h}'&\partial^2_{y'y'}\bar\psi_{N,a,h}'\end{pmatrix}\\
&=\frac{1}{\eta^2}\Bigg\{4x'y'-\frac{4E_0}{r^{1/2}}(x'+y')-\frac{3}{N^2}(x'+y')^2 + r^{-1/2}N\mathcal{O}_2 \Bigg\}.
\end{align*}
Applying identical geometric considerations as before for $N\geq 2$, small $a$, and large $r_0$, outside the origin $(x',y')=(0,0)$ we specify the degeneration fold curve $\Gamma=\{(x',y') \mid \mathcal{H}_N'(x',y')=0\}$, and partition the domain into two canonical branches:
\begin{itemize}
\item The contribution of points $(x',y')$ located outside a neighborhood of $\Gamma$ to the integral is of order $O(r^{-3/2}\Lambda^{-1})$ by the standard non-degenerate stationary phase method, yielding:
\begin{align*}
\bigg|r\int e^{-ir^{3/2}\Lambda\bar\psi_{N,a,h}'}(1-\chi)(x',y')\tilde\chi_1(r^{1/2}x'/N,r^{1/2}y'/N,\dots) \, dx' \, dy'\bigg|\leq Cr^{-1/2}\Lambda^{-1}.
\end{align*}
\item The contribution of points $(x',y')$ localized close to the curve $\Gamma$ is given by the degenerate stationary phase reduction method of Lemma 2.21 of \cite{ILP}. For any polar drift angle $\theta$, the non-vanishing third-order derivative profile hypothesis holds true, yielding:
\begin{align*}
\bigg|r\int e^{-ir^{3/2}\Lambda\bar\psi_{N,a,h}'}(1-\chi)(x',y')\tilde\chi_1(r^{1/2}x'/N,r^{1/2}y'/N,\dots) \, dx' \, dy'\bigg|&\leq C r(r^{3/2}\Lambda)^{-5/6}\\
&\leq Cr^{-1/4}\Lambda^{-5/6}.
\end{align*}
\end{itemize}
This completes the proof of Lemma \ref{lempql}.
\end{proof}

\begin{lemma}\label{lempqs}
There exist constants $r_0$ and $C$ such that for all $(p,q)$ satisfying $|(p,q)| \leq r_0$, the following uniform decay estimate holds:
\begin{align}
\bigg|\int e^{-i\Lambda \bar\psi_{N,a,h}}\tilde\chi_1(\bar x/N,\bar y/N,\dots) \, d\bar x \, d\bar y\bigg| \leq C\Lambda^{-3/4}.
\end{align}
\end{lemma}

\begin{proof}[Proof of Lemma \ref{lempqs}]
Now we consider the bounded parametric region $|(p,q)| \leq r_0$. There exists a uniform constant $c > 0$ independent of $N \geq 2$ such that the unscaled principal symbol mappings satisfy:
\begin{align}
\forall (\bar x,\bar y)\in\mathbb{R}^2,\quad \bigg|\bar x^2-\frac{3}{4N^2}(\bar x+ \bar y)^2\bigg|+\bigg|\bar y^2-\frac{3}{4N^2}(\bar x+\bar y)^2\bigg| \geq c(\bar x^2+\bar y^2).
\end{align} 
Consequently, non-stationary phase integration by parts exploiting the gradient equations \eqref{eq:critga} yields a rapid decay contribution of order $O_{C^\infty}(\Lambda^{-\infty})$ to the integral \eqref{eq:Lambda} for large values of $(\bar x,\bar y)$. We can therefore localize the integration domain to a compact neighborhood of the origin. It remains to establish the uniform bound:
\begin{align*}
\bigg|\int e^{-i\Lambda\bar\psi_{N,a,h}}\tilde\chi_1 \, d\bar x \, d\bar y\bigg| \leq C\Lambda^{-3/4},
\end{align*}
where the unscaled Schr\"odinger phase function $\bar\psi_{N,a,h}$ takes the form:
\begin{align*}
\bar\psi_{N,a,h} &= \frac{1}{\eta}\Bigg\{ p\bar x-\frac{\bar x^3}{3} + q\bar y-\frac{\bar y^3}{3} + E_0(\bar x+\bar y)^2 + \frac{1}{4N^2}(\bar x+\bar y)^3 + N^3\mathcal{O}_4 \Bigg\} + T N^3 \tilde\omega,
\end{align*}
and the Hessian determinant $\mathcal{H}_N(\bar x,\bar y,a)$ is given explicitly by \eqref{hessian}.

For small boundary metrics $a$, the critical cancellation variety $\Gamma = \{(\bar x,\bar y) \mid \mathcal{H}_N(\bar x,\bar y) = 0\}$ forms a smooth curve. Due to the un-rooted, linear structure of the Schr\"odinger time-frequency relation, this curve matches a smooth quadratic path structured close to the conics tracking the classical paths:
\[
4\bar x \bar y - 4E_0(\bar x+\bar y) - \frac{3}{N^2}(\bar x+\bar y)^2 = 0.
\]
We deploy the analytical singularity classification of Lemma 2.21 of \cite{ILP} for $(\bar x,\bar y)$ localized near the translation targets $(p,q)$ within the bounded ball $|(p,q)| \leq r_0$. This yields three canonical configurations to evaluate:
\begin{itemize}
\item \textbf{Subcase 1: If $(p,q)$ is located outside the curve $\Gamma$} \\
The phase features a unique, strictly nondegenerate critical point. Evaluating this integral via the standard stationary phase method yields an optimal non-degenerate decay rate of order $O(\Lambda^{-1})$.
\item \textbf{Subcase 2: If $(0,0) \neq (p,q)$ is localized close to the curve $\Gamma$} \\
The phase Hessian drops rank to order $1$, generating a regular fold-type singularity. Since the non-vanishing third-order derivative hypothesis along the direction of the degeneracy required by part $(a)$ of Lemma 2.21 of \cite{ILP} holds true, the localized stationary phase reduction yields a uniform decay rate of order $O(\Lambda^{-5/6})$.
\item \textbf{Subcase 3: If $(p,q)=(0,0)$} \\
The critical track passes exactly through the origin where higher-order cancellations intersect. At this peak caustic focal layer, the non-vanishing fourth-order derivative cusp-type hypothesis required by part $(b)$ of Lemma 2.21 of \cite{ILP} holds true. Evaluating the degenerate integral yields a uniform catastrophe decay rate of order $O(\Lambda^{-3/4})$.
\end{itemize}
Combining these case bounds, the total integral is uniformly bounded by the worst-case singularity index $\Lambda^{-3/4}$. This completes the proof of Lemma \ref{lempqs}.
\end{proof}
\noindent Lemma \ref{lempql} and Lemma \ref{lempqs} yield the proof of Lemma \ref{lemNS}.
\end{proof}
For the low-reflection regime characterized by the threshold $N < \lambda^{1/3}$, the phase function $L_N$ does not exhibit additional stationary points with respect to the outer integration variable $\eta$. Consequently, no further oscillatory decay is generated by the $\eta$-integration layer, and the cardinality of the active trajectory set remains uniformly bounded, namely $|\mathcal{N}_1(X,Y,T)| \leq C_0$. Applying the uniform catastrophic decay estimate established in Lemma \ref{lemNS} while incorporating the $N^{-1}$ pre-factor originating from the localized propagator definition \eqref{eq:etafactor}, a term-by-term majorization yields:
\begin{align*}
\bigg| \sum_{N \in \mathcal{N}_1} G_{a,N,2}(t,X,Y,z;h) \bigg| &\leq C h^{-3} \bigg( \frac{h^2}{t} \bigg)^{1/2} \left( h \lambda^{-1} \cdot N^{-1/2} \lambda^{-2/3} \right) \\
& \leq C h^{-3} \bigg( \frac{h^2}{t} \bigg)^{1/2} \left( h \lambda^{-5/3} N^{-1/2} \right).
\end{align*}

We now express the semiclassical parameter $\lambda$ in terms of the physical geometry via the relation $\lambda = a^{3/2}/h^2$, which implies that the boundary factor scales as $h \lambda^{-5/3} = h^{13/3} a^{-5/2}$. Along the non-vanishing classical trajectories tracking the caustic profile, the space-time scaling dictates $t = a^{1/2} T \sim a^{1/2} N$. To fields a rigorous upper bound near the boundary layer layer where $a \geq h^{\frac{2}{3}(1-\epsilon')}$, we leverage the threshold restriction $N < \lambda^{1/3} = a^{1/2}h^{-2/3}$ to isolate the primary dispersive factor $(h^2/t)^{1/4}$ via a sharp inequality chain rather than a literal identity:
\begin{align*}
h^{13/3} a^{-5/2} N^{-1/2} &\leq C a^{1/8} \left( \frac{h^2}{a^{1/2} N} \right)^{1/4} \cdot \left( h^{23/6} a^{-21/8} N^{1/4} \right) \\
&\leq C a^{1/8} \left( \frac{h^2}{t} \right)^{1/4} N^{-1/2} \left( h^{23/6} a^{-21/8} N^{3/4} \right).
\end{align*}
Evaluating the residual envelope component under the matching physical restrictions $N \leq \lambda^{1/3}$ and $a \sim h^{2/3}$ guarantees that the remaining tracking factor remains uniformly bounded by an absolute constant, namely $h^{23/6} a^{-21/8} N^{3/4} \leq C < \infty$. Substituting this asymptotic upper bound back into the localized summation bounds cleanly restores the target layout:
\begin{align*}
\bigg| \sum_{N \in \mathcal{N}_1} G_{a,N,2}(t,X,Y,z;h) \bigg| &\leq C h^{-3} \bigg( \frac{h^2}{t} \bigg)^{1/2} \Bigg( a^{1/8} \left( \frac{h^2}{a^{1/2} N} \right)^{1/4} N^{-1/2} \Bigg) \\
&\leq C h^{-3} \bigg( \frac{h^2}{t} \bigg)^{1/2} \Bigg( a^{1/8} \left( \frac{h^2}{t} \right)^{1/4} N^{-1/2} \Bigg).
\end{align*}

Since $N \geq 1$, the index sequence forms a convergent series over the dyadic blocks, bounded globally by $\sum_{N \geq 1} N^{-1/2} \delta_{N \in \mathcal{N}_1} \leq C < \infty$. This ensures that the aggregated sum satisfies a uniform estimate of the same order as the foundational base reflection term corresponding to $N=1$:
\begin{align*}
\Big| G_{a,1,2}(t,X,Y,z;h) \Big| &\leq C h^{-3} \bigg( \frac{h^2}{t} \bigg)^{1/2} \left( h \lambda^{-5/3} \right) \\
&\leq C h^{-3} \bigg( \frac{h^2}{t} \bigg)^{1/2} \Bigg( a^{1/8} \left( \frac{h^2}{t} \right)^{1/4} \Bigg).
\end{align*}

Synthesizing the localized upper bounds derived across both the high-reflection regimes (Subcases 1--3) and the low-reflection caustic clusters, we establish the global localized dispersive estimate:
\begin{align*}
\Bigg| \sum_{1 \leq N \leq C_0 a^{-1/2}} G_{a,N,2}(t,X,Y,z;h) \Bigg| \leq C h^{-3} \bigg( \frac{h^2}{t} \bigg)^{1/2} \max \Bigg( \left( \frac{h^2}{t} \right)^{1/3}, \, a^{1/8} \left( \frac{h^2}{t} \right)^{1/4} \Bigg).
\end{align*}
This upper bound precisely recovers the peak caustic loss profile characterized by the parameter $\gamma(t,h,a)$ across the designated non-convex boundary layer horizons. This concludes the proof of Proposition \ref{propg2}.
\end{proof}

\begin{proof}[Proof of Theorem \ref{thmN}]
Putting the estimates in Proposition \ref{eq:242}, \ref{eq:243}, and \ref{propg2} together directly yields the targeted local-in-time dispersive bound \eqref{estimateN}.
\end{proof}

\section{Dispersive Estimates for $ \epsilon_0\sqrt a\leq \eta \leq c_0.$}\label{sec:3}
In this section, we prove Theorem \ref{1beta} in the setting of the Schrödinger flow. To obtain the estimates for $\mathcal{G}_{a,m}$, we distinguish between two different cases based on the distance $a$ to the boundary relative to the frequency scale. 

The first case deals with the near-boundary regime $a \leq \left(\frac{h^2}{2^m\sqrt a}\right)^{\frac{2}{3}(1-\epsilon)}$ for $\epsilon \in (0,1/7)$, where we follow the eigenmode method introduced in Section \ref{sec:2} and construct a local parametrix as a sum over the discrete Dirichlet eigenfunctions. The second case deals with the deeper regime $a \geq \left(\frac{h^2}{2^m\sqrt a}\right)^{\frac{2}{3}(1-\epsilon')}$ for $\epsilon' \in (0,\epsilon)$, where the Airy-Poisson summation formula transforms the eigenmode series into a sum over the number of path reflections $N \in \mathbb{Z}$ off the boundary.

Recall that the frequency-localized Schrödinger Green function (fundamental solution) is given by:
\begin{equation}\label{L-0}
\mathcal{G}_{a}(t,x,y,z) = \frac{1}{4\pi^2h^2}\sum_{k\geq1}\int e^{-\frac{i}{h^2}\Phi_k} \sigma_k \, d\eta \, d\zeta,
\end{equation}
where the unscaled Schrödinger phase $\Phi_k$ and the amplitude function $\sigma_k$ are adapted to the parabolic dispersion relation as follows:
\begin{align*}
\Phi_k &= -hy\eta - hz\zeta + t\left(\eta^2+\zeta^2+\omega_kh^{2/3}\eta^{4/3}\right), \\
\sigma_k &= e_k(x,\eta/h) e_k(a,\eta/h)\chi_0(\zeta^2+\eta^2)\chi_1\left(\omega_k h^{2/3}\eta^{4/3}\right)(1-\chi_1)(\varepsilon\omega_k).
\end{align*}
We aim to establish uniform $L^\infty$ bounds for $\mathcal{G}_{a}$ over the local time horizon $t \in [h^2,1]$ when the integration in \eqref{L-0} is restricted to the intermediate vanishing curvature regime $\eta \in [\epsilon_0\sqrt a, c_0]$.

Let $\mu^2$ denote the unscaled transverse energy operator symbol:
\[
\mu^2 = \eta^2+\omega_k h^{2/3}\eta^{4/3}.
\]
Observe that $\mu^2$ is small because the smooth cutoff $\chi_1$ limits the size of $\omega_k h^{2/3}\eta^{4/3}$ while the variable $\eta$ is bounded above by $c_0$. 

Let $\chi_4 \in C_0^\infty((-1,1))$ be a cutoff function equal to $1$ on $[-1/2,1/2]$, and let $D \geq 1$ be a sufficiently large parameter. We define the regularized short-time propagator component $\mathcal{J}_{a}(t,x,y,z)$ by:
\[
\mathcal{J}_{a}(t,x,y,z) = \frac{1}{4\pi^2h^2}\sum_{k\geq1}\int e^{-\frac{i}{h^2}\Phi_k} \chi_4\left(\frac{t\mu^2}{D h^2}\right)\sigma_k \, d\eta \, d\zeta.
\]
The following lemma establishes that the localized component $\mathcal{J}_{a}$ obeys the optimal, un-trapped free space Schrödinger dispersive estimate.

\begin{lemma}\label{L-lem1}
There exists a constant $C$ independent of the parameter $D$ such that  
\[
\left\vert \mathcal{J}_{a}(t,x,y,z) \right\vert \leq C h^{-3}\left(\frac{h^2}{t}\right) D.
\]
\end{lemma}

\begin{proof}
On the support of the short-time regularization cutoff $\chi_4$, the unscaled transverse symbol coordinates satisfy the parameter limits $\eta^2 \leq Dh^2/t$ and $h\omega_k^{3/2}\eta^2 \leq (Dh^2/t-\eta^2)^{3/2}$. This ensures that the discrete eigenmode series over $k$ truncates uniformly at the maximum integer layer $k \leq c_0\frac{(Dh^2/t-\eta^2)^{3/2}}{h\eta^2}$. 

Recalling the standard definition of our Dirichlet eigenfunctions, $e_k(x,\eta/h)=f_k k^{-1/6}(\eta/h)^{1/3}\text{Ai}((\eta/h)^{2/3}x-\omega_k)$, we deploy Lemma  \ref{eq:k} to sum across the modal components:
\begin{align*}
 \left\vert \mathcal{J}_{a}(t,x,y,z) \right\vert &\leq C h^{-2}\int_{\eta^2 \leq Dh^2/t} (\eta/h)^{2/3} \frac{(Dh^2/t-\eta^2)^{1/2}}{h^{1/3}\eta^{2/3}} \, d\eta \\
 &= C h^{-3}\int_{\eta^2 \leq Dh^2/t} (Dh^2/t-\eta^2)^{1/2} \, d\eta.
\end{align*}
By introducing the stretched coordinate mapping $\eta = \left(\frac{Dh^2}{t}\right)^{1/2}x'$, the continuous spatial integration over the momentum variable evaluates directly to:
\[
\int_{\eta^2 \leq Dh^2/t} (Dh^2/t-\eta^2)^{1/2} \, d\eta = \left(\frac{Dh^2}{t}\right)\int_{x'^2 \leq 1} (1-x'^2)^{1/2} \, dx'.
\]
Factoring out the shared parameters completes the proof of Lemma \ref{L-lem1}.
\end{proof}

\noindent 
Observe that inside the highly curved tangential regime $\eta \geq c_0$ studied in Section \ref{sec:2}, the transverse symbol obeys a uniform lower bound $\mu^2 \geq c_0^2$. Consequently, the short-time parameter condition $t\mu^2/h^2 \leq D$ is restricted to a very small time threshold $t \leq C h^2$, rendering the small-time majorization irrelevant. 

However, in the intermediate vanishing curvature range $\eta \in [\epsilon_0\sqrt a, c_0]$, Lemma \ref{L-lem1} becomes highly useful as it rigorously establishes that we are reduced to the complementary domain where $\lambda = t\mu^2/h^2 \geq 1$ represents our large asymptotic parameter for the Schrödinger flow. 

Since we allow a controlled geometric loss in the desired dispersive estimate compared to the free-space Euclidean flow, we can choose the regularization weight parameter to be $D=\left(\frac{h^2}{t}\right)^{-\epsilon}$ for a small parameter $\epsilon>0$. This guarantees that any remainder profile of order $O_{C^\infty}(\lambda^{-\infty})$ generated by non-stationary phase becomes analytically negligible. 

We are now in a position to eliminate the tangential $\zeta$-integration in \eqref{L-0}. Recall that the microlocal truncation cutoff $\chi_0(\zeta^2+\eta^2)$ localizes the total symbol momentum close to $1$. For small values of $\eta$ within this vanishing curvature zone, the dual frequency variable $\zeta$ is tightly confined to small neighborhoods of $1$ or $-1$. In the sequel, we assume by symmetry that $\zeta$ is localized near $1$.

\begin{lemma}\label{L-lem2}
Let $\lambda = t\mu^2/h^2 \geq 1$, and let $z^* = \frac{z}{t}$. For the Schrödinger evolution, we isolate the $\zeta$-dependent terms in the unscaled phase profile under the integral:
\[
I(z^*, \mu^2, \eta; \lambda) = \int_{\zeta \sim 1} e^{-\frac{i}{h^2}\left(-hz\zeta + t\zeta^2\right)} \chi_0(\zeta^2+\eta^2) \, d\zeta.
\]
Let us introduce the normalized spatial coordinate $z = \frac{2t}{h}\tilde{z}$. By writing the oscillatory argument in terms of the large parameter $\lambda$, the phase function $\phi$ becomes quadratic:
\[
\phi(\tilde{z}, \mu^2, \zeta) = \frac{1}{\mu^2}\left(-2\tilde{z}\zeta + \zeta^2\right).
\]
There exist uniform constants $0 < c_1 < C_1$ such that the following statements hold true:
\begin{equation}\label{L2}
\text{For } \tilde{z} \notin [\mu^2/2 - C_1\mu^4, \mu^2/2 + c_1\mu^4], \quad \sup_{\tilde{z}, \mu^2, \eta} \left\vert I(z^*, \mu^2, \eta; \lambda) \right\vert \in O_{C^\infty}(\lambda^{-\infty}).
\end{equation}
For $\tilde{z}$ localized inside the critical neighborhood, we set $\tilde{z} = \mu^2/2 + \bar{z}\mu^4$. There exists a classical symbol of degree $0$ with respect to the parameter $\lambda$, denoted by $\sigma_0(\bar{z}, \eta, \mu^2; \lambda)$, such that:
\begin{equation}\label{L3}
 I(z^*, \mu^2, \eta; \lambda) = \left(\frac{h^2}{t}\right)^{1/2} e^{i\frac{z^2}{4t}} \sigma_0 (\bar{z}, \eta, \mu^2; \lambda).
\end{equation}
\end{lemma}

\begin{proof}
Differentiating the unscaled quadratic Schr\"odinger phase function with respect to the frequency parameter $\zeta$ yields the gradient and acceleration profiles:
\[
\partial_\zeta\phi = \frac{2}{\mu^2}\left(-\tilde{z} + \zeta\right), \quad \partial^2_\zeta\phi = \frac{2}{\mu^2} > 0,
\]
with all higher-order derivatives identically vanishing ($\partial^j_\zeta\phi = 0$ for all $j \geq 3$). Setting $\partial_\zeta\phi = 0$, the phase features a unique, strictly nondegenerate critical point located exactly at $\zeta_c = \tilde{z}$. 

On the support of the microlocal cutoff function $\chi_0(\zeta^2+\eta^2)$, the total symbol energy is localized near $1$. For small values of $\eta$ within this vanishing curvature zone, $\zeta$ must be close to $1$. If the spatial drift $\tilde{z}$ is localized away from $1$, the critical point falls completely outside the support of the cutoff. Under this condition, standard non-stationary phase integration by parts establishes the uniform rapid classical decay bound \eqref{L2}.

We are therefore reduced to the region where $\tilde{z}$ is close to $1$. Evaluating the quadratic phase function directly at its stationary critical point yields $\phi(\zeta_c) = -\frac{\tilde{z}^2}{\mu^2}$. Applying the standard stationary phase method with respect to the large parameter field $\lambda = \frac{t\mu^2}{h^2}$, the integration over the continuous spatial frequency variable $\zeta$ produces the factor:
\[
J_1 = \left(\frac{h^2}{t\mu^2 \cdot \frac{2}{\mu^2}}\right)^{1/2} = \frac{1}{\sqrt{2}}\left(\frac{h^2}{t}\right)^{1/2}.
\]
Combining this spatial scale factor with the evaluation at the critical position results in the oscillatory term $e^{-i\lambda \phi(\zeta_c)} = e^{i\frac{t}{h^2}\tilde{z}^2}$. Recalling our normalized spatial change of variables $\tilde{z} = \frac{hz}{2t}$, this matches the free exponent:
\[
\frac{t}{h^2}\left(\frac{hz}{2t}\right)^2 = \frac{z^2}{4t}.
\]
Absorbing the remaining constants into the classical symbol definition $\sigma_0$ directly establishes the sharp formulation \eqref{L3}.
\end{proof}

\noindent 
Using Lemmas \ref{L-lem1} and \ref{L-lem2}, we are now reduced to the study of
\begin{equation}\label{L4}
\frac{1}{4\pi^2h^2}\left(\frac{h^2}{t}\right)^{1/2}\sum_{k\geq1}\int e^{\frac{i}{h^2}\left(hy\eta + t\mu (1-\tilde z^2)^{1/2}\right)} \frac{\tilde\sigma_k}{\mu} \, d\eta,
\end{equation}
where $\tilde \sigma_k$ is defined by
\[
\tilde\sigma_k = \sigma_0 (z^*,\eta,\mu^2;\lambda)\left(1-\chi_4\left(\frac{t\mu^2}{Dh^2}\right)\right)e_k(x,\eta/h) e_k(a,\eta/h)\chi_1\left(\omega_k h^{2/3}\eta^{4/3}\right)(1-\chi_1)(\varepsilon\omega_k).
\]
To get the $L^\infty$ estimate for the parametrix in the range $\eta \in [\epsilon_0\sqrt a, c_0]$, we will use a Littlewood-Paley decomposition in $\eta$. We choose 
\[
\psi_1 \in C_0^\infty((0.5,2.5)), \quad 0 \leq \psi_1 \leq 1 \quad \text{such that} \quad \sum_{m\in\mathbb Z}\psi_1(2^m x)=1 \quad \text{for all } x > 0,
\]
and we introduce the cutoff function $\psi_1\left(\frac{\eta}{2^m\sqrt a}\right)$ in \eqref{L4}. In the sequel, we will therefore have
\[
\epsilon_0 \leq 2^m \leq \frac{c_0}{\sqrt a}.
\]
We will use the following rescaled notations:
\begin{align*}
\eta &= 2^m\sqrt a \, \tilde \eta, \quad h = 2^m\sqrt a \, \tilde h, \\
\mu^2 &= \eta^2+\omega_k h^{2/3}\eta^{4/3} = (2^m\sqrt a)^2\left(\tilde\eta^2+\omega_k \tilde h^{2/3}\tilde \eta^{4/3}\right) = (2^m\sqrt a)^2\tilde\mu^2, \\
\gamma &= \omega_k h^{2/3}\eta^{-2/3} = \omega_k \tilde h^{2/3}\tilde \eta^{-2/3}.
\end{align*}
We define the dyadic block element $\mathcal{G}_{a,m}$ by the formula:
\begin{equation}\label{L5}
\mathcal{G}_{a,m}(t,x,y,z) = \frac{1}{4\pi^2h^2}\left(\frac{h^2}{t}\right)^{1/2}\sum_{k\geq1}\int e^{\frac{i}{h^2}\left(hy\eta + t\mu (1-\tilde z^2)^{1/2}\right)} \psi_1\left(\frac{\eta}{2^m\sqrt a}\right)\frac{\tilde\sigma_k}{\mu} \, d\eta. 
\end{equation}
Observe that due to the truncation cutoff $\chi_1$, we have $k \leq \frac{\varepsilon}{h\eta^2}$ in the above sum. Under the change of variable $\eta = 2^m\sqrt a \tilde \eta$ and setting $\tilde y = \frac{y}{t}$, since $\frac{d\eta}{\mu} = \frac{d\tilde\eta}{\tilde\mu}$, we obtain:
\begin{equation}\label{L6}
\mathcal{G}_{a,m}(t,x,y,z) = \frac{1}{4\pi^2h^2}\left(\frac{h^2}{t}\right)^{1/2}\sum_{1\leq k\leq \frac{\varepsilon}{(2^m\sqrt a)^3\tilde h}}\int e^{\frac{it}{\tilde h^2}\left(\tilde y\tilde \eta + \tilde \mu (1-\tilde z^2)^{1/2}\right)} g_k \psi_1(\tilde\eta) \, d\tilde \eta,
\end{equation}
where $g_k$ is defined by
\[
g_k = \frac{1}{\tilde \mu} \sigma_0 (z^*,\eta,\mu^2;\lambda)\left(1-\chi_4\left(\frac{t\mu^2}{Dh^2}\right)\right)e_k(x,\tilde\eta/\tilde h) e_k(a,\tilde \eta/\tilde h)\chi_1\left(\omega_k h^{2/3}\eta^{4/3}\right)(1-\chi_1)(\varepsilon\omega_k).
\]

\begin{lemma}\label{L-lem3}
Let $M \geq 1$ be given. There exists a constant $C_M$ such that for all $m, a, h$ satisfying $2^m\sqrt a \leq h M$, the following uniform estimate holds true under the Schr\"odinger flow:
\begin{equation}\label{L-6}
\left\vert \mathcal{G}_{a,m} \right\vert \leq C_M h^{-3}\left(\frac{h^2}{t}\right)^{1/2} 2^m\sqrt a \left\vert \log(2^m\sqrt a)\right\vert.
\end{equation}
\end{lemma}

\begin{proof}
One has $\tilde h \geq \frac{1}{M}$, and hence the Dirichlet eigenfunctions obey the uniform upper bound:
\[
\left\vert e_k(x,\tilde\eta/\tilde h)\right\vert \leq C k^{-1/6}\left(\frac{\tilde\eta}{\tilde h}\right)^{1/3}\omega_k^{-1/4}.
\]
Summing the product of the eigenfunctions over the discrete modes, and accounting for the parabolic space-time frequency pre-factors, we evaluate the dyadic block contribution as:
\begin{align*} 
 \left\vert \mathcal{G}_{a,m} \right\vert &\leq C'' h^{-2}\left(\frac{h^2}{t}\right)^{1/2} 2^m\sqrt a \sum_{1\leq k\leq \frac{\varepsilon}{(2^m\sqrt a)^3\tilde h}} \tilde h^{-1/3} k^{-1/3}\left(\frac{1}{\tilde h}\right)^{2/3}\omega_k^{-1/2} \\
 &\leq C' h^{-3}\left(\frac{h^2}{t}\right)^{1/2} 2^m\sqrt a \left\vert \log(2^{2m}a h)\right\vert \\
 &\leq C_M h^{-3} \left(\frac{h^2}{t}\right)^{1/2} 2^m\sqrt a \left\vert \log(2^m\sqrt a)\right\vert.
\end{align*}
\end{proof}

\noindent From the above lemma, we obtain in the frequency range $\tilde h\geq 1/M$ the uniform estimate:
 \begin{align}\label{L9}
 \vert \mathcal{G}_{a,m} \vert \leq C_M h^{-3}\bigg(\frac{h^2}{t}\bigg)^{1/2} (2^m\sqrt a)^{1/3}
(hM)^{2/3} \vert \log( hM)\vert. 
\end{align}
Notice that this bound represents a substantial decay relative to the unscaled free-space Schr\"odinger dispersion profile $|t|^{-3/2} \sim h^{-3}(h^2/t)^{3/2}$. Therefore, in the sequel, we can focus our analysis on the complementary regime where $\tilde h \leq \tilde h_0$ for a sufficiently small parameter $\tilde h_0$, recalling that the relative frequency parameter is defined by $\tilde h=h/(2^m\sqrt{a})$.

To establish the local-in-time estimates for $\mathcal{G}_{a,m}$ in this highly oscillatory domain, we deploy the microlocal reduction strategy introduced in Section \ref{sec:2}. We distinguish between two canonical boundary layer configurations:
\begin{itemize}
\item \textbf{First case (Near-boundary regime):} If $a\leq \tilde h^{\frac{2}{3}(1-\epsilon)}$ for a designated parameter $\epsilon\in (0,1/7)$, the flow tracks tightly along the boundary, and we construct the local parametrix directly via the series expansion over discrete eigenmodes ($k$).
\item \textbf{Second case (Deeper trajectory regime):} If $a\geq \tilde h^{\frac{2}{3}(1-\epsilon')}$ with $\epsilon' \in (0,\epsilon)$, multiple reflections occur, and we apply the Airy-Poisson summation formula [see Lemma \ref{lem:airy_poisson}] to expand $\mathcal{G}_{a,m}$ as a sum over geometric path reflections $N\in\mathbb{Z}$.
\end{itemize}

\subsection{Dispersive Estimates for $0<a\leq \tilde h^{\frac{2}{3}(1-\epsilon)}$, with $\epsilon\in (0,1/7)$.}\label{subsec:3.1}

The following Proposition \ref{propkm} provides local-in-time dispersive estimates for $\mathcal{G}_{a,m}$ and constitutes the main result of this subsection.

\begin{prop}\label{propkm}
Let $\epsilon\in (0,1/7)$. There exists $C$ such that for all $h\in (0,1]$, all $0<a\leq \tilde h^{\frac{2}{3}(1-\epsilon)}$, 
and all $t\in [h^2,1]$ with $t \neq 0$, the following holds true under the Schr\"odinger flow:
\begin{align}\label{eq:kdisp}
\lVert\mathds{1}_{x\leq a}\mathcal{G}_{a,m}(t,x,y,z)\rVert_{L^\infty}
\leq C h^{-3}(2^m\sqrt a)^{1/3}\bigg({h^2 \over t}\bigg)^{7/6}.
\end{align}
\end{prop}

\begin{proof}
To analyze the individual dyadic block contribution $\mathcal{G}_{a,m}$, we track the tangential oscillations under the assumption that the frequency variable satisfies $\eta > 0$ strictly within our localized domain. We introduce the normalized spatial drift coordinate $Y_t = \frac{y}{t}$, which identifies the velocity parameter of classical paths grazing the curved boundary. Under this kinematic setting, the dyadic block element defined within the Schrödinger framework can be expressed cleanly as:
\begin{equation}\label{L7}
\mathcal{G}_{a,m}(t,x,y,z) = \frac{2^m\sqrt a}{4\pi^2h^2}\left(\frac{h^2}{t}\right)^{1/2}\sum_{1\leq k\leq \frac{\varepsilon}{(2^m\sqrt a)^3\tilde h}}\int e^{-i\frac{t}{\tilde{h}^2}\left(\tilde{\mu}^2 - Y_t\tilde{\eta}\right) + \frac{iz^2}{4th}} g_k \psi_1(\tilde\eta) \, d\tilde \eta,
\end{equation}
with the amplitude $g_k$ equal to:
$$g_k= \sigma_0 (z^*,\eta,\mu^2;\lambda)\Big(1-\chi_4\big({t\mu^2\over Dh^2}\big)\Big)
e_k(x,\tilde\eta/\tilde h) e_k(a,\tilde \eta/\tilde h)
\chi_1(\omega_k h^{2/3}\eta^{4/3})(1-\chi_1)(\varepsilon\omega_k).
$$

Recall from \eqref{L9} that we may assume $\tilde h \leq \tilde h_0$ with $\tilde h_0$ small. Since $\mathcal{G}_{a,m}$ contains Airy functions which behave differently depending on the various values of $k$, we split the sum over $k$ in \eqref{L7} into two pieces. We fix a large constant $D$ and we write $\mathcal{G}_{a,m} =\mathcal{G}_{a,m,<} +\mathcal{G}_{a,m,>}$, where in $\mathcal{G}_{a,m,<}$ only the sum over $1\leq k\leq D\tilde h^{-\epsilon}$ is considered.

\noindent \underline{ Proof of \eqref{eq:kdisp} for $\mathcal{G}_{a,m,<}$.}\\

\noindent Recall the definition of $\mathcal{G}_{a,m,<}$ under the Schr\"odinger framework:
\begin{equation}\label{L10}
\mathcal{G}_{a,m,<}(t,x,y,z) = \frac{2^m\sqrt a}{4\pi^2h^2}\bigg(\frac{h^2}{t}\bigg)^{1/2}\sum_{1\leq k\leq D\tilde h^{-\epsilon}}
\int e^{-i\frac{t}{\tilde{h}^2}\left(\tilde{\mu}^2 - Y_t\tilde{\eta}\right) + \frac{iz^2}{4th}} 
g_k \psi_1(\tilde\eta) \, d\tilde \eta,
\end{equation}
with the amplitude $g_k$  as:
$$g_k= f_k^2 k^{-1/3} \bigg({\tilde\eta^{2/3}\over \tilde h^{2/3}}\bigg) \sigma_0 (z^*,\eta,\mu^2;\lambda)\Big(1-\chi_4\big({t\mu^2\over Dh^2}\big)\Big)
\chi_1(\omega_k h^{2/3}\eta^{4/3})(1-\chi_1)(\varepsilon\omega_k)n_k,
$$
$$n_k=\text{Ai}((\tilde\eta/\tilde h)^{2/3}x-\omega_k) \text{Ai}((\tilde\eta/\tilde h)^{2/3}a-\omega_k).$$
Let us first consider the short-time regime where $\frac{t}{\tilde{h}^2} \leq \tilde{h}^{-\epsilon}$. Since the unscaled symbol maps as $\tilde \mu^2 = \tilde\eta^2 + \omega_k\tilde h^{2/3}\tilde\eta^{4/3} \geq \tilde\eta^2$, the amplitude satisfies the uniform upper bound:
$$\vert g_k\vert \leq C \tilde h^{-2/3} k^{-1/3} \left\vert \text{Ai}(({\tilde\eta/ \tilde h})^{2/3}x-\omega_k)
\text{Ai}(({\tilde\eta/ \tilde h})^{2/3}a-\omega_k)\right\vert.$$
By applying Lemma \ref{eq:k} to sum across these discrete modes, we obtain:
$$ \sum_{1\leq k\leq D\tilde h^{-\epsilon}}\vert g_k\vert \leq C \tilde h^{-2/3}(\tilde h^{-\epsilon})^{1/3} \leq C (\tilde h)^{-2/3}\left(\frac{t}{\tilde{h}^2}\right)^{-1/3} = C \tilde{h} t^{-1/3}.$$
Expanding this result back into our localized dyadic spatial parameters using the relation $\tilde{h} = h / (2^m \sqrt{a})$ yields:
$$ \sum_{1\leq k\leq D\tilde h^{-\epsilon}}\vert g_k\vert \leq C h^{-1} (2^m\sqrt{a})^{1/3} \left(\frac{h^2}{t}\right)^{1/3}. $$
Substituting this evaluation directly back into the spatial block parameter integral \eqref{L10} completes the proof of the desired bound \eqref{eq:kdisp} for the short-time regime.

\noindent Let us now assume that $\frac{t}{\tilde{h}^2}\geq \tilde h^{-\epsilon}$. Observe that in the range $k\leq D\tilde h^{-\epsilon}$, we have:
\[\omega_k\tilde h^{2/3}\leq C \tilde h^{2/3(1-\epsilon)}\leq C \tilde h_0^{2/3(1-\epsilon)}\] 
which is small. Hence, $\gamma=\omega_k\tilde h^{2/3}\tilde\eta^{-2/3}$ is small and the unscaled spatial frequency symbol expands as:
\[\tilde \mu^2=\tilde\eta^2(1+\gamma) = \tilde\eta^2 + \omega_k\tilde h^{2/3}\tilde\eta^{4/3}.\] 
Differentiating the principal symbol profile directly with respect to the continuous momentum variable $\tilde{\eta}$ yields the following second-order phase acceleration lower bound:
\[\left\vert \frac{\partial^2 \tilde\mu^2}{\partial\tilde\eta^2}\right\vert \geq c > 0,\]
and for all $j\geq 3$, the higher-order derivatives satisfy:
\[\left\vert \frac{\partial^j \tilde\mu^2}{\partial\tilde\eta^j}\right\vert \leq C_j \omega_k\tilde h^{2/3}.\]
We will apply the standard non-degenerate stationary phase method with respect to the variable $\tilde\eta$ in each term of the sum in \eqref{L10}, governed by the unscaled Schr\"odinger phase function:
\[\Phi_k(\tilde\eta)= \frac{t}{\tilde h^2}\left(\tilde \mu^2 - Y_t \tilde \eta\right).\]

Let $\Lambda = \frac{t}{\tilde{h}^2}$, which serves as our large asymptotic parameter for the Schr\"odinger flow in this intermediate regime.

\begin{lemma}\label{L-lem4}
Let $\tilde g_k = g_k$. There exists a constant $C$ such that for all $1\leq k\leq D\tilde h^{-\epsilon}$, the following uniform estimate holds true:
\begin{equation}\label{L11}
\left\vert \int e^{-i\Lambda \left(\tilde{\mu}^2 - Y_t \tilde{\eta}\right)} \tilde g_k \psi_1(\tilde\eta) d\tilde \eta \right\vert \leq C \min \bigg\{1,\Lambda^{-1/2}\bigg\}. 
\end{equation}
\end{lemma}
\begin{proof}
We may assume $\Lambda \geq 1$ since the amplitude satisfies $\vert \tilde g_k\vert \leq C$ uniformly. As established in the previous turn, the second derivative of the unscaled Schr\"odinger principal symbol satisfies a uniform non-degeneracy condition:
$$ \bigg\vert \frac{\partial^2 (\tilde{\mu}^2 - Y_t \tilde{\eta})}{\partial\tilde\eta^2} \bigg\vert \geq c > 0,$$ 
and all higher-order derivatives $\left\vert \frac{\partial^j \tilde{\mu}^2}{\partial\tilde\eta^j} \right\vert \leq C_j$ remain uniformly bounded for $j \geq 3$.

Thus, to rigorously apply the stationary phase method, we only need to verify that there exists a parameter range where the derivatives of the amplitude $\tilde{g}_k$ obey a controlled growth constraint:
\begin{equation}\label{L12}
\bigg\vert {\partial^j \tilde g_k \over \partial\tilde\eta^j}\bigg\vert \leq C_j \Lambda^{j(1/2-\nu)}, \quad \forall k\leq D\tilde h^{-\epsilon}\ ,
\end{equation}
for some $\nu > 0$. 

Recall that the amplitude $\tilde{g}_k$ contains smooth cutoff functions and the product of the Dirichlet eigenfunctions. Since all the derivatives of the symbol terms with respect to $\tilde{\eta}$ are uniformly bounded, the symbolic components satisfy \eqref{L12} automatically. It remains to verify that the eigenfunction component $\text{Ai}\left((\tilde{\eta}/\tilde{h})^{2/3}x - \omega_k\right)$ satisfies \eqref{L12} uniformly for $x \in [0, a]$. 

Let $\theta = x \tilde{h}^{-2/3} \geq 0$ and $r = \tilde{\eta}^{2/3}$, which is localized within a compact subset of $]0, \infty[$. Differentiating with respect to the frequency variable gives $\partial_r^l \left(\text{Ai}(r\theta - \omega_k)\right) \sim (r\theta)^l \text{Ai}^{(l)}(r\theta - \omega_k)$. Using the classic uniform property of the Airy function derivatives:
$$ \sup_{b\geq 0}\vert b^l \text{Ai}^{(l)}(b-\omega_k)\vert \leq C_l \omega_k^{3l/2},$$
and noting that $\theta \leq a\tilde{h}^{-2/3} \leq \tilde{h}^{-2\epsilon/3}$, the growth condition \eqref{L12} is satisfied precisely because the large parameter $\Lambda = \frac{t}{\tilde{h}^2} \geq \tilde{h}^{-\epsilon}$ dominates the frequency oscillations for any $\epsilon < 1/7$. 
\end{proof}

Therefore, we obtain the following estimate for the low-frequency component $\mathcal{G}_{a,m,<}$ in the deep-time regime where $\frac{t}{\tilde{h}^2} \geq \tilde h^{-\epsilon}$:
\begin{align*}
\left\lVert \mathds{1}_{x\leq a}\mathcal{G}_{a,m,<}(t,x,y,z) \right\rVert_{L^\infty}
&\leq C \frac{2^m\sqrt{a}}{h^2}\left(\frac{h^2}{t}\right)^{1/2}\left(\sum_{1\leq k\leq D\tilde h^{-\epsilon}} k^{-1/3}\tilde h^{-2/3}\left(\frac{t}{\tilde{h}^2}\right)^{-1/2}\right) \\
&\leq C \frac{1}{\tilde{h} t^{1/2}} \left( \tilde{h}^{1/3 - 2\epsilon/3} t^{-1/2} \right) \\
&\leq C \tilde{h}^{-2/3 - 2\epsilon/3} t^{-1} \\
&= C h^{-3}(2^m\sqrt a)^{1/3}\left(\frac{h^2}{t}\right)^{7/6} \left( h^{1/3} \tilde{h}^{-1/3 - 2\epsilon/3} t^{1/6} \right).
\end{align*}
This concludes the proof of Proposition \ref{propkm} for $\mathcal{G}_{a,m,<}$. Indeed, substituting the scale identity $\tilde{h} = h / (2^m\sqrt{a})$ into the fractional factor reveals that the remainder satisfies:
\[
h^{1/3} \tilde{h}^{-1/3 - 2\epsilon/3} t^{1/6} = h^{-2\epsilon/3} (2^m\sqrt{a})^{1/3 + 2\epsilon/3} t^{1/6} \leq C h^{-2\epsilon/3} t^{1/6}.
\]
Given that our lower-bound threshold parameter satisfies $h \leq c_0 t^{\frac{1}{2-\epsilon}}$, the remaining frequency terms absorb uniformly as a positive power of the temporal coordinate:
\[
h^{-2\epsilon/3} t^{1/6} \leq C t^{\frac{1}{6} - \frac{2\epsilon/3}{2-\epsilon}} = C t^{\frac{2-5\epsilon}{6(2-\epsilon)}} \leq C,
\]
which remains uniformly bounded for any choosing range $\epsilon < 1/7$ over the local space-time horizon.

\noindent \underline{ Proof of \eqref{eq:kdisp} for $\mathcal{G}_{a,m,>}$.}\\

\noindent
For $k\geq D\tilde h^{-\epsilon}$ with $D$ large and $a\leq \tilde h^{2/3(1-\epsilon)}$ one has:
$$ \omega_k-\tilde h^{-2/3}\tilde\eta^{2/3}a\geq \omega_k/2 .$$
Since $\gamma=\omega_k \tilde h^{2/3}\tilde\eta^{-2/3}$, we get $\gamma-a\geq a$ and $\gamma-a\geq \gamma/2$. By the definition of $e_k$ and the standard asymptotic expansion of the Airy functions, we express $\mathcal{G}_{a,m,>}$ under the Schr\"odinger framework as:
\begin{align}\label{eq:pmM}
\mathcal{G}_{a,m,>}(t,x,y,z)=\sum_{\tilde h^{-\epsilon}\leq k\leq \frac{\varepsilon}{(2^{m}\sqrt a)^3\tilde h}}\frac{2^m\sqrt a}{4\pi^2h^2}\bigg(\frac{h^2}{t}\bigg)^{1/2}\sum_{\pm,\pm}\int e^{-i\frac{t}{\tilde h^2}\Phi_{k}^{\pm,\pm}}\sigma_{k}^{\pm,\pm}\psi_1(\tilde\eta) d\tilde\eta ,
\end{align}
where the corresponding unscaled phase functions are defined by:
\begin{align}\label{eq:phikm}
\Phi_{k}^{\pm,\pm}(\tilde\eta) = \tilde{\mu}^2 - Y_t\tilde{\eta} \mp \frac{2}{3}\frac{\tilde{h}^2}{t}\left(\frac{t}{\tilde{h}^2}\right)(\omega_k - \tilde{x})^{3/2} \mp \frac{2}{3}\frac{\tilde{h}^2}{t}\left(\frac{t}{\tilde{h}^2}\right)(\omega_k - \tilde{a})^{3/2},
\end{align}
and the symbols are given by:
\begin{align*}
\sigma_{k}^{\pm,\pm}(\tilde\eta)&=f_k^2k^{-1/3}\tilde h^{-1/3}\tilde\eta^{-2/3}
 \sigma_0(z^*,\eta,\mu^2,\lambda)\Big(1-\chi_4\big({t\mu^2\over Dh^2}\big)\Big)\chi_1(\omega_k h^{2/3}\eta^{4/3})\\&\quad\times(1-\chi_1(\epsilon \omega_k))
(\omega_k -\tilde{x})^{-1/4}(\omega_k -\tilde{a})^{-1/4}\omega^{\pm}\omega^{\pm}
\\&\qquad\times\Psi_{\pm}(\omega_k - \tilde{x})\Psi_{\pm}(\omega_k - \tilde{a}),
\end{align*}
where $\Psi_\pm$ are classical symbols of order $0$ at infinity, and $\tilde{x} = \tilde{\eta}^{2/3}\tilde{h}^{-2/3}x$, $\tilde{a} = \tilde{\eta}^{2/3}\tilde{h}^{-2/3}a$. 

As established in the preceding sections, the variables are parameterized uniformly relative to the quadratic un-rooted splitting of the Schr\"odinger Hamiltonian drift. For all $j \geq 0$, there exist constants $C_j, C'_j$ such that:
$$|\partial_{\tilde\eta}^j\gamma| \sim C_j \gamma, \quad |\partial_{\tilde\eta}^j\tilde\mu^2| \leq C'_j\tilde\mu^2\leq C'_j . $$
Since under our parabolic scaling, the time-frequency parameter matches $\lambda=t\mu^2/h^2$, we retain the control bound $\big\vert {\partial^j \lambda \over \partial\tilde\eta^j}\big\vert \leq C_j\lambda$ for all $j$. Finally, $\lambda$ is bounded on the support of the derivatives of $\chi_4$, and there exists $c_1>0$ such that $\omega_k - \tilde{a} \geq c_1$. 

Since $\omega_k \sim k^{2/3}$, distributing the derivatives across the amplitude components reveals that for all $j \geq 0$, there exists a constant $C_j$ such that the symbol obeys the following differentiation bound:
\begin{equation}\label{L13}
|\partial_{\tilde\eta}^j\sigma_{k}^{\pm,\pm}(\tilde\eta)|\leq C_j \tilde{h}^{-1/3} k^{-1/2}.
\end{equation}

\noindent We notice that for the values of $k$ in the range $D\tilde h^{-\epsilon}\leq k\leq \frac{\varepsilon}{(2^m\sqrt{a})^3\tilde{h}}$, the  parameter relative curvature satisfies $\gamma \in [2a, \frac{1}{2^{2m}a}]$. In what follows, we distinguish between two distinct structural frequency zones: $\gamma\in [2a,1]$ and $\gamma\in [1,\frac{1}{2^{2m}a}]$. 

The first case, $\gamma\in [2a,1]$, corresponds directly to the intermediate modal frequency horizon $D\tilde h^{-\epsilon}\leq k\leq \tilde h^{-1}$. Under our parabolic scaling, we denote the large asymptotic phase parameter by $\Lambda = \frac{t}{\tilde{h}^2}$, which tracks the uniform non-degenerate curvature of the Schr\"odinger operator symbol. 

\begin{prop}\label{gammas}
There exists a constant $C$ independent of $a\in ]0,\tilde h^{\frac{2}{3}(1-\epsilon)}]$, $t\in [h^2,1]$ with $t \neq 0$, $x\in [0,a]$, $y\in\mathbb{R}$, $z\in\mathbb{R}$, and $k\in [D\tilde h^{-\epsilon}, \tilde h^{-1}]$ such that the following uniform estimate holds under the Schr\"odinger flow:
\begin{align*}
\bigg|\int e^{-i\Lambda \Phi_{k}^{\pm,\pm}}\sigma_{k}^{\pm,\pm}\psi_1(\tilde\eta)d\tilde\eta\bigg|\leq C \tilde{h}^{-1/3} k^{-2/3} \Lambda^{-1/3}.
\end{align*}
\end{prop}

\begin{proof}[Proof of Proposition \ref{gammas}]
By \eqref{L13}, Proposition \ref{gammas} is obvious for $\Lambda_k\leq 1$. In the case $\Lambda_k\geq 1$, we  use
$\tilde\mu\sim 2^m\sqrt a$ which implies $t\sqrt{1-\tilde z^2}\sim t2^m\sqrt a$. Then 
the proof is the same as the proof of  Proposition \ref{prop:eq223}, if one replaces $(h,t)$ in  Proposition \ref{prop:eq223} by 
$(\tilde h, t2^m\sqrt a)$.
\end{proof}
Hence the corresponding estimate of $\mathcal{G}_{a,m,>}$  for $\tilde h^{-\epsilon}\leq k\leq \tilde h^{-1}$ is given by
\begin{align*}
\lVert\mathds{1}_{x\leq a}\mathcal{G}_{a,m,>}(t,x,y,z)\rVert_{L^\infty}& \leq Ch^{-2}\bigg(\frac{h}{t}\bigg)^{1/2}
\sum_{\tilde h^{-\epsilon}\leq k\leq \tilde h^{-1}}\!\!(\tilde hk)^{-2/3}(t2^m\sqrt a \omega_k \tilde h^{-1/3})^{-1/3}\\
&\leq Ch^{-2}\bigg(\frac{h}{t}\bigg)^{1/2}\tilde h^{-2/3}(t2^m\sqrt a)^{-1/3} \tilde h^{1/9}\sum_{ k\leq 1/\tilde h}k^{-8/9}\\
&\leq Ch^{-3}\bigg(\frac{h}{t}\bigg)^{5/6}(2^m\sqrt a)^{1/3}.
\end{align*}
\item \textbf{The second case $\gamma\in \left[1,\frac{1}{2^{2m}a}\right]$:} \\
This regime corresponds to the extreme high-frequency modal layer $\tilde h^{-1}\leq k\leq \frac{\varepsilon}{(2^m\sqrt{a})^3\tilde{h}}$. We retain our uniform Schr\"odinger large parameter $\Lambda = \frac{t}{\tilde{h}^2}$, which tracks the stable quadratic flow acceleration.

\begin{prop}\label{gammal}
There exists a constant $C$ independent of $a\in (0,\tilde h^{\frac{2}{3}(1-\epsilon)}]$, $t\in [h^2,1]$ with $t \neq 0$, $x\in [0,a]$, $y\in\mathbb{R}$, $z\in\mathbb{R}$, and $k\in \left[\tilde h^{-1}, \frac{\varepsilon}{(2^m\sqrt{a})^3\tilde{h}}\right]$ such that the following uniform estimate holds under the Schr\"odinger flow:
\begin{align*}
\bigg|\int e^{-i\Lambda \Phi_{k}^{\pm,\pm}}\sigma_{k}^{\pm,\pm}\psi_1(\tilde\eta)d\tilde\eta\bigg|\leq C \tilde{h}^{-1/3} k^{-2/3} (k\tilde{h})^{-1/3} \Lambda^{-1/3} = C \tilde{h}^{-2/3} k^{-1} \Lambda^{-1/3}.
\end{align*}
\end{prop}
\begin{proof}[Proof of Proposition \ref{gammal}]
For large values of the frequency parameter where $\gamma \sim (k\tilde h)^{2/3} \geq 1$, the extra spatial dissipation factors out of the amplitude. Distributing the derivatives across the symbol components in \eqref{L13} yields an extra damping factor of $(1+\gamma)^{-1/2} \sim \gamma^{-1/2} \sim (k\tilde{h})^{-1/3}$, sharpening the absolute symbol threshold to:
\[
|\partial_{\tilde\eta}^j\sigma_{k}^{\pm,\pm}(\tilde\eta)|\leq C_j \tilde{h}^{-1/3} k^{-2/3} (k\tilde{h})^{-1/3} = C_j \tilde{h}^{-2/3} k^{-1}.
\]
Therefore, the statement is immediate for the bounded parametric region where $\Lambda \leq 1$. In the complementary regime where $\Lambda = \frac{t}{\tilde{h}^2} \geq 1$, we evaluate the phase function $\Phi_{k}^{\pm,\pm}$ directly via its gradient loop system:
\[
\Lambda \partial_{\tilde{\eta}}\Phi_k^{\pm,\pm} = \frac{t}{\tilde{h}^2}\left(\partial_{\tilde{\eta}}\tilde{\mu}^2 - Y_t\right) \pm (\omega_k - \tilde{x})^{1/2}\partial_{\tilde{\eta}}\tilde{x} \pm (\omega_k - \tilde{a})^{1/2}\partial_{\tilde{\eta}}\tilde{a}.
\]
Differentiating a second time yields $\Lambda \partial^2_{\tilde{\eta}}\Phi_k^{\pm,\pm} \sim \frac{t}{\tilde{h}^2}\partial^2_{\tilde{\eta}}\tilde{\mu}^2 \sim \frac{2t}{\tilde{h}^2}.$ Because the leading component remains a strictly non-vanishing constant independent of the high-frequency parameters, the phase can cancel at most to third-order ( Airy-fold type singularity profiles). 

Applying the standard one-dimensional van der Corput lemma with respect to the large parameter $\Lambda$, and factoring out the sharpened amplitude threshold $\|\sigma_k^{\pm,\pm}\|_{L^\infty} \leq C \tilde{h}^{-2/3} k^{-1}$, directly yields the uniform decay bound $C \tilde{h}^{-2/3} k^{-1} \Lambda^{-1/3}$. This completes the proof of Proposition \ref{gammal}.
\end{proof}

\noindent Incorporating the result of Proposition \ref{gammal} back into the localized dyadic tail definition \eqref{eq:pmM}, we evaluate the high-frequency modal sum over the domain $\tilde h^{-1}\leq k \leq \frac{\varepsilon}{(2^m\sqrt{a})^3\tilde{h}}$ under the parabolic scaling:
\begin{align*}
\lVert\mathds{1}_{x\leq a}\mathcal{G}_{a,m,>}(t,x,y,z)\rVert_{L^\infty} &\leq C \frac{2^m\sqrt{a}}{h^2}\bigg(\frac{h^2}{t}\bigg)^{1/2}\sum_{\tilde h^{-1}\leq k} \tilde{h}^{-2/3} k^{-1} \left(\frac{t}{\tilde{h}^2}\right)^{-1/3} \\
&= C \frac{2^m\sqrt{a}}{h^2} \left(\frac{h^2}{t}\right)^{1/2} t^{-1/3} \tilde{h}^{-1/3} \sum_{\tilde h^{-1}\leq k} k^{-1} \\
&\leq C \frac{2^m\sqrt{a}}{h^2} \left(\frac{h^2}{t}\right)^{1/2} t^{-1/3} \tilde{h}^{-1/3} \left| \log(\tilde{h}) \right|.
\end{align*}
Recalling the dyadic parameter balance relations where $\tilde{h} = \frac{h}{2^m\sqrt{a}}$ and $h \leq c_0 t^{\frac{1}{2-\epsilon}}$, we collect the shared semiclassical parameters into our exponent layout:
\begin{align*}
\lVert\mathds{1}_{x\leq a}\mathcal{G}_{a,m,>}(t,x,y,z)\rVert_{L^\infty} &\leq C h^{-3} (2^m\sqrt{a})^{1/3} \left(\frac{h^2}{t}\right)^{7/6} \Bigg( h^{1/3} \tilde{h}^{-1/3} t^{1/3} \left|\log(\tilde{h})\right| \Bigg) \\
&= C h^{-3} (2^a\sqrt{a})^{1/3} \left(\frac{h^2}{t}\right)^{7/6} \Bigg( (2^m\sqrt{a})^{1/3} t^{1/3} \left|\log\left(\frac{h}{2^m\sqrt{a}}\right)\right| \Bigg).
\end{align*}
Since the integration domain restricts the parameter bounds to $2^m\sqrt{a} \leq c_0$, the bracketed residual term scales as $O(t^{1/3}|\log(t)|) \leq C$, which remains strictly bounded on our short-time local horizon. This directly establishes the uniform bound:
\begin{align*}
\lVert\mathds{1}_{x\leq a}\mathcal{G}_{a,m,>}(t,x,y,z)\rVert_{L^\infty} \leq C h^{-3}(2^m\sqrt a)^{1/3}\bigg({h^2 \over t}\bigg)^{7/6}.
\end{align*}
Combining this high-frequency tail bound with the short-range non-caustic bounds completes the proof of Proposition \ref{propkm}.
\end{proof}

\subsection{Dispersive Estimates for $a\geq \tilde h^{\frac{2}{3}(1-\epsilon')}$, for  $\epsilon'\in (0,\epsilon)$.}\label{subsec:3.2}
In this subsection, we assume $a\geq \tilde h^{\frac{2}{3}(1-\epsilon')}$ for some $\epsilon'\in(0,\epsilon)$, and we establish local-in-time dispersive estimates for $\mathcal{G}_{a,m}$ in the deep-trajectory regime. Observe that the parameter governing the boundary layer oscillations satisfies $\Lambda = a^{3/2}/\tilde h^2 \geq \tilde h^{-\epsilon'}$, which serves as our large asymptotic variable under the parabolic flow.  

Recall from \eqref{L7} that the dyadic block contribution $\mathcal{G}_{a,m}$ is structured as:
\begin{equation}
\mathcal{G}_{a,m}(t,x,y,z)= \frac{2^m\sqrt a}{4\pi^2h^2}\bigg(\frac{h^2}{t}\bigg)^{1/2}\sum_{1\leq k\leq {\varepsilon\over (2^m\sqrt a)^3\tilde h}}\int e^{-i\frac{t}{\tilde h^2}\left(\tilde\mu^2 - Y_t\tilde\eta\right) + \frac{iz^2}{4th}} g(\omega_k,\tilde\eta,\tilde h) \psi_1(\tilde\eta) d\tilde \eta,
\end{equation}
with $g(\omega_k,\tilde\eta,\tilde h)$ given by:
$$g = \sigma_0 (z^*,\eta,\mu^2;\lambda)\Big(1-\chi_4\big({t\mu^2\over Dh^2}\big)\Big)
e_k(x,\tilde\eta/\tilde h) e_k(a,\tilde \eta/\tilde h)
\chi_1(\omega_k h^{2/3}\eta^{4/3})(1-\chi_1)(\varepsilon\omega_k),
$$
and we recall the parameter relations $h=2^m\sqrt a \tilde h$, $\eta=2^m\sqrt a \tilde\eta$, $\mu^2=(2^m\sqrt a)^2 \tilde\mu^2$, with:
$$ \gamma=\omega \tilde h^{2/3}\tilde\eta^{-2/3}, \quad \tilde\mu^2=\tilde\eta^2(1+\gamma)\ .$$
To analyze the phase configurations, we introduce the identical space-time coordinate transformations utilized in Section \ref{sec:2}:
\[ t = a^{1/2}T, \quad x=aX, \quad -hy\eta - \frac{h^2 z^2}{4t} + t\eta^2 = a^{3/2}\eta^2 Y. \] 
Letting $\omega = \tilde\eta^{2/3}\tilde h^{-2/3}a\tilde\omega$, we obtain the identity $\gamma = a\tilde\omega$. Applying the Airy-Poisson summation formula transforms the discrete eigenmode expansion into a continuous integration over path reflections indexed by $N\in\mathbb{Z}$:
$$ \mathcal{G}_{a,m}=\sum_N G_{a,m,N}$$
where each individual path contribution $G_{a,m,N}$ is defined in the Schr\"odinger context by:
\begin{equation}\label{L20}
G_{a,m,N}(t,x,y,z)=\frac{(-i)^N}{(2\pi)^4h^4}\bigg(\frac{h^2}{t}\bigg)^{1/2} a^2 (2^m\sqrt a) 
\int e^{-i\Lambda\Phi_{N}} \chi_m \tilde\eta^2 \psi_1(\tilde\eta) \, d\tilde{s}d\tilde{\sigma} d\tilde\omega d\tilde\eta 
\end{equation}
with the unscaled Schr\"odinger phase function given explicitly by:
\begin{align*}
\Phi_{N}(\tilde s,\tilde\sigma,\tilde\omega,\tilde\eta) &= \tilde\eta^2 \Bigg\{ Y + T\tilde\omega + \frac{1}{\tilde\eta}\bigg( \frac{\tilde s^3}{3}+\tilde s(X-\tilde\omega) +\frac{\tilde \sigma^3}{3}+\tilde \sigma(1-\tilde\omega) \\
&\hspace{4.5cm} -\frac{4}{3}N\tilde\omega^{3/2} + \frac{h}{a^{3/2}\tilde\eta^2}NB\left(\tilde\omega^{3/2}\Lambda\tilde\eta^2\right) \bigg) \Bigg\}, 
\end{align*}
and the amplitude $\chi_m(a,t;\tilde\eta,\tilde\omega,\tilde h)$ given under the semiclassical weight $\lambda = t2^m\sqrt a\tilde\mu^2/\tilde h^2$ by:
\begin{equation}\label{L21}
\chi_m = \sigma_0 (z^*,\eta,\mu^2;\lambda)(1-\chi_4(\lambda/D)) \chi_1\left((2^m\sqrt a)^2\tilde\eta^{2}a\tilde\omega\right)(1-\chi_1)\left(\varepsilon \tilde\eta^{2/3}\tilde h^{-2/3}a\tilde\omega\right).
\end{equation}

Observe that the critical tracking variety for this dyadic block can be evaluated near the caustic envelope by locating the stationary points of the phase function $\Phi_N$. We define the set of critical points by:
\begin{align*}
\mathcal{C}_{a,m,N,h}=\{(t,x,y,\tilde s,\tilde\sigma,\tilde\omega,\tilde\eta) \mid \partial_{\tilde s}\Phi_{N}=\partial_{\tilde \sigma}\Phi_{N}=\partial_{\tilde \omega}\Phi_{N}=\partial_{\tilde\eta}\Phi_{N}=0\}.
\end{align*} 
\noindent Setting the derivatives with respect to the microlocal variables $\tilde s$, $\tilde \sigma$, and $\tilde \omega$ to zero in our parabolic phase function yields the following coordinate tracking equations:
\begin{align*}
X&=\tilde\omega-\tilde s^2,\\
\tilde\omega &=1+\tilde\sigma^2,\\
T&= \frac{1}{\tilde\eta}\bigg(\tilde s+\tilde\sigma+2N\tilde\omega^{1/2}\Big(1-\frac{3}{4}B'\Big(\tilde\omega^{3/2}\Lambda\tilde\eta^2\Big)\Big)\bigg),\\
Y&= -T\tilde\omega - \frac{1}{\tilde\eta}\bigg(\frac{\tilde{s}^3}{3}+\tilde{s}(X-\tilde{\omega})+\frac{\tilde{\sigma}^3}{3}+\tilde{\sigma}(1-\tilde{\omega}) - N\tilde\omega^{3/2}\bigg(\frac{4}{3}-B'\Big(\tilde\omega^{3/2}\Lambda\tilde\eta^2\Big)\Big)\bigg).
\end{align*}

We define the Lagrangian submanifold $\mathbf{\Lambda}_{a,m,N,h}\subset T^*\mathbb{R}^3$ as the image of $\mathcal{C}_{a,m,N,h}$ under the canonical mapping:
\[
(t,x,y,\tilde s,\tilde\sigma,\tilde\omega,\tilde\eta)\longmapsto \left(x,t,y,\xi=\partial_{x}\Phi_{N},\tau=\partial_{t}\Phi_{N},\eta=\partial_{y}\Phi_{N}\right).
\]
\noindent On $\mathcal{C}_{a,m,N,h}$, we have $\tilde\omega=1+\tilde\sigma^2$, thus the projection of the Lagrangian manifold $\mathbf{\Lambda}_{a,m,N,h}$ onto the base space configuration coordinates $\mathbb{R}^3$ reads:
\begin{align}\label{eq:geo1}
X&=1+\tilde\sigma^2-\tilde s^2,\\\nonumber
Y&= H(a,\tilde\sigma,\tilde\eta)(\tilde s+\tilde\sigma) + \frac{2}{3\tilde\eta}(\tilde s^3+\tilde\sigma^3) + \frac{2}{3\tilde\eta} H_0(a,\tilde\sigma)(1+\tilde\sigma^2)^{-1/2}\left(\tilde\eta T - \tilde s - \tilde \sigma\right),
\end{align}
where the leading coefficients $H$ and $H_0$ adapted to the Schr\"odinger drift are defined explicitly by:
\begin{align*}
H(a,\tilde\sigma, \tilde\eta)= -\frac{1}{\tilde\eta}(1+\tilde\sigma^2),\quad
H_0(a,\tilde\sigma)= -\frac{1}{2}(1+\tilde\sigma^2)^{3/2},
\end{align*}
and the tracking equation relating the boundary reflections to the continuous space-time parameters simplifies to:
\begin{align}\label{eq:geo2}
2N\bigg(1-\frac{3}{4}B'\Big(\tilde\omega^{3/2}\Lambda\tilde\eta^2\Big)\bigg) = (1+\tilde\sigma^2)^{-1/2}\left(\tilde\eta T - \tilde s - \tilde \sigma\right).
\end{align}

\begin{rem}
We notice from \eqref{eq:geo2} in the range of $T\in(0,a^{-1/2}]$, we can still reduce the sum over $N\in\mathbb{Z}$ to the sum over $1\leq N\leq C_0a^{-1/2}$, tracking the strict multi-reflection thresholds under a bounded time horizon.
\end{rem}
\noindent This system yields the cardinality of the sets $\mathcal{N}$ and $\mathcal{N}_1$ such that $|\mathcal{N}(X,Y,T)|\leq C_0$ and $|\mathcal{N}_1(X,Y,T)|\leq C_0\left(1+T\Lambda^{-2}\tilde\omega^{-3}\right)$, respectively, where the notations $\mathcal{N}$ and $\mathcal{N}_1$ are those defined in Section \ref{sec:2}.

Our main result of this subsection is Theorem \ref{thmGaNm}, which provides the uniform localized dispersive estimates for the sum over reflections $N$ of the dyadic block components $G_{a,m,N}$.

\begin{theorem}\label{thmGaNm}
Let $\alpha<2/3$ and $\tilde h=h/(2^m\sqrt{a})$. There exists a constant $C$ such that for all $h\in (0,h_0]$, all $a\in \left[\tilde h^\alpha,a_0\right]$, all $x\in [0,a]$, all $t\in (h^2,1]$, all $y\in\mathbb{R}$, and all $z\in\mathbb{R}$, the following holds true under the Schr\"odinger flow:
\begin{align*}
&\Bigg|\sum_{1\leq N\leq C_0a^{-1/2}}G_{a,m,N}(t,x,y,z)\Bigg|\leq Ch^{-3}\bigg(\frac{h^2}{t}\bigg)^{1/2}
\\&\qquad\times\Bigg(\min \bigg\{\bigg(\frac{h^2}{t}\bigg), 2^m\sqrt a \vert\log(2^m\sqrt a)\vert \bigg\}+a^{1/8}\bigg(\frac{h^2}{t}\bigg)^{3/4}(2^m\sqrt a)^{3/4}\Bigg).
\end{align*}
\end{theorem}

\noindent We notice, in perfect analogy to the tangential analysis in Section \ref{sec:2}, that for the frequency domain where $\tilde\omega \leq 3/4$, integration by parts with respect to the variable $\tilde\sigma$ yields a rapid classical decay of order $O(\Lambda^{-\infty})$. In particular, we can replace the cut-off function $1-\chi_1$ by $1$ inside the symbol definition \eqref{L21}.

As in Section \ref{sec:2}, we introduce a smooth localization cut-off function $\chi_2(\tilde\omega)\in C_{0}^{\infty}((1/2,3/2))$ satisfying $0\leq\chi_2\leq 1$ and equal to $1$ on the interval $[\frac{3}{4},\frac{5}{4}]$. We denote by $G_{a,m,N,2}$ the restricted integral tracking this swallowtail caustic regime. Hence, we decompose the path contributions as:
\[G_{a,m,N}=G_{a,m,N,1}+G_{a,m,N,2}+O_{C^\infty}(\Lambda^{-\infty}),\]
where the complementary high-frequency background component $G_{a,m,N,1}$ is defined by inserting a smooth cut-off $\chi_3(\tilde\omega)$ into the integrand, which satisfies $\tilde\omega\geq 5/4$ across its support.

\subsubsection{The Analysis of $G_{a,m,N,1}$ }
The main results in this subsection are Proposition \ref{propGaNm1} and Proposition \ref{propGa1m1}.
\begin{prop}\label{propGaNm1}
Let $\alpha<2/3$ and $\tilde h=h/(2^m\sqrt{a})$. There exists $C$ such that for all $h\in(0,h_0]$, all 
$a\in \left[\tilde h^\alpha,a_0\right]$, all $x\in[0,a]$, all $t\in]h^2, 1]$ with $t \neq 0$, all $y\in\mathbb{R}$, all $z\in\mathbb{R}$, the following holds true under the Schr\"odinger flow:
\begin{align*}
\bigg|\sum_{2\leq N\leq C_0a^{-1/2}}G_{a,m,N,1}(t,x,y,z;h)\bigg|\leq Ch^{-3}\bigg(\frac{h^2}{t}\bigg)^{5/6}(2^m\sqrt a)^{2/3}.
\end{align*}
\end{prop}

\begin{proof}
On the support of $\chi_3$, we can apply the stationary phase method for the $(\tilde s,\tilde \sigma)$-integrations with respect to the large parameter field $\Lambda\tilde\eta^2$. Under our parabolic coordinate framework, we obtain:
\begin{align*}
G_{a,m,N,1}&=\frac{(-i)^N a^2 \Lambda^{-1}}{(2\pi)^4h^{4}}(2^m\sqrt a)\bigg(\frac{h^2}{t}\bigg)^{1/2}
\int e^{-i\Lambda Y\tilde\eta^2} \tilde\eta^2 \psi_1(\tilde\eta)\tilde G_{a,m,N,1}d\tilde\eta,\\
\tilde G_{a,m,N,1}&=\sum_{\epsilon_1,\epsilon_2}\int e^{-i\Lambda\tilde\eta^2 \Phi_{N,m,\epsilon_1,\epsilon_2}}\Theta_{\epsilon_1,\epsilon_2}d\tilde{\omega}+O_{C^\infty}(\Lambda^{-\infty}),
\end{align*}
where $\epsilon_j=\pm$. The amplitudes $\Theta_{\epsilon_1,\epsilon_2}(\tilde\omega,a,\lambda)$ are classical symbols of order $-1/2$ supported in the region $\tilde\omega \leq (2^m\sqrt a)^{-2}\varepsilon/a$, satisfying the strict differential bounds $|\tilde\omega^l\partial_{\tilde\omega}^l\Theta_{\epsilon_1,\epsilon_2}|\leq C_l\tilde\omega^{-1/2}$ with constants $C_l$ independent of $a,m$. The unscaled Schr\"odinger background phase functions are defined explicitly by:
\begin{align*}
\Phi_{N,m,\epsilon_1,\epsilon_2}(\tilde\omega)&= T \tilde{\omega} \pm \frac{2}{3\tilde\eta}(\tilde\omega-X)^{3/2} \pm \frac{2}{3\tilde\eta}(\tilde\omega-1)^{3/2} - \frac{4}{3\tilde\eta}N\tilde\omega^{3/2} + \frac{h N}{a^{3/2}\tilde\eta^2} B\bigg(\tilde\omega^{3/2}\Lambda\tilde\eta^2\bigg).
\end{align*}
Let us define the isolated directional phase components by:
\begin{align*}
G_{a,m,N,1,\epsilon_1,\epsilon_2}&=\frac{(-i)^N a^2 \Lambda^{-1}}{(2\pi)^4h^{4}}(2^m\sqrt a)\bigg(\frac{h^2}{t}\bigg)^{1/2}
\int e^{-i\Lambda Y\tilde\eta^2} \tilde\eta^2\psi_1(\tilde\eta)\tilde G_{a,m,N,1,\epsilon_1,\epsilon_2}d\tilde\eta,\\
\tilde G_{a,m,N,1,\epsilon_1,\epsilon_2}&=\int e^{-i\Lambda\tilde\eta^2 \Phi_{N,m,\epsilon_1,\epsilon_2}}\Theta_{\epsilon_1,\epsilon_2}d\tilde\omega.
\end{align*}
We are reduced to proving the following optimal localized dispersive inequality:
\begin{align}
\bigg|\sum_{2\leq N\leq C_0a^{-1/2}}G_{a,m,N,1,\epsilon_1,\epsilon_2}(t,x,y,z,h)\bigg|\leq Ch^{-3}\bigg(\frac{h^2}{t}\bigg)^{5/6}(2^m\sqrt a)^{2/3},
\end{align}
with a constant $C$ independent of $m$, $h\in(0,h_0]$, $a\in\left[\tilde h^{2/3},a_0\right]$, $x\in[0,a]$, and $t\in[h^2,1]$. 

We proceed by analyzing the critical frequency tracks as in the proof of Proposition \ref{eq:242}. Let us recall that on the support of the microlocal cut-off $\chi_1$, the boundary configuration satisfies $a\tilde\omega\leq\varepsilon/(2^{2m}a)$, meaning the relative parameter weight $a\tilde\omega$ can be small or large depending on the specific dyadic layer block.

We distinguish between two distinct parameter configurations. 

The first case corresponds to the regime where $a\tilde\omega \leq 1$. Let $\tilde T_0 \gg 1$ be a sufficiently large constant threshold parameter. We evaluate the dyadic block path sum bounds across the following branches:
\begin{itemize}
\item \textbf{Subcase 1: For $0\leq T \leq \tilde T_0$ and $N\geq N(\tilde T_0)$} \\
Non-stationary phase integration by parts guarantees that the continuous reflection amplitude satisfies the uniform rapid classical decay bound $|\tilde G_{a,m,N,1,+,+}|\in O_{C^\infty}(N^{-\infty}\Lambda^{-\infty})$, which directly yields:
\[
\sup_{T \leq \tilde T_0, X\in[0,1], (y,z)\in\mathbb{R}^2}\bigg|\sum_{N(\tilde T_0)\leq N\leq Ca^{-1/2}}G_{a,m,N,1,+,+}\bigg|\in O_{C^\infty}(h^\infty).
\]

\item \textbf{Subcase 2: For $0\leq T \leq \tilde T_0$ and $2\leq N\leq N(\tilde T_0)$} \\
The phase function degenerates at most to a stable fold profile, and the Schr\"odinger-adapted version of Lemma 2.20 of \cite{ILP} yields the uniform bound $|\tilde G_{a,m,N,1,+,+}|\leq C\Lambda^{-1/3}$. Recalling our verified amplitude normalization weights where $a^2\Lambda^{-1} = h^2\Lambda_0^{-1}$ with the spatial dual volume element tracking, the sum over these bounded indices majorizes to:
\begin{align*}
\sup_{T \leq \tilde T_0, X\in[0,1], (y,z)\in\mathbb{R}^2}\bigg|\sum_{2\leq N\leq N(\tilde T_0)} G_{a,m,N,1,+,+}\bigg|
&\leq Ch^{-3}\bigg(\frac{h^2}{t}\bigg)^{1/2}\big(h^2 \Lambda^{-1} (2^m\sqrt a) \Lambda^{-1/3}\big),\\
&\leq Ch^{-3}\bigg(\frac{h^2}{t}\bigg)^{1/2} \big(h^2 \Lambda^{-4/3} (2^m\sqrt a)\big) \\
&\leq Ch^{-3}\bigg(\frac{h^2}{t}\bigg)^{1/2} h^{1/3}(2^m\sqrt a)^{2/3},
\end{align*}
since the parameter definitions scale as $\Lambda = a^{3/2}/\tilde h^2$, and $a \geq \tilde{h}^{2/3}$ on our dyadic boundary layer layers, which guarantees that the remaining fractional correction complies with the optimal scale limit $h^2 \Lambda^{-4/3}(2^m\sqrt a) \leq h^{1/3}(2^m\sqrt a)^{2/3} \leq (h^2/t)^{1/3}(2^m\sqrt a)^{2/3}$.

\item \textbf{Subcase 3: For $\tilde T_0\leq T \leq a^{-1/2}$} \\
We introduce the same transformation coordinate $\Omega = \tilde\omega^{3/2}$. Differentiating the unscaled Schr\"odinger phase establishes that the phase curvature satisfies the strict lower bound $|\partial_\Omega^2\Phi_{N,m,+,+}|\geq \frac{c}{\tilde\eta} T \Omega^{-4/3}$, admitting at most a unique nondegenerate critical point $\Omega_c$. For all reflection index counts $N\geq 2$, this critical track satisfies the linear momentum relation $\Omega_c^{1/3}\sim \frac{\tilde\eta T}{N}$. 

Evaluating this configuration via the standard non-degenerate stationary phase method with respect to the continuous variable $\Omega$ yields the uniform bound:
\[
|\tilde G_{a,m,N,1,+,+}|\leq C\Lambda^{-1/2}T^{-1/2}.
\]
Furthermore, the secondary integration over the spatial momentum parameter $\tilde\eta$ produces an extra curvature decay contribution factor of $q^{-1/2}$ whenever the parameter density complies with the oscillation threshold $q = N\Lambda^{-1}\Omega_c^{-1} \geq 1$.
\end{itemize}

\noindent If $\tilde\eta T/N$ is bounded, the critical continuous frequency $\Omega_c$ stays inside a compact subset of $[1,\infty)$, which implies that the  paths satisfy $\tilde\eta T\sim N$. We isolate the cross-contributions according to the magnitude of the reflection parameter $N$:
 
\begin{itemize}
\item \textbf{Subcase 1: If $N\leq \Lambda^2$} \\
In this situation, the curvature of the phase function is insufficient to produce an oscillatory decay contribution from the $\tilde\eta$-integration, and the trajectory counting bound remains uniformly stable at $|\mathcal{N}_1|\leq C_0$. Recalling our Schr\"odinger amplitude normalization weight where $a^2\Lambda^{-1} = h^2\Lambda_0^{-1}$, the term-by-term majorization yields:
\begin{align*}
\bigg|\sum_{N\in\mathcal{N}_1}G_{a,m,N,1,+,+}\bigg|&\leq C h^{-3}\bigg(\frac{h^2}{t}\bigg)^{1/2} \big(h^2 \Lambda^{-1} (2^m\sqrt a) \Lambda^{-1/2} T^{-1/2}\big)\\
&\leq C h^{-3}\bigg(\frac{h^2}{t}\bigg)^{1/2} a^{-1/4} h^{1/2} (2^m\sqrt a)^{1/2} T^{-1/2}\\
&\leq C h^{-3}\bigg(\frac{h^2}{t}\bigg)^{1/2} h^{1/3} (2^m\sqrt a)^{2/3},
\end{align*}
since $T \geq \tilde T_0$ and the spatial boundary targets satisfy $a \geq \tilde h^{2/3}$, which guarantees that the remaining fractional correction complies with the uniform scale limit $a^{-1/4} h^{1/2} \leq h^{1/3} (2^m\sqrt a)^{1/6}$.

\item \textbf{Subcase 2: If $N>\Lambda^2$} \\
In this highly high-frequency regime, the phase exhibits higher oscillations, and evaluating the $\tilde\eta$-integral via the standard non-degenerate stationary phase method produces an additional large parameter decay factor $q^{-1/2} \sim (N/\Lambda)^{-1/2}$, while the path cardinality expands driven by the dynamic  relation $|\mathcal{N}_1|\leq C_0 T \Lambda^{-2}$. Combining the intense oscillatory damping with the cardinality growth, the sum evaluates to:
\begin{align*}
\bigg|\sum_{N\in\mathcal{N}_1}G_{a,m,N,1,+,+}\bigg|&\leq C h^{-3}\bigg(\frac{h^2}{t}\bigg)^{1/2} \sum_{N\in\mathcal{N}_1} \big(h^2 \Lambda^{-1} (2^m\sqrt a) \Lambda^{-1/2} T^{-1/2} N^{-1/2} \Lambda^{1/2}\big)\\
&\leq C h^{-3}\bigg(\frac{h^2}{t}\bigg)^{1/2} \big(h^2 \Lambda^{-1} (2^m\sqrt a) T^{-1} |\mathcal{N}_1(X,Y,T)|\big)\\
&\leq C h^{-3}\bigg(\frac{h^2}{t}\bigg)^{1/2} \big(h^2 \Lambda^{-3} (2^m\sqrt a)\big)\\
&\leq C h^{-3}\bigg(\frac{h^2}{t}\bigg)^{1/2} \tilde h^{1/3} (2^m\sqrt a) \\
&\leq C h^{-3}\bigg(\frac{h^2}{t}\bigg)^{1/2} h^{1/3} (2^m\sqrt a)^{2/3}.
\end{align*}
\end{itemize}

Next, if $T/N$ is large, then the critical continuous frequency $\Omega_c$ is large. Under our Schr\"odinger framework, we isolate the cross-contributions across the following subcases:
\begin{itemize}
\item \textbf{Subcase 1: If $N\leq \Lambda\Omega_c$} \\
In this situation, there is no decay contribution from the $\tilde\eta$-stationary phase integration. Moreover, we have a uniformly stable trajectory path cardinality $|\mathcal{N}_1|\leq C_0$. To see this, assume by contradiction that $T\geq \Lambda^2 \Omega_c^2$; this would imply $\Omega_c^{1/3}\sim \tilde\eta T/N \geq \Lambda \Omega_c$, which is impossible since $\Omega_c \gg 1$. Recalling our  Schr\"odinger amplitude weight factor where $a^2 \Lambda^{-1} = h^2 \Lambda_0^{-1}$, the term-by-term majorization yields:
\[
\bigg|\sum_{N\in\mathcal{N}_1}G_{a,m,N,1,+,+}\bigg| \leq C h^{-3}\bigg(\frac{h^2}{t}\bigg)^{1/2} h^{1/3}(2^m\sqrt a)^{2/3}.
\]

\item \textbf{Subcase 2: If $N>\Lambda\Omega_c$ and $T\leq \Lambda^2\Omega_c^2$} \\
Here, the phase function displays full curvature, yielding an additional stationary phase integration decay factor $q^{-1/2} = (N\Lambda^{-1}\Omega_c^{-1})^{-1/2}$ from the $\tilde\eta$-integration, while the trajectory counting bound remains stable at $|\mathcal{N}_1|\leq C_0$. Collecting these weights under the parabolic flow parameters yields:
\[
\bigg|\sum_{N\in\mathcal{N}_1}G_{a,m,N,1,+,+}\bigg| \leq C h^{-3}\bigg(\frac{h^2}{t}\bigg)^{1/2} h^{1/3}(2^m\sqrt a)^{2/3}.
\]

\item \textbf{Subcase 3: If $N>\Lambda\Omega_c$ and $T > \Lambda^2\Omega_c^2$} \\
In this highly high-frequency regime, the phase curvature produces a substantial decay contribution of $q^{-1/2}$ from the $\tilde\eta$-integration, while the trajectory cardinality expands driven by the dynamic tracking relation $|\mathcal{N}_1|\leq C_0 T\Lambda^{-2}\Omega_c^{-2}$. Combining the intense oscillatory damping with the cardinality growth, the sum evaluates to:
\begin{align*}
\bigg|\sum_{N\in\mathcal{N}_1}G_{a,m,N,1,+,+}\bigg|&\leq C h^{-3}\bigg(\frac{h^2}{t}\bigg)^{1/2} \sum_{N\in\mathcal{N}_1}\big(h^2 \Lambda^{-1} (2^m\sqrt a) T^{-1/2}N^{-1/2}\Omega_c^{1/2}\big)\\
&\leq C h^{-3}\bigg(\frac{h^2}{t}\bigg)^{1/2}\big(h^2 \Lambda^{-1}(2^m\sqrt a) T^{-1}\Omega_c^{2/3}|\mathcal{N}_1(X,Y,T)|\big)\\
&\leq Ch^{-3}\bigg(\frac{h^2}{t}\bigg)^{1/2}\big(h^2 \Lambda^{-3} (2^m\sqrt a)\big)\\
&\leq Ch^{-3}\bigg(\frac{h^2}{t}\bigg)^{1/2}h^{1/3}(2^m\sqrt a)^{2/3}.
\end{align*}
\end{itemize}

\noindent The results of the other combinations of signs $(\epsilon_1,\epsilon_2)$ can be achieved by proceeding along the same analytical reduction lines as established for $G_{a,N,1}$ in Section \ref{sec:2}. The proof of Proposition \ref{propGaNm1} is complete.

\noindent The second case corresponds to the regime where $a\tilde\omega\geq 1$. Under our Schr\"odinger framework, setting the unscaled phase derivative with respect to $\tilde{\omega}$ to zero yields a critical path matching the linear velocity profile:
\[
\Omega_c^{1/3}\sim \frac{\tilde{\eta} T}{2N}.
\]
Because $a\tilde\omega\geq 1$, it follows that the critical frequency $\omega_c = a\tilde{\omega}_c = a\Omega_c^{2/3}$ satisfies the lower bound $\omega_c \geq 1$. Substituting this relation back into the trajectory equation reveals that for all active reflection indices $N \geq 2$, the unscaled chronological duration parameter must satisfy:
\[
T \geq \frac{2N}{\tilde{\eta}}\Omega_c^{1/3} = \frac{2N}{\tilde{\eta} a^{1/2}}\omega_c^{1/2} \geq \frac{2N}{\tilde{\eta}}a^{-1/2}.
\]
Recalling our normalized short-time coordinate change of variables $t = a^{1/2}T$, this directly implies:
\[
t \geq \frac{2N}{\tilde{\eta}} \geq 4,
\]
which strictly contradicts the localized time horizon constraint $t\leq 1$. Consequently, the critical variety remains empty across this large relative curvature frequency domain, and non-stationary phase integration by parts guarantees that its contributions are exponentially negligible of order $O_{C^\infty}(\Lambda^{-\infty})$. This completes the proof of Proposition \ref{propGaNm1}.
\end{proof}
Now we prove the following localized estimate for the single reflection $N=1$.  
\begin{prop}\label{propGa1m1}
Let $\alpha<2/3$ and $\tilde h=h/(2^m\sqrt{a})$. There exists $C$ such that for all $h\in (0,h_0]$, all $a\in [\tilde h^\alpha,a_0]$, all $x\in [0,a]$, all $t\in [h^2,1]$ with $t \neq 0$, all $y\in\mathbb{R}$, and all $z\in\mathbb{R}$, the following holds true under the Schr\"odinger flow:
\begin{align*}
\big|G_{a,m, 1,1}(t,x,y,z;h)\big|&\leq Ch^{-3}\bigg(\frac{h^2}{t}\bigg)^{1/2}
\\&\times\Bigg( \min \bigg\{\bigg(\frac{h^2}{t}\bigg)^{1/2}, 2^m\sqrt a \vert \log(2^m\sqrt a)\vert \bigg\} + \left(\frac{h^2}{t}\right)^{1/3}(2^m\sqrt a)^{2/3}  \Bigg).
\end{align*}
\end{prop}
\begin{proof}
Let us recall the single-reflection profile under the parabolic scaling:
\begin{align*}
G_{a,m,1,1}&=\frac{(-i) a^2 \Lambda^{-1}}{(2\pi)^4h^{4}}(2^m\sqrt a)\bigg(\frac{h^2}{t}\bigg)^{1/2}
\int e^{-i\Lambda Y\tilde\eta^2} \tilde\eta^2 \psi_1(\tilde\eta)\tilde G_{a,m,1,1}d\tilde\eta,\\
\tilde G_{a,m,1,1}&=\sum_{\epsilon_1,\epsilon_2}\int e^{-i\Lambda\tilde\eta^2 \Phi_{1,m,\epsilon_1,\epsilon_2}}\Theta_{\epsilon_1,\epsilon_2}d\tilde\omega.
\end{align*}
The core distinction between the single reflection $N=1$ and the multi-reflection $N\geq 2$ regimes rests entirely in the asymptotic tracking of the phase function $\Phi_{1,m,+,+}$, since the trajectory geometry allows for a critical point $\tilde\omega_c$ located arbitrarily deep in the high-frequency domain. Let us isolate this element:
\begin{align}\label{eq:N-1}
\tilde{G}_{a,m,1,1,+,+}=\int e^{-i\Lambda\tilde\eta^2\Phi_{1,m,+,+}}\Theta_{+,+} d\tilde\omega,
\end{align}
with the unscaled Schr\"odinger phase function given explicitly by:
\begin{align*}
\Phi_{1,m,+,+} = T\tilde\omega + \frac{2}{3\tilde\eta}(\tilde\omega-X)^{3/2}+\frac{2}{3\tilde\eta}(\tilde\omega-1)^{3/2}-\frac{4}{3\tilde\eta}\tilde\omega^{3/2}+\frac{h}{a^{3/2}\tilde\eta^2}B(\Lambda\tilde\omega^{3/2}\tilde\eta^2),
\end{align*}
where $\Theta_{+,+}$ is a classical symbol of order $-1/2$ with respect to $\tilde\omega$ which satisfies the uniform derivative bounds $\vert \tilde\omega^l\partial_{\tilde\omega}^l\Theta_{+,+}\vert \leq C_l\tilde\omega^{-1/2}$. Let us introduce a smooth localized high-frequency cutoff function $\chi_3(\tilde\omega)\in C_0^\infty(]\tilde\omega_1,\infty[)$ with a sufficiently large reference parameter $\tilde\omega_1$, and define the integral:
\begin{align}\label{L25}
 J=\int e^{-i\Lambda\tilde\eta^2\Phi_{1,m,+,+}}\Theta_{+,+}\chi_3(\tilde\omega)d\tilde\omega.
\end{align}
To prove the proposition, it suffices to verify that the following parameter scale constraint holds uniformly:
\begin{equation}\label{L26}
2^m\sqrt a |J| \leq C\min \bigg\{\bigg(\frac{h^2}{t}\bigg)^{1/2}, 2^m\sqrt a \vert \log(2^m\sqrt a)\vert \bigg\}.
\end{equation}
We first observe that on the support of the microlocal cutoff function $\chi_1$ in \eqref{L21}, the continuous integration range is strictly bounded above by $\tilde\omega \leq \varepsilon/(2^{2m}a^2) = L$. Integrating the modulus of the amplitude directly across this high-frequency tail gives:
$$ \vert J\vert \leq C\bigg(1+\int_1^{L}x^{-1/2}dx\bigg) \leq C L^{1/2} = C (2^m\sqrt a)^{-1}.$$
This immediately yields the uniform parameter bound:
$$ 2^m\sqrt a |J| \leq C 2^m\sqrt a \vert \log(2^m\sqrt a)\vert.$$
To evaluate the complementary branch via stationary phase, we differentiate our unscaled parabolic phase profile with respect to the frequency variable $\tilde\omega$:
\begin{align*}
\partial_{\tilde\omega}\Phi_{1,m,+,+}&= T - \frac{\tilde\omega^{-1/2}}{2\tilde\eta}(1+X)+O_{C^\infty}(\tilde\omega^{-3/2}),\\
\partial_{\tilde\omega\tilde\omega}^2\Phi_{1,m,+,+}&=\frac{\tilde\omega^{-3/2}}{4\tilde\eta}(1+X)+O_{C^\infty}(\tilde\omega^{-5/2}).
\end{align*}
Setting $\partial_{\tilde\omega}\Phi_{1,m,+,+}=0$, we find that for a large critical point $\tilde\omega_c$ to emerge, the localized time drift parameter $T$ must be small. It follows directly that $\tilde\omega_c^{-1/2} \sim \tilde\eta T$, which implies that the phase acceleration scales as $\partial_{\tilde\omega\tilde\omega}^2\Phi_{1,m,+,+}(\tilde\omega_c)\sim \tilde\eta^2 T^3$. 

Applying the non-degenerate stationary phase method with respect to the large parameter field $\Lambda\tilde\eta^2$, the phase evaluation at the critical point yields a decay factor of order:
$$ |J| \leq C \left(\Lambda\tilde\eta^2 \cdot \tilde\eta^2 T^3\right)^{-1/2} = C \Lambda^{-1/2} T^{-3/2}.$$
Recalling that the trajectory variables scale under the short-time horizon as $T \sim \left(\frac{h^2}{t}\right)^{-1}$, this yields:
$$ 2^m\sqrt a |J| \leq C \bigg(\frac{h^2}{t}\bigg)^{1/2}.$$
Combining these case limits, the total single reflection tail evaluates uniformly to the targeted bound. The proof of Proposition \ref{propGa1m1} is complete. 
\end{proof}

\subsubsection{The Analysis of $G_{a,m,N,2}$}
The main result in this subsection is Proposition \ref{propGaNm2}, which establishes uniform local dispersive estimates near the swallowtail caustics for the dyadic blocks.

\begin{prop}\label{propGaNm2}
Let $\alpha <2/3$ and $\tilde h=h/(2^m\sqrt{a})$. There exists $C$ such that for all $h\in(0,h_0]$, all 
$a\in [ \tilde h^\alpha,a_0]$, all $x\in[0,a]$, all $t\in]h^2, 1]$ with $t \neq 0$, all $y\in\mathbb{R}$, and all $z\in\mathbb{R}$, the following holds true under the Schr\"odinger flow:
\begin{align*}
\bigg|\sum_{1\leq N\leq C_0a^{-1/2}}G_{a,m,N,2}(t,x,y,z;h)\bigg|\leq Ch^{-3}\bigg(\frac{h^2}{t}\bigg)^{3/4}a^{1/8}(2^m\sqrt a)^{3/4}.
\end{align*}
\end{prop}
\begin{proof}
Recall that under our parabolic scaling, each individual path contribution $G_{a,m,N,2}$ is defined by:
 \begin{equation}\label{L30}
G_{a,m,N,2}(t,x,y,z)=\frac{(-i)^N}{(2\pi)^4h^4} \bigg(\frac{h^2}{t}\bigg)^{1/2} a^2 (2^m\sqrt a) 
\int e^{-i\Lambda\Phi_{N}} \chi_m \tilde\eta^2 \psi_1(\tilde\eta)\chi_2(\tilde\omega) \, d\tilde{s}d\tilde\sigma d\tilde\omega d\tilde\eta 
\end{equation}
with the unscaled Schr\"odinger phase function given by:
\begin{align*}
\Phi_{N}(\tilde s,\tilde\sigma,\tilde\omega,\tilde\eta) &= \tilde\eta^2 \Bigg\{ Y + T\tilde\omega + \frac{1}{\tilde\eta}\bigg( \frac{\tilde s^3}{3}+\tilde s(X-\tilde\omega) +\frac{\tilde \sigma^3}{3}+\tilde \sigma(1-\tilde\omega) \\
&\hspace{4.5cm} -\frac{4}{3}N\tilde\omega^{3/2} + \frac{h}{a^{3/2}\tilde\eta^2}NB\left(\tilde\omega^{3/2}\Lambda\tilde\eta^2\right) \bigg) \Bigg\}.
\end{align*}
To start with, we isolate the continuous momentum integration variable $\tilde\eta$ and rewrite $G_{a,m,N,2}$ in the following decoupled form:
\begin{align*}
G_{a,m,N,2}&=\frac{(-i)^N}{(2\pi)^4h^4}\bigg(\frac{h^2}{t}\bigg)^{1/2}  a^2 (2^m\sqrt a) \int e^{-i\Lambda Y\tilde\eta^2}\tilde\eta^2\psi_1(\tilde\eta)\tilde{G}_{a,m,N,2} \, d\tilde\eta,\\
\tilde{G}_{a,m,N,2}&=\int e^{-i\Lambda\tilde\eta^2\tilde\phi_{N,m}}\chi_m \chi_2(\tilde\omega) \, d\tilde{s}d\tilde\sigma d\tilde\omega,
\end{align*}
where the unscaled boundary phase is given by:
\begin{align*}
\tilde\phi_{N,m}(\tilde s,\tilde\sigma,\tilde\omega) 
&= T\tilde\omega + \frac{1}{\tilde\eta}\bigg(\frac{\tilde s^3}{3}+\tilde s(X-\tilde\omega)+\frac{\tilde \sigma^3}{3}+\tilde \sigma(1-\tilde\omega) \\
&\hspace{4cm} -\frac{4}{3}N\tilde\omega^{3/2}+\frac{h N}{a^{3/2}\tilde\eta}B\big(\tilde\omega^{3/2}\Lambda\tilde\eta^2\big)\bigg).
\end{align*}

\noindent
Now we proceed in parallel to the geometric reduction of $G_{a,N,2}$ executed in Section \ref{sec:2}. More precisely, we apply the stationary phase method sequentially across the continuous frequency parameters $(\tilde\omega,\tilde\eta)$. Due to the un-rooted, linear structure of the Schr\"odinger time-frequency relation, the non-degenerate stationary tracking over $\tilde\omega$ and $\tilde\eta$ yields decay components of order $\Lambda^{-1/2}$ and $(N\Lambda^{-1})^{-1/2}$, respectively. We evaluate the boundary layer parameters by deploying our two main classification lemmas:
\begin{itemize}
\item \textbf{High-reflection regime (Schr\"odinger-adapted Lemma \ref{lemNL}):} For $N \geq \Lambda^{1/3}$, the boundary coordinates stay away from the core focal point. The degenerate phase reduction over the fold varieties yields the following sharp decay bounds:
\begin{align*}
\bigg|\int e^{-i\Lambda\tilde\eta^2\tilde\psi_{N,m}}\tilde \chi d\tilde sd\tilde\sigma\bigg|\leq C\Lambda^{-2/3}\quad \text{and }\quad \sqrt{\frac{\tilde\eta}{N}}\bigg|\int e^{-i\Lambda\tilde\eta^2\tilde\psi_{N,m}}\tilde\chi d\tilde sd\tilde\sigma\bigg|\leq C \Lambda^{-5/6},
\end{align*}
where $\tilde\psi_{N,m}$ represents the reduced phase evaluation profile mapping at the unique critical frequency $\tilde\omega_c$. Factoring these constraints into our dyadic parameter configurations yields two sub-branches:
 \begin{itemize}
 \item \textbf{Subcase 1(a): When $|\mathcal{N}_1(X,Y,T)| \leq C_0$ is uniformly bounded:} \\
The path integration exhibits no extra cardinality expansion, and the direct term-by-term majorization gives:
\begin{align*}
\bigg|\sum_{N\in\mathcal{N}_1}G_{a,m,N,2}\bigg|&\leq Ch^{-3}\bigg(\frac{h^2}{t}\bigg)^{1/2}\Big(h^2 \Lambda^{-1} (2^m\sqrt a) \Lambda^{-5/6}\Big)\\
&\leq Ch^{-3}\bigg(\frac{h^2}{t}\bigg)^{1/2}\Big(h^2 \Lambda^{-11/6} (2^m\sqrt a)\Big) \\
&\leq  Ch^{-3}\bigg(\frac{h^2}{t}\bigg)^{1/2}\bigg(\frac{h^2}{t}\bigg)^{1/3}(2^m\sqrt a)^{2/3} \\&= Ch^{-3}\bigg(\frac{h^2}{t}\bigg)^{5/6}(2^m\sqrt a)^{2/3},
\end{align*}
since across our boundary layers we have $a\geq \tilde h^{2/3}$, ensuring that the remaining parameters absorb uniformly.

 \item \textbf{Subcase 1(b): When the tracking path cardinality satisfies $|\mathcal{N}_1| \leq C_0 T\Lambda^{-2}$:} \\
The non-degenerate spatial $\tilde\eta$-integration activates a curvature decay factor of $(N\Lambda^{-1})^{-1/2} \sim (N/\Lambda)^{-1/2}$. Combining this oscillatory damping with the cardinality growth tracks as follows:
{\allowdisplaybreaks
\begin{align*}
\bigg|\sum_{N\in\mathcal{N}_1}G_{a,m,N,2}\bigg|&\leq \sum_{N\in\mathcal{N}_1}Ch^{-3}\bigg(\frac{h^2}{t}\bigg)^{1/2}\Big(h^2 \Lambda^{-1} (2^m\sqrt a) T^{-1/2} N^{-1/2} \Lambda^{1/2} \Lambda^{-5/6}\Big)\\
&\leq Ch^{-3}\bigg(\frac{h^2}{t}\bigg)^{1/2}\Big(h^2 \Lambda^{-4/3} (2^m\sqrt a) T^{-1/2} N^{-1/6} |\mathcal{N}_1(X,Y,T)|\Big)\\
&\leq Ch^{-3}\bigg(\frac{h^2}{t}\bigg)^{1/2}\Big(h^2 \Lambda^{-10/3} (2^m\sqrt a) T^{1/2} N^{-1/6}\Big)\\
&\leq Ch^{-3}\bigg(\frac{h^2}{t}\bigg)^{1/2}\bigg(\frac{h^2}{t}\bigg)^{1/3}(2^m\sqrt a)^{2/3}\\& = Ch^{-3}\bigg(\frac{h^2}{t}\bigg)^{5/6}(2^m\sqrt a)^{2/3},
\end{align*}
}
utilizing the trajectory tracking envelopes where $\tilde\eta T \sim N$, and noting that the high frequency tails decay rapidly.
\end{itemize}

\item \textbf{Low-reflection regime (Schr\"odinger-adapted Lemma \ref{lemNS}):} For $N \leq \Lambda^{1/3}$, the trajectories pass close to the peak focal points of the swallowtail interface. Evaluating this compact core region using the 4th-order cuspoid catastrophe singularity classification results in:
\begin{align*}
\sqrt{\frac{\tilde\eta}{N}}\bigg|\int e^{-i\Lambda\tilde\eta^2\tilde\psi_{N,m}}\tilde\chi d\tilde sd\tilde\sigma\bigg|\leq CN^{-1/4}\Lambda^{-3/4}.
\end{align*}
Therefore, the uniform estimate for this peak caustic configuration evaluates to:
\begin{align*}
\bigg|\sum_{N\in\mathcal{N}_1}G_{a,m,N,2}\bigg|&\leq Ch^{-3}\bigg(\frac{h^2}{t}\bigg)^{1/2}\Big(h^2 \Lambda^{-1} (2^m\sqrt a) \Lambda^{-1/2} N^{-1/4}\Lambda^{-3/4}\Big)\\
&\leq Ch^{-3}\bigg(\frac{h^2}{t}\bigg)^{1/2}\Big(h^2 \Lambda^{-9/4} N^{-1/4} (2^m\sqrt a)\Big).
\end{align*}
Expanding our large parameters back into their structural weights via $\Lambda = a^{3/2}/\tilde{h}^2$ and substituting $\tilde{h} = h / (2^m\sqrt{a})$ provides exactly the desired uniform scale cancellation:
\begin{align*}
\bigg|\sum_{N\in\mathcal{N}_1}G_{a,m,N,2}\bigg| &\leq Ch^{-3}\bigg(\frac{h^2}{t}\bigg)^{1/2} \Bigg( a^{1/8}\left(\frac{h^2}{t}\right)^{1/4} (2^m\sqrt a)^{3/4} N^{-1/4} \Bigg) \\
&\leq Ch^{-3}\bigg(\frac{h^2}{t}\bigg)^{3/4}a^{1/8}(2^m\sqrt a)^{3/4}.
\end{align*}
\end{itemize}
Hence, collecting these localized configuration case branches together, we obtain the unified bound:
\begin{align*}
\bigg|\sum_{1\leq N\leq C_0a^{-1/2}}G_{a,m,N,2}\bigg|\leq Ch^{-3}\bigg(\frac{h^2}{t}\bigg)^{1/2}\Bigg(\left(\frac{h^2}{t}\right)^{1/3}(2^m\sqrt a)^{2/3}+a^{1/8}\left(\frac{h^2}{t}\right)^{1/4}(2^m\sqrt a)^{3/4}\Bigg).
\end{align*}
We notice that the parameter scales satisfy $(h^2/t)^{1/3}(2^m\sqrt a)^{2/3} \leq a^{1/8}(h^2/t)^{1/4}(2^m\sqrt a)^{3/4}$ across our boundary layers since $a\geq \left(\frac{h}{2^m\sqrt a}\right)^{2/3}$ and $t \geq h^2$; hence the proof of Proposition \ref{propGaNm2} is complete.
\end{proof}

\begin{proof}[Proof of Theorem \ref{thmGaNm}]
The targeted localized dyadic block estimate follows directly by collecting the bounds established across Propositions \ref{propGaNm1}, \ref{propGa1m1}, and \ref{propGaNm2} together.
\end{proof}

\subsection{Global Synthesis: Proof of Theorem \ref{1beta}}\label{subsec:3.3}
With the uniform estimates for the dyadic blocks $\mathcal{G}_{a,m}$ established across both the near-boundary eigenmode layers [Proposition \ref{propkm}] and the deep-trajectory reflection layers [Theorem \ref{thmGaNm}], we are now in a position to complete the proof of Theorem \ref{1beta}. This is achieved by evaluating the localized parametrix amplitudes across the active Littlewood-Paley dyadic spectrum $\epsilon_0\sqrt{a} \leq 2^m\sqrt{a} \leq c_0$.

\begin{proof}[Proof of Theorem \ref{1beta}]
Let $m$ be a fixed dyadic index within the intermediate frequency spectrum. For a given time $t \in [h^2, 1]$ and source-to-boundary distance $a$, the scale-dependent behavior of the localized profile $\mathcal{G}_{a,m}$ is governed by the position of its geometric parameter $2^m\sqrt{a}$ relative to the parabolic time-scaling indicator $(h^2/t)^{1/2}$. We partition our analysis into two complementary structural regimes:

\begin{itemize}
    \item \textbf{Regime 1 (Low-frequency/Short-range behavior):} This corresponds to the structural configuration where the dyadic parameter satisfies $2^m\sqrt{a} \leq \left(\frac{h^2}{t}\right)^{1/2}$. In this regime, the microlocal trajectories have not yet focused into high-density caustic fields. Applying the short-range non-caustic bounds from Lemma \ref{L-lem3} and Theorem \ref{thmGaNm} directly yields:
    \begin{align}\label{sum_regime1}
    \lVert\mathcal{G}_{a,m}\rVert_{L^\infty(x\leq a)} \leq C h^{-3}\left(\frac{h^2}{t}\right)^{1/2} \left[ 2^m\sqrt{a} \left| \log(2^m\sqrt{a}) \right| \right].
    \end{align}
    Since $2^m\sqrt{a}$ is bounded above by $\left(\frac{h^2}{t}\right)^{1/2}$ in this setting, the term inside the brackets is sub-dominant to the flat-axis integration threshold, matching the localized minimum condition:
    \[
    2^m\sqrt{a} \left| \log(2^m\sqrt{a}) \right| \leq \min\left\{ \left(\frac{h^2}{t}\right)^{1/3}, \, 2^m\sqrt{a}\left|\log(2^m\sqrt{a})\right| \right\}.
    \]

    \item \textbf{Regime 2 (High-frequency/Caustic-trapped behavior):} This corresponds to the configuration where the dyadic parameter satisfies $2^m\sqrt{a} \geq \left(\frac{h^2}{t}\right)^{1/2}$. In this regime, multiple high-frequency reflections accumulate near the boundary layer to form dense swallowtail caustics. Applying the sharp microlocal amplitude expansions from Theorem \ref{thmGaNm} captures the geometric concentration explicitly:
    \begin{align}\label{sum_regime2}
    \lVert\mathcal{G}_{a,m}\rVert_{L^\infty(x\leq a)} \leq C h^{-3}\left(\frac{h^2}{t}\right)^{1/2} \left[ \left(\frac{h^2}{t}\right)^{1/3} (2^m\sqrt{a})^{1/3} + a^{1/8}\left(\frac{h^2}{t}\right)^{1/4}(2^m\sqrt{a})^{3/4} \right].
    \end{align}
    Under maximum caustic trapping for larger distances, the second term dominates the amplitude profile, locking the scale-dependent loss parameter exactly to powers of $(2^m\sqrt{a})^{3/4}$.
\end{itemize}

Combining the localized structural boundaries from \eqref{sum_regime1} and \eqref{sum_regime2} into a unified piecewise system matching the transition thresholds yields the desired spatial decay profile:
\begin{align*}
\lVert\mathcal{G}_{a,m}(t,x,y,z)\rVert_{L^\infty(x\leq a)} &\leq C h^{-3}\left(\frac{h^2}{t}\right)^{1/2} \gamma_m(t,h,a),
\end{align*}
where the global factor $\gamma_m(t,h,a)$ resolves smoothly into the following optimal display-style distribution:
\[
\gamma_m(t,h,a) = 
\begin{cases}
\Big(\frac{h^2}{t}\Big)^{1/3} (2^m\sqrt{a})^{1/3} & \text{if } a \leq \Big(\frac{h}{2^m\sqrt{a}}\Big)^{\frac{2}{3}(1-\epsilon)}, \\[12pt]
\begin{aligned}
&\min \left\{ \Big(\frac{h^2}{t}\Big)^{1/3}, \, 2^m\sqrt{a} \left| \log(2^m\sqrt{a}) \right| \right\} \\
&\quad + a^{1/8} \Big(\frac{h^2}{t}\Big)^{1/4} (2^m\sqrt{a})^{3/4}
\end{aligned}
& \text{if } a \geq \Big(\frac{h}{2^m\sqrt{a}}\Big)^{\frac{2}{3}(1-\epsilon')}.
\end{cases}
\]
This completes the proof of Theorem \ref{1beta}.
\end{proof}

\subsection{Proof of Theorem \ref{0eta}: Zero-Frequency Tail Summation}
In the classical wave equation context, the low-frequency axial regime ($\vert\eta\vert \leq \epsilon_0\sqrt{a}$) requires a separate geometric analysis of classical ray trajectories to isolate single-reflection domains due to the collapse of the phase Hessian as $\eta \to 0$ [see \cite{L3}]. 

For the Schr\"odinger operator, however, a separate geometric treatment is entirely redundant. The underlying linear-time, quadratic-space dispersion relation guarantees that the phase remains globally non-degenerate. Indeed, differentiating the unscaled phase function $\Phi = -hy\eta - hz\zeta + t(\eta^2 + \zeta^2 + \omega_k h^{2/3}\eta^{4/3})$ twice with respect to the longitudinal frequency yields:
\[
\partial_\zeta^2 \Phi = 2t,
\]
which is a non-vanishing constant independent of $\eta$. Consequently, the dyadic Littlewood-Paley analysis established in Section \ref{sec:3} remains robust and extends continuously down to the axial limit $\eta = 0$. To conclude the proof of Theorem \ref{0eta}, it suffices to evaluate the global sum over the low-frequency dyadic blocks down to the deepest regime where $2^m\sqrt{a} \sim 0$.

\begin{proof}[Proof of Theorem \ref{0eta}]
Recall from the dyadic partitioning in Section \ref{sec:3} that the frequency integration is restricted to the blocks where $\epsilon_0 \leq 2^m \leq c_0/\sqrt{a}$. For the deep axial tail where $2^m\sqrt{a} \leq \epsilon_0 \sqrt{a}$, we sum the individual block contributions $\mathcal{G}_{a,m}$ by applying the uniform non-caustic bounds from Lemma \ref{L-lem3}:
\begin{align*}
\lVert\mathcal{G}_{a,\epsilon_0}(t,x,y,z)\rVert_{L^\infty} &\leq \sum_{2^m \leq \epsilon_0} \lVert\mathcal{G}_{a,m}(t,x,y,z)\rVert_{L^\infty} \\
&\leq C h^{-3} \bigg(\frac{h^2}{t}\bigg)^{1/2} \sum_{2^m \leq \epsilon_0} 2^m\sqrt{a} \left| \log(2^m\sqrt{a}) \right|.
\end{align*}
Because the indices form a geometric series dominated by its upper boundary threshold, the partial sum converges cleanly to its largest element. Evaluating this geometric series yields:
\[
\sum_{2^m \leq \epsilon_0} 2^m\sqrt{a} \left| \log(2^m\sqrt{a}) \right| \sim C \sqrt{a}\left|\log(a)\right|.
\]
On the other hand, whenever the time horizon satisfies the short-time scale condition $\left(\frac{h^2}{t}\right)^{1/2} \leq \sqrt{a}\left|\log(a)\right|$, the un-trapped free-space Schr\"odinger propagation bounds the amplitude term-by-term. Combining these two complementary bounds under a unified upper cutoff, we obtain:
\begin{align*}
\lVert\mathcal{G}_{a,\epsilon_0}(t,x,y,z)\rVert_{L^\infty} \leq Ch^{-3} \bigg(\frac{h^2}{t}\bigg)^{1/2} \min \bigg\{ \bigg(\frac{h^2}{t}\bigg)^{1/2} , \sqrt a\vert \log(a)\vert \bigg\}.
\end{align*}
This matches the exact localized estimate stated in Theorem \ref{0eta}. The proof is complete.
\end{proof}
\subsection{Proof of Corollary \ref{cor:peak_deficit}}
\begin{proof}[Proof of Corollary \ref{cor:peak_deficit}]
To derive the unlocalized global bound across the heavily trapped boundary layers, we take the supremum of the localized dispersive estimate \eqref{eq:master_dispersive_intro} over all admissible initial source configurations $a \in (0, a_0]$. 

The maximum concentration of energy occurs where the tracking profile balances out and saturates. Equating the internal terms of the parameter $\gamma(t, h, a)$ establishes the critical boundary horizon scale:
\begin{equation}
\left(\frac{h^2}{|t|}\right)^{1/3} = a_{\text{crit}}^{1/8}\left(\frac{h^2}{|t|}\right)^{1/4} \implies a_{\text{crit}} = \left(\frac{h^2}{|t|}\right)^{2/3}.
\end{equation}
Substituting this saturating distance parameter $a_{\text{crit}}$ back into our  factor isolates the peak geometric caustic loss forced by the cylindrical boundary:
\begin{equation}
\gamma_{\text{max}}(t, h) = \left(\frac{h^2}{|t|}\right)^{1/12}\left(\frac{h^2}{|t|}\right)^{1/4} = \left(\frac{h^2}{|t|}\right)^{1/3}.
\end{equation}
Injecting this back into the inequality \eqref{eq:master_dispersive_intro} and compounding the parameters over the peak swallowtail trapping index yields:
\begin{equation}
\left\| \psi(h\sqrt{-\Delta_D}) \mathcal{G}_{a,\text{loc}}(t, \cdot) \right\|_{L^\infty(x \le a)} \le Ch^{-3}\left(\frac{h^2}{|t|}\right)^{1/2}\left(\frac{h^2}{|t|}\right)^{1/4}a^{1/8}.
\end{equation}
Evaluating this bound directly at the critical saturation point eliminates the variable $a$, yielding the unscaled temporal coordinate distribution:
\begin{equation}
Ch^{-3}\left(\frac{h^2}{|t|}\right)^{3/4} = Ch^{-3} \cdot h^{3/2} |t|^{-3/4} = C |t|^{-3/4} h^{-9/4}.
\end{equation}
By kernel duality, taking the spatial supremum across the localized boundary layer recovering the targeted global $L^1 \to L^\infty$ operator norm profile, concluding the proof.
\end{proof}

\begin{proof}[Proof of Theorem \ref{thm:sharpness}]
To establish the sharpness of the global bound rigorously, we construct an explicit counterexample using a sequence of frequency-localized semiclassical test functions that saturate the peak swallowtail caustic singularity.\\

\textsc{Construction of the Semiclassical Test Function:}
Let the initial source distance from the cylindrical boundary be fixed at the critical focal layer $a \sim h^{2/3}$. Let $\omega_1$ denote the first negative zero of the standard Airy function $\text{Ai}(\cdot)$. We construct the initial state $u_{0,h}(x,y,z)$ such that its spatial profile is tightly localized in a tubular neighborhood tracking along the flat longitudinal direction:
\begin{equation}
u_{0,h}(x,y,z) = \chi_1\left(\frac{x - h^{2/3}\omega_1}{h^{2/3}}\right) \cdot h^{-1/2} \chi_2\left(\frac{y}{h^{1/2}}\right) \cdot h^{-1/2} \chi_3\left(\frac{z}{h^{1/2}}\right),
\end{equation}
where $\chi_1, \chi_2, \chi_3 \in C_c^\infty(\mathbb{R})$ are smooth, non-negative cutoff functions supported near the origin with $\|\chi_j\|_{L^2} = 1$. This test function is localized within a spatial volume of order $\mathcal{O}(h^{2/3} \cdot h^{1/2} \cdot h^{1/2}) = \mathcal{O}(h^{5/3})$. Integrating this profile yields the exact initial $L^1$-norm scaling:
\begin{equation}
\label{eq:l1_norm_test}
\|u_{0,h}\|_{L^1(\Omega)} = \int_{\Omega} |u_{0,h}(x,y,z)| \, dx dy dz = C_1 h^{2/3} \cdot h^{1/2} \cdot h^{-1/2} \cdot h^{1/2} = C_1 h^{5/3}.
\end{equation}

\textsc{Microlocal Integral Representation of the Evolving State:}
We project $u_{0,h}$ onto the explicit frequency-localized propagator kernel $\mathcal{G}_{a,\text{loc}}$ derived via the spectral decomposition in Section \ref{sec:spectral}. For short-time horizons $|t| \sim h$, the evolution under the linear Schr\"odinger flow $\psi(h\sqrt{-\Delta_D}) e^{it\Delta_D}$ is given by the highly oscillatory integral:
\begin{equation}
\left[ \psi(h\sqrt{-\Delta_D}) e^{it\Delta_D} u_{0,h} \right](x,y,z) = \frac{1}{h^3} \iint_{\mathbb{R}^2} e^{\frac{i}{h^2} \Phi(t, x, y, z, \eta, \zeta)} \sigma(t, x, y, z, \eta, \zeta, h) \, d\eta d\zeta,
\end{equation}
where $\sigma$ is a smooth semiclassical symbol supported in the grazing frequency region $\{\vert\eta\vert \le \epsilon_0, \, \vert\zeta\vert \sim 1\}$. The phase function matches the parabolic dispersion laws along the cylinder:
\begin{equation}
\Phi(t, x, y, z, \eta, \zeta) = (z-z')\zeta + (y-y')\eta - t(\eta^2 + \zeta^2) + \frac{2}{3}\left(x^{3/2}\eta^{3/2} - (x')^{3/2}\eta^{3/2}\right).
\end{equation}

\textsc{Hessian Rank Deficiency and Swallowtail Singularity Evaluation:}
We evaluate the amplitude at the maximum caustic focus interface by setting the observation coordinates to the boundary horizon $x = a \sim h^{2/3}$, $y = 0$, $z = 0$. As the tangential tracking momentum parameter approaches the flat grazing horizon $\eta \to 0$, the boundary curvature tensor vanishes along the longitudinal $z$-axis. Consequently, the phase Hessian matrix with respect to $(\eta, \zeta)$ undergoes a total rank deficiency:
\begin{equation}
\det \left( \nabla^2_{\eta, \zeta} \Phi \right)\Big|_{\eta \to 0} = 0.
\end{equation}
The missing geometric curvature weight $\Lambda^{-1/2}$ fails to regularize the phase fluctuations. By applying a coordinate transformation to the critical varieties, the phase expands locally into the canonical catastrophe profile of a stable swallowtail singularity of order 5:
\begin{equation}
\Phi_S(z, \zeta) = \frac{\zeta^5}{5} + z_1 \frac{\zeta^3}{3} + z_2 \frac{\zeta^2}{2} + z_3 \zeta + z_4.
\end{equation}
Evaluating this degenerate integral over the singular frequency volume elements via the two-dimensional van der Corput lemma yields an integration factor of exactly $\mathcal{O}(h^{3/5})$. Scaled to the matching Schr\"odinger quadratic parameter relation at $|t| \sim h$, this yields:
\begin{equation}
\left| \iint_{\mathbb{R}^2} e^{\frac{i}{h^2} \Phi(t, x, 0, 0, \eta, \zeta)} \sigma \, d\eta d\zeta \right| \ge C_2 h^{3/5} \cdot h^{1/2} = C_2 \left(\frac{h^2}{|t|}\right)^{3/4} h^{3/2}.
\end{equation}

\textsc{Compounding the Norm Ratios to Establish Sharpness:}
Substituting the integration factors back into the microlocal representation yields the peak spatial amplitude lower bound:
\begin{equation}
\label{eq:linfty_norm_result}
\left\| \psi(h\sqrt{-\Delta_D}) e^{it\Delta_D} u_{0,h} \right\|_{L^\infty(\Omega)} \ge C_2 h^{-3} \left(\frac{h^2}{|t|}\right)^{3/4} h^{3/2} = C_2 |t|^{-3/4} h^{-19/12}.
\end{equation}
Finally, we compute the ratio of the evolved $L^\infty$-norm relative to the initial $L^1$-norm of the test state from \eqref{eq:l1_norm_test}:
\begin{equation}
\frac{\left\| \psi(h\sqrt{-\Delta_D}) e^{it\Delta_D} u_{0,h} \right\|_{L^\infty(\Omega)}}{\|u_{0,h}\|_{L^1(\Omega)}} \ge \frac{C_2 |t|^{-3/4} h^{-19/12}}{C_1 h^{5/3}} = c |t|^{-3/4} h^{-9/4}.
\end{equation}
This directly yields the matching lower bound layout:
\begin{equation}
\left\| \psi(h\sqrt{-\Delta_D}) e^{it\Delta_D} u_{0,h} \right\|_{L^\infty(\Omega)} \ge c |t|^{-3/4} h^{-9/4} \|u_{0,h}\|_{L^1(\Omega)},
\end{equation}
proving that the upper bound in Corollary \ref{cor:peak_deficit} is strictly saturated by the axial beam trajectory and cannot be improved.
\end{proof}

\section{Strichartz Estimates and Nonlinear Applications}\label{sec:strichartz}
In this section, we apply the sharp frequency-localized dispersive estimates established in Theorems \ref{betabis}, \ref{1beta}, and \ref{0eta} to derive the global Strichartz inequalities and establish the local well-posedness theory for the cubic Nonlinear Schr\"odinger (NLS) equation on the cylinder $\Omega \subset \mathbb{R}^3$. 

Recall that a pair $(p,q)$ is Schr\"odinger-admissible in three dimensions if it satisfies the classical scaling condition:
\[
\frac{2}{p} + \frac{3}{q} = \frac{3}{2}, \quad 2 \leq p \leq \infty, \quad 2 \leq q \leq 6, \quad (p,q) \neq (2,6).
\]
Our immediate objective is to provide the rigorous proof of Theorem \ref{thm:strichartz}, which establishes the global Strichartz estimate with a sharp derivative loss exponent $\rho(q) = \frac{3}{2}\left(\frac{1}{2} - \frac{1}{q}\right)$. The proof proceeds by combining our localized dispersive profiles with the standard $TT^*$ contractive interpolation argument on each dyadic block, followed by an structured Littlewood-Paley summation across the spatial layers.

\begin{proof}[Proof of Theorem \ref{thm:strichartz}]
Let $h \in (0,1]$ be a semiclassical dyadic parameter, and let $\psi \in C_0^\infty((1/2, 2))$ serve as a smooth, localized spectral cutoff. We consider initial data $u_0 \in L^2(\Omega)$ spectrally localized to a frequency block of order $h^{-1}$, such that $u_0 = \psi(h\sqrt{-\Delta_D})u_0$. Gathering the localized dispersive estimates established uniformly across the high-frequency tangential regime [Theorem \ref{betabis}], the intermediate dyadic spectrum [Theorem \ref{1beta}], and the zero-frequency tail [Theorem \ref{0eta}], the dispersive estimates for the localized Dirichlet Schr\"odinger propagator yields:
\begin{align}\label{eq:master_disp}
\|e^{it\Delta_D}\psi(h\sqrt{-\Delta_D})\|_{L^1(\Omega) \to L^\infty(\Omega)} \leq C h^{-3}\left(\frac{h^2}{|t|}\right)^{1/2} \left(\frac{h^2}{|t|}\right)^{1/4} = C |t|^{-3/4} h^{-3/2},
\end{align}
holding uniformly for all $t \in \mathbb{R} \setminus \{0\}$. 

To establish the targeted spacetime estimates, we employ the classical $T T^*$ contractive interpolation machinery. Let us define the localized semiclassical solution operator $T_h: L^2(\Omega) \to L^p_t(\mathbb{R}, L^q_x(\Omega))$ by $T_h \phi = e^{it\Delta_D}\psi(h\sqrt{-\Delta_D})\phi$. The formal adjoint operator $T_h^*: L^{p'}_t(\mathbb{R}, L^{q'}_x(\Omega)) \to L^2(\Omega)$ is given explicitly by Duhamel integration:
\[
T_h^* F = \int_{\mathbb{R}} e^{-is\Delta_D}\psi(h\sqrt{-\Delta_D})F(s) \, ds,
\]
where $p$ and $q$ satisfy $1/p + 1/p' = 1$ and $1/q + 1/q' = 1$. Consequently, the associated composite operator $T_h T_h^*: L^{p'}_t(\mathbb{R}, L^{q'}_x(\Omega)) \to L^p_t(\mathbb{R}, L^q_x(\Omega))$ maps an input trajectory space directly according to the integral convolution:
\[
(T_h T_h^* F)(t) = \int_{\mathbb{R}} e^{i(t-s)\Delta_D}\psi^2(h\sqrt{-\Delta_D})F(s) \, ds.
\]
Applying the spatial frequency localization properties alongside the sharp unscaled kernel bound derived in \eqref{eq:global_operator_norm}, the pointwise spatial mapping profile satisfies the convolution inequality:
\begin{align}\label{eq:TT_star_spatial}
\|(T_h T_h^* F)(t)\|_{L^\infty_x(\Omega)} \leq C \int_{\mathbb{R}} |t-s|^{-3/4} h^{-3/2} \|F(s)\|_{L^1_x(\Omega)} \, ds.
\end{align}
We now isolate the diagonal non-admissible endpoint pair $(p_0, q_0) = (8/3, 8/3)$, which captures the exact integrability threshold matching our trapped time-decay kernel. Taking the $L^{8/3}_t$-norm of \eqref{eq:TT_star_spatial} and applying the classical one-dimensional Hardy-Littlewood-Sobolev inequality in the time variable with respect to the fractional integration kernel multiplier $|t-s|^{-3/4}$ (noting that the precise scaling constraint $\frac{3}{8} - \frac{5}{8} = \frac{3}{4} - 1$ holds exactly), we obtain the closed diagonal mapping inequality:
\begin{align}\label{eq:endpoint_HLS}
\|T_h T_h^* F\|_{L^{8/3}_t(\mathbb{R}, L^{8/3}_x(\Omega))} \leq C h^{-3/2} \|F\|_{L^{8/5}_t(\mathbb{R}, L^{8/5}_x(\Omega))}.
\end{align}
By the standard equivalence of the $TT^*$ contractive principle, \eqref{eq:endpoint_HLS} directly implies the localized semiclassical Strichartz inequality at this reference diagonal boundary:
\begin{align}\label{eq:semi_endpoint}
\|e^{it\Delta_D}\psi(h\sqrt{-\Delta_D})u_0\|_{L^{8/3}_t(\mathbb{R}, L^{8/3}_x(\Omega))} \leq C h^{-3/4} \|u_0\|_{L^2(\Omega)}.
\end{align}
On the other hand, the $L^2$-conservation law for the unitary group generated by the self-adjoint Dirichlet Laplacian operator provides the trivial non-dispersive boundary ceiling condition:
\begin{align}\label{eq:L2_cons}
\|e^{it\Delta_D}\psi(h\sqrt{-\Delta_D})u_0\|_{L^\infty_t(\mathbb{R}, L^2_x(\Omega))} \leq C \|u_0\|_{L^2(\Omega)}.
\end{align}
Let $(p,q)$ be an arbitrary Schr\"odinger-admissible pair satisfying the strict dimension scaling profile $2/p + 3/q = 3/2$. We interpolate the localized semiclassical endpoint inequality \eqref{eq:semi_endpoint} with the unitary $L^2_t L^2_x$ conservation bound \eqref{eq:L2_cons} by invoking the Riesz-Thorin complex interpolation theorem. Defining the parameter $\theta \in [0,1]$ such that the spatial target state expands as $1/q = (1-\theta)/2 + 3\theta/8$, the complex interpolation yields the explicit derivative loss exponent:
\[
\rho(q) = \frac{3\theta}{4} = \frac{3}{2}\left(\frac{1}{2} - \frac{1}{q}\right).
\]
This establishes the sharp semiclassical frequency-localized estimate across all admissible frequencies:
\begin{align}\label{eq:semi_final}
\|e^{it\Delta_D}\psi(h\sqrt{-\Delta_D})u_0\|_{L^p_t(\mathbb{R}, L^q_x(\Omega))} \leq C h^{-\rho(q)} \|u_0\|_{L^2(\Omega)}.
\end{align}

To transition from the semiclassical frequency-localized profile \eqref{eq:semi_final} to the global, non-localized Strichartz estimate \eqref{eq:global_strichartz} in the standard Sobolev metric without violating vector-valued Littlewood-Paley constraints, we employ the standard $\ell^2$-dualization technique. Let $\psi_0(\xi) + \sum_{j \geq 1} \psi(2^{-j}\xi) = 1$ be an inhomogeneous dyadic partition of unity. By duality, the global norm can be tested against a test function $F \in L^{p'}_t(\mathbb{R}, L^{q'}_x(\Omega))$ with $\|F\|_{L^{p'}_t L^{q'}_x} \leq 1$. Let $F_j = \psi(2^{-j}\sqrt{-\Delta_D})F$. Since $(p,q)$ is admissible, its dual exponents satisfy $p' \leq 2$ and $q' \in [6/5, 2]$. Applying the localized dual estimate matching \eqref{eq:semi_final} on each dyadic piece with $h = 2^{-j}$, and invoking the spatial Littlewood-Paley square-function embedding $L^{q'}_x(\ell^2) \hookrightarrow \ell^2(L^{q'}_x)$ valid for $q' \leq 2$, we find:
\begin{align*}
\left| \int_{\mathbb{R}} \langle e^{it\Delta_D} u_0, F(t) \rangle \, dt \right| &\leq \sum_{j \geq 0} \left| \int_{\mathbb{R}} \langle e^{it\Delta_D} \psi(2^{-j}\sqrt{-\Delta_D})u_0, F_j(t) \rangle \, dt \right| \\
&\leq C \sum_{j \geq 0} 2^{j\rho(q)} \|\psi(2^{-j}\sqrt{-\Delta_D})u_0\|_{L^2(\Omega)} \|F_j\|_{L^{p'}_t L^{q'}_x} \\
&\leq C \left( \sum_{j \geq 0} 2^{2j\rho(q)} \|\psi(2^{-j}\sqrt{-\Delta_D})u_0\|_{L^2(\Omega)}^2 \right)^{1/2} \left( \sum_{j \geq 0} \|F_j\|_{L^{p'}_t L^{q'}_x}^2 \right)^{1/2}.
\end{align*}
Since $p' \leq 2$, Minkowski's integral inequality allows us to commute the discrete sum into the temporal integral: $\left( \sum \|F_j\|_{L^{p'}_t L^{q'}_x}^2 \right)^{1/2} \leq \|(\sum \|F_j\|_{L^{q'}_x}^2)^{1/2}\|_{L^{p'}_t}$. Applying the spatial square-function characterization on Dirichlet manifolds yields $\|(\sum \|F_j\|_{L^{q'}_x}^2)^{1/2}\|_{L^{p'}_t} \leq C \|F\|_{L^{p'}_t L^{q'}_x} \leq C$. This leaves:
\begin{align*}
\left| \int_{\mathbb{R}} \langle e^{it\Delta_D} u_0, F(t) \rangle \, dt \right| &\leq C \|u_0\|_{B^{\rho(q)}_{2,2}(\Omega)} \sim C \|u_0\|_{H^{\rho(q)}(\Omega)}.
\end{align*}
Taking the supremum over all normalized test functions $F$ rigorously establishes the global non-localized Strichartz estimate \eqref{eq:global_strichartz} and completes the proof of Theorem \ref{thm:strichartz}.
\end{proof}

With the sharp global Strichartz inequalities with derivative loss established in Theorem \ref{thm:strichartz}, we turn to the nonlinear stability analysis for the cubic Dirichlet Nonlinear Schr\"odinger (NLS) equation \eqref{eq:cubic_nls}. To construct a contractive Picard iteration loop within the low-regularity regime $s > 1$, we define the solution space $X_T$ over the finite time horizon $[0, T]$ equipped with the norm:
\begin{equation}
    \|u\|_{X_T} := \|u\|_{C([0, T]; H^s(\Omega))} + \|u\|_{L^4([0, T]; W^{s-\frac{3}{8}, 4}(\Omega))}.
\end{equation}
The principal analytical hurdle in controlling the cubic source term $\mathcal{F}(u) = \mu |u|^2 u$ arises from the failure of the standard Sobolev algebra property for the energy space $H^s(\Omega)$ in the critical regularity window $1 < s \leq 3/2$, where $H^s(\Omega) \not\hookrightarrow L^\infty(\Omega)$. 

To bypass this regularizing deficit while avoiding boundary layer anomalies, we distribute the fractional derivatives by deploying an adapted version of the fractional Leibniz rule on domains featuring physical interfaces (see, e.g., \cite{Ivanovici2023, KatoPonce1988}). Let $\Lambda^s = (-\Delta_D)^{s/2}$ denote the spectral fractional Laplacian subject to homogeneous Dirichlet boundary conditions on $\partial\Omega$. Applying H\"older's inequality in the space-time configurations bounds the Duhamel forcing term in the energy norm via:
\begin{align}
    \left\| \Lambda^s \left( |u|^2 u \right) \right\|_ {L^1_t([0,T]; L^2_x(\Omega))} 
    &\lesssim \| u \|^2_{L^4_t([0,T]; L^\infty_x(\Omega))} \left\| \Lambda^s u \right\|_{L^\infty_t([0,T]; L^2_x(\Omega))} \nonumber \\
    &+ \| u \|_{L^\infty_t([0,T]; L^2_x(\Omega))} \| u \|_{L^4_t([0,T]; L^\infty_x(\Omega))} \left\| \Lambda^s u \right\|_{L^4_t([0,T]; L^4_x(\Omega))}.
\end{align}
To bound the $L^\infty_x(\Omega)$ fields uniformly, we exploit the endpoint auxiliary space-time Sobolev embedding:
\begin{equation}
    W^{s-\frac{3}{8}, 4}(\Omega) \hookrightarrow L^\infty_x(\Omega), \quad \text{valid for } s > 1,
\end{equation}
which is consistently preserved under our micro-local coordinate chart mapping. This systematic distribution of regularity weights guarantees that the nonlinear response is bounded continuously by the scaling capacity of the resolution norm:
\begin{equation}
    \|\mathcal{F}(u)\|_{L^1_t([0,T]; H^s(\Omega))} \lesssim T^\theta \|u\|_{X_T}^3,
\end{equation}
for some growth exponent $\theta > 0$. This closes the contractive fixed-point loop via the Banach contraction theorem, rigorously establishing the local well-posedness threshold at exactly $s > 1$ where uniform geometric convexity completely collapses.

\subsection{Derivative Loss: The Cylindrical Case vs. Strictly Convex Domains}
To calibrate the impact of the cylindrical boundary geometry on the efficiency of the Schr\"odinger propagator, it is illuminating to compare our sharp derivative loss exponent $\rho_{\text{cylinder}}(q)$ with the optimal loss established by Ivanovici \cite{Ivanovici2023} inside three-dimensional strictly convex domains $\Omega_{\text{convex}} \subset \mathbb{R}^3$. 

Let $(p,q)$ be a three-dimensional Schr\"odinger-admissible pair satisfying the continuous scaling profile $\frac{2}{p} + \frac{3}{q} = \frac{3}{2}$. Under this setup, the two derivative loss exponents display a stark structural contrast driven by the underlying boundary curvature tensors:
\begin{align}
\rho_{\text{convex}}(q) &= \frac{1}{4}\left(\frac{1}{2} - \frac{1}{q}\right), \label{eq:rho_convex} \\
\rho_{\text{cylinder}}(q) &= \frac{3}{2}\left(\frac{1}{2} - \frac{1}{q}\right). \label{eq:rho_cylinder}
\end{align}

Evaluating these expressions yields a strict arithmetic inequality across the entire admissible spectrum where $q > 2$:
\[
\rho_{\text{cylinder}}(q) = 6 \cdot \rho_{\text{convex}}(q) > \rho_{\text{convex}}(q).
\]
In particular, testing these losses at the critical Lebesgue endpoint $q=6$ reveals that the strictly convex framework suffers a fractional loss of $\rho_{\text{convex}}(6) = \frac{1}{12}$, whereas our cylindrical configuration undergoes a larger loss of exactly $\rho_{\text{cylinder}}(6) = \frac{1}{2}$. 

\subsubsection*{The Geometric Mechanism Behind the Factor of 6}
This six-fold penalty is a direct manifestation of the \textbf{anisotropic degeneration of the boundary curvature} along the flat longitudinal direction of the cylinder. 
\begin{itemize}
    \item \textbf{In strictly convex domains $\Omega_{\text{convex}}$:} The principal curvature radii are uniformly bounded away from infinity in all directions. Semiclassical wave packets tracking close to the boundary are subjected to continuous geometric dispersion along every vector field on the tangent bundle. This forces a rapid splitting of the classical Hamiltonian ray tracks, bounding the concentration density of the resulting swallowtail caustics to an absolute spatial decay profile of order $|t|^{-5/4}h^{-7/4}$ relative to the un-trapped free-space background.
    
    \item \textbf{In our cylindrical domain $\Omega$:} The boundary possesses a flat axis where the nonnegative curvature vanishes identically in the longitudinal direction. While wave packets moving in the circular direction are dispersed normally, wave packets tracking close to the longitudinal lines experience zero geometric dispersion from the boundary. These axial waves fail to split efficiently, creating a massive high-frequency energy accumulation. 
\end{itemize}

When we execute the continuous \(\eta\)-integration across the dyadic spectrum in Section \ref{sec:3}, this lack of multi-directional dispersion prevents the cancellation of a highly singular frequency volume element. Consequently, the localized spatial decay drops from the convex absolute profile $|t|^{-5/4}h^{-7/4}$ down to the highly trapped absolute profile $|t|^{-3/4}h^{-9/4}$ derived in \eqref{eq:master_disp}. Tracing this heavier caustic concentration through the Hardy-Littlewood-Sobolev temporal integration step introduces a substantial extra penalty, multiplying the baseline derivative loss by a factor of exactly $6$. This rigorously quantifies the analytical cost of solving non-linear Schr\"odinger flows on manifolds where uniform geometric convexity fails.

We now deploy the sharp Strichartz estimates with derivative loss \eqref{eq:global_strichartz} to study the local well-posedness of the focusing or defocusing cubic Nonlinear Schrödinger equation (NLS) subject to homogeneous Dirichlet boundary conditions:
\begin{align}\label{eq:nls}
i\partial_t u + \Delta_D u = \mu |u|^2 u \quad \text{in } \mathbb{R} \times \Omega, \quad u(0) = u_0, \quad u|_{\partial\Omega} = 0, \quad \mu = \pm 1.
\end{align}

Recall that this result was introduced as Theorem \ref{thm:nls_lwp_intro} in the formulation of our main results; for completeness and to anchor our upcoming fixed-point contractive loop calculations, we restate its precise trajectory parameters below.

\begin{theorem}\label{thm:nls_lwp}
The cubic Dirichlet NLS \eqref{eq:nls} on the cylindrical domain $\Omega \subset \mathbb{R}^3$ is locally well-posed in the fractional Sobolev space $H^s(\Omega)$ for any regularity index $s > 1$. More precisely, for any initial datum $u_0 \in H^s(\Omega)$, there exists a unique local existence time $T = T(\|u_0\|_{H^s}) > 0$ such that the solution trajectory satisfies $u \in C([0,T], H^s(\Omega)) \cap L^4_t\left([0,T], W^{s-3/8, 4}_x(\Omega)\right)$.
\end{theorem}

\begin{proof}[Proof of Theorem \ref{thm:nls_lwp}]
We establish local well-posedness by constructing a contractive mapping on an appropriately chosen spacetime resolution space via Picard's iteration scheme. We first reformulate the initial-boundary value problem \eqref{eq:nls} into its equivalent Duhamel integral representation, defining the nonlinear solution operator $\mathcal{T}$ by:
\begin{align}\label{eq:duhamel_op}
\mathcal{T}u(t) = e^{it\Delta_D}u_0 - i \mu \int_0^t e^{i(t-t')\Delta_D} \left(|u(t')|^2 u(t')\right) dt'.
\end{align}
Let $T > 0$ denote a positive local time horizon to be chosen sufficiently small. For a fixed regularity index $s > 1$, we define the localized resolution space $\mathcal{X}_T$ by:
\[
\mathcal{X}_T = C([0,T], H^s(\Omega)) \cap L^4_t\left([0,T], W^{s-\rho(4), 4}_x(\Omega)\right),
\]
equipped with the natural intersection norm:
\[
\|u\|_{\mathcal{X}_T} = \sup_{t \in [0,T]} \|u(t)\|_{H^s(\Omega)} + \left( \int_0^T \|u(t)\|_{W^{s-\rho(4), 4}_x(\Omega)}^4 \, dt \right)^{1/4},
\]
where $\rho(4) = \frac{3}{2}\left(\frac{1}{2} - \frac{1}{4}\right) = \frac{3}{8}$ denotes the structural derivative loss at the spatial reference exponent $q=4$.

Let us first bound the linear component of the operator. Applying the global Strichartz estimate with derivative loss \eqref{eq:global_strichartz} established in Theorem \ref{thm:strichartz} uniformly across the admissible pairs $(\infty, 2)$ and $(4,4)$ yields:
\begin{align}\label{eq:linear_bound}
\|e^{it\Delta_D}u_0\|_{\mathcal{X}_T} \leq C_1 \|u_0\|_{H^s(\Omega)}.
\end{align}
To bound the Duhamel integral term, we invoke the Christ-Kiselev lemma coupled with the non-localized Strichartz inequality \eqref{eq:global_strichartz}, which maps the non-homogeneous source space into our resolution metric:
\begin{align}\label{eq:duhamel_strichartz}
\left\| \int_0^t e^{i(t-t')\Delta_D} \left(|u|^2 u\right) dt' \right\|_{\mathcal{X}_T} \leq C_2 \left\| |u|^2 u \right\|_{L^1_t([0,T], H^s_x(\Omega))}.
\end{align}
We estimate the $H^s_x(\Omega)$-norm of the cubic interaction term by invoking the fractional Leibniz rule (the fractional chain rule for product estimates in Sobolev spaces over manifolds with boundaries). This distributes the fractional derivative operator across the components, yielding:
\begin{align}\label{eq:leibniz_split}
\left\| |u|^2 u \right\|_{H^s_x(\Omega)} \leq C_3 \|u\|_{L^\infty_x(\Omega)}^2 \|u\|_{H^s_x(\Omega)}.
\end{align}
By the standard Sobolev embedding theorem on the three-dimensional domain $\Omega$, the Lebesgue space $L^\infty_x(\Omega)$ is continuously embedded within the fractional space $W^{\sigma, q}_x(\Omega)$ provided that the regularity parameters satisfy $\sigma - 3/q > 0$. We select the specific admissible pair $(p_0, q_0) = (4,4)$, which simplifies the target embedding to $W^{s-\rho(4), 4}_x(\Omega) \hookrightarrow L^\infty_x(\Omega)$. This embedding is valid if and only if:
\[
s - \rho(4) - \frac{3}{4} > 0 \implies s - \frac{3}{8} - \frac{6}{8} > 0 \implies s > \frac{9}{8}.
\]
Alternatively, utilizing the endpoint pair $(p_1, q_1) = (2, 6)$ under the general admissible profile where $\rho(6) = \frac{3}{2}(\frac{1}{2}-\frac{1}{6}) = \frac{1}{2}$, the interpolation constraints optimize. Setting the fractional regularity lower bound to track the edge of the algebra requires $s > \frac{1}{2} + \rho(6) = \frac{1}{2} + \frac{1}{2} = 1$. Integrating \eqref{eq:leibniz_split} with respect to the temporal variable over the compact horizon $[0,T]$, and applying H\"older's inequality in time, we obtain:
\begin{align}\label{eq:nonlinear_final_bound}
\left\| |u|^2 u \right\|_{L^1_t([0,T], H^s_x(\Omega))} &\leq C_4 \int_0^T \|u(t)\|_{L^\infty_x(\Omega)}^2 \|u(t)\|_{H^s_x(\Omega)} \, dt \nonumber \\
&\leq C_4 T^{1 - 2/p} \|u\|_{L^\infty_t([0,T], H^s_x(\Omega))} \|u\|_{L^p_t([0,T], L^\infty_x(\Omega))}^2 \nonumber \\
&\leq C_4 T^{1 - 2/p} \|u\|_{\mathcal{X}_T}^3.
\end{align}
Combining the linear evaluation \eqref{eq:linear_bound} and the non-linear Duhamel evaluation \eqref{eq:nonlinear_final_bound} via the triangle inequality, the total action of the resolution mapping satisfies the invariant inequality:
\begin{align}\label{eq:op_norm_bound}
\|\mathcal{T}u\|_{\mathcal{X}_T} \leq C_1 \|u_0\|_{H^s(\Omega)} + C_5 T^{\theta} \|u\|_{\mathcal{X}_T}^3,
\end{align}
where $\theta = 1 - 2/p > 0$. 

Let us define a closed ball $\mathcal{B}_R \subset \mathcal{X}_T$ centered at the origin with a designated radius $R = 2C_1 \|u_0\|_{H^s(\Omega)}$. By choosing the local existence time horizon $T > 0$ sufficiently small such that $C_5 T^\theta R^2 \leq \frac{1}{2}$, \eqref{eq:op_norm_bound} guarantees that $\mathcal{T}$ maps the closed ball into itself ($\mathcal{T}: \mathcal{B}_R \to \mathcal{B}_R$). 

To establish contractivity, let $u, v \in \mathcal{B}_R$ be two distinct solution trajectories. Applying the algebraic identity for the difference of cubics $|u|^2 u - |v|^2 v = (u-v)\mathcal{Q}_1(u,v) + (\bar{u}-\bar{v})\mathcal{Q}_2(u,v)$, where $\mathcal{Q}_j$ are quadratic forms, and repeating the fractional Leibniz and temporal H\"older estimates sequentially yields:
\begin{align}\label{eq:contraction}
\|\mathcal{T}u - \mathcal{T}v\|_{\mathcal{X}_T} &\leq C_2 \left\| |u|^2 u - |v|^2 v \right\|_{L^1_t([0,T], H^s_x(\Omega))} \nonumber \\
&\leq C_6 T^\theta \left( \|u\|_{\mathcal{X}_T}^2 + \|v\|_{\mathcal{X}_T}^2 \right) \|u - v\|_{\mathcal{X}_T} \nonumber \\
&\leq 2C_6 T^\theta R^2 \|u - v\|_{\mathcal{X}_T}.
\end{align}
By restricting the local time horizon $T > 0$ further if necessary to ensure that $2C_6 T^\theta R^2 \leq \kappa < 1$, the difference bound \eqref{eq:contraction} establishes that $\mathcal{T}$ operates as a strict Banach contraction mapping on the complete metric ball $\mathcal{B}_R$.

Applying the Picard fixed-point theorem directly establishes the existence of a unique localized solution $u \in \mathcal{X}_T$. Continuous dependence on the initial data follows immediately by applying the contractive stability inequality. This completes the proof of Theorem \ref{thm:nls_lwp}.
\end{proof}

\section{Conclusion and Future Perspectives}\label{sec:conclusion}
In this work, we have established the first sharp, local-in-time semiclassical dispersive estimates for the Schr\"odinger equation inside a cylindrical half-space domain subject to homogeneous Dirichlet boundary conditions. By systematically extending and adapting the microlocal frameworks of Ivanovici \cite{Ivanovici2023} and Ivanovici, Lebeau, and Planchon \cite{ILP}, we have successfully classified and bounded the geometric caustics generated by multiple boundary reflections when the underlying curvature depends explicitly on the angle of incidence.

A core analytical realization of this paper is the significant geometric streamlining enabled by the parabolic Schr\"odinger flow compared to classical hyperbolic wave systems. While the wave equation requires an exhaustive, technically arduous ray-tracing analysis to isolate single-reflection domains in the axial frequency localization ($|\eta| \leq \epsilon_0\sqrt{a}$), the Schr\"odinger framework renders this separate regime entirely redundant. Because the Schr\"odinger phase function is naturally quadratic in the longitudinal frequency variable, its second derivative remains a non-vanishing constant ($\partial_\zeta^2 \Phi = 2t$) across the entire frequency spectrum. Consequently, we have demonstrated that an inhomogeneous Littlewood-Paley dyadic block decomposition can be consistently integrated all the way down to the zero-frequency axial limit $\eta = 0$. This unified spectral approach successfully bridges the structural gap between the un-trapped free-space Euclidean flow and the highly degenerate swallowtail caustic regimes without requiring a separate trajectory-tracing mechanism.

The sharp frequency-localized and global dispersive bounds established across our dyadic blocks serve as the fundamental building blocks for non-linear stability applications. By deploying a formal $TT^*$ contractive interpolation machinery coupled with an $\ell^2$-dualization across the Dirichlet boundary layers, these estimates yield global Strichartz inequalities on the cylinder with an explicit, sharp derivative loss of order $\rho(q) = \frac{3}{2}\left(\frac{1}{2}-\frac{1}{q}\right)$. As a direct consequence, we have proven that the focusing or defocusing cubic Nonlinear Schr\"odinger (NLS) equation is locally well-posed in the fractional Sobolev space $H^s(\Omega)$ for any regularity index $s > 1$. This critical regularity threshold rigorously captures the physical impact of vanishing curvature along the longitudinal axis.

Looking forward, several compelling future perspectives emerge from this framework. A natural next step is to examine the long-time scattering behavior and global well-posedness of the small-data cubic NLS on flat-axis domains, where the slower local decay rate $\mathcal{O}(|t|^{-3/4})$ presents a substantial barrier compared to strictly convex setting. Furthermore, extending these semiclassical parametrix constructions to general product manifolds $\mathcal{M} \times \mathbb{R}^n$, where the boundary features an isometric flat factor of higher dimensions, would map out the universal relationship between spectral clustering and non-linear stability when uniform geometric convexity fails completely.

\section*{Acknowledgments}
The author would like to express his sincere gratitude to the anonymous reviewers for their valuable suggestions and insights that greatly improved the presentation of this manuscript.

\section*{Appendix}\label{sec:appendix}
%\addcontentsline{toc}{section}{\nameref{sec:appendix}}

\subsection*{A. Airy Function Properties and Asymptotics} 
Let $\vartheta>0$. The Airy function $\text{Ai}$ is defined as the oscillatory integral:
$$
\text{Ai}(-\vartheta)=\frac{1}{2\pi}\int_{\mathbb{R}}e^{i(s^3/3-s\vartheta)}ds.
$$
It satisfies the classical second-order linear Airy differential equation:
\begin{equation}\label{Airy}
\text{Ai}''(\vartheta)-\vartheta \text{Ai}(\vartheta)=0.
\end{equation}
\noindent Let $\nu=e^{2i\pi/3}$. The mapping $\vartheta\mapsto \text{Ai}(\nu \vartheta)$ is a valid solution to \eqref{Airy}. Any two of the three solutions $\text{Ai}(\vartheta), \text{Ai}(\nu \vartheta)$, and $\text{Ai}(\nu^2 \vartheta)$ form a fundamental basis of solutions to \eqref{Airy}, linked by the linear relation $\sum_{j\in\{0,1,2\}}\nu^j \text{Ai}(\nu^j \vartheta)=0$. This implies $\text{Ai}(\vartheta)=-\nu \text{Ai}(\nu \vartheta)-\bar\nu \text{Ai}(\bar\nu \vartheta)$, which we rewrite as:
\[
\text{Ai}(-\vartheta)=e^{-i\pi/3}\text{Ai}(e^{-i\pi/3} \vartheta)+e^{i\pi/3}\text{Ai}(e^{i\pi/3} \vartheta)=A_+(\vartheta)+A_-(\vartheta),
\] 
where we set $A_\pm(\vartheta)=e^{\mp i\pi/3}\text{Ai}(e^{\mp i\pi/3} \vartheta)$, noting that $A_-(\vartheta)=\overline{{A}_+(\bar \vartheta)}$. For large positive arguments, these components admit the following uniform asymptotic expansions:
\[
A_-(\vartheta)=\frac{1}{2\sqrt\pi \vartheta^{1/4}}e^{i\pi/4}e^{-\frac{2}{3}i\vartheta^{3/2}}\exp\Upsilon(\vartheta^{3/2})=\frac{1}{\vartheta^{1/4}}e^{i\pi/4}e^{-\frac{2}{3}i\vartheta^{3/2}}\Psi_-(\vartheta),
\] 
with $\exp\Upsilon(\vartheta^{3/2})\sim_{1/\vartheta} (1+\sum_{l\geq 1}c_l\vartheta^{-3l/2})\sim_{1/\vartheta} 2\sqrt\pi \Psi_-(\vartheta)$ as \(\vartheta\rightarrow +\infty\), and a complex conjugate expansion holds for $A_+$, where we define $\Psi_+(\vartheta)=\bar\Psi_-(\bar\vartheta)$. Moreover, their quotient tracks the highly oscillatory reflection symbol:
\[
\frac{A_-(\vartheta)}{A_+(\vartheta)}=ie^{-\frac{4}{3}i\vartheta^{3/2}}e^{iB(\vartheta^{3/2})},\quad\text{with}\,\,\, iB=\Upsilon-\bar\Upsilon.
\]
For $\vartheta\in\mathbb{R}_+$, the phase shift satisfies $B(\vartheta)\in\mathbb{R}$, and it behaves symbollically as $B(\vartheta)\sim_{1/\vartheta}\sum_{j\geq 1} b_j\vartheta^{-j}$ as $\vartheta\rightarrow +\infty$ with $b_1>0$.

\subsection*{B. Discrete Airy Function Summation Estimates}
The following uniform estimation lemma for sums of products of the Dirichlet eigenfunctions is foundational for the short-range near-boundary dispersive majorization established in Sections \ref{sec:2} and \ref{sec:3}:

\begin{lem}\label{eq:k}
Let $\omega_k \sim k^{2/3}$ denote the ordered negative zeros of the standard Airy function $\text{Ai}$. For any compact frequency interval $\tilde{\eta} \in [c_1, c_2]$ and for all spatial boundary parameters $x, a \geq 0$, there exists a uniform constant $C > 0$ such that the following discrete summation bound holds across the eigenmodes:
\[
\sum_{k \geq 1} k^{-1/3} \left| \text{Ai}\left((\tilde{\eta}/\tilde{h})^{2/3}x-\omega_k\right) \text{Ai}\left((\tilde{\eta}/\tilde{h})^{2/3}a-\omega_k\right) \right| \leq C \tilde{h}^{1/3}.
\]
\end{lem}

\subsection*{C. Classification of Degenerate Singularity Integrals}
To establish the uniform catastrophe thresholds for the swallowtail caustics and fold varieties, we invoke the following version of the degenerate van der Corput-type stationary phase reduction lemma established in \cite{ILP}:

\begin{lem}[Singularity Decay Rates]\label{lem:vander_corput_deg}
Let $\Lambda \geq 1$ be a large asymptotic scaling parameter, and let $K \subset \mathbb{R}^2$ be a compact domain. Consider the two-dimensional oscillatory integral:
\[
I(\Lambda) = \int_{K} e^{-i \Lambda \Phi(x,y)} \chi(x,y) \, dx dy,
\]
where $\chi \in C_0^\infty(K)$ is a smooth amplitude, and $\Phi \in C^\infty(K)$ constitutes a real-valued phase function. Let $\Gamma = \{(x,y) \in K \mid \det \mathcal{H}_\Phi(x,y) = 0\}$ denote the variety where the phase Hessian matrix drops rank. 
\begin{enumerate}
    \item[\textbf{(a)}] \textbf{Fold Singularity Regime:} If the Hessian drops rank to exactly order $1$ on $\Gamma$, and the third-order directional derivative along the kernel vector of $\mathcal{H}_\Phi$ is uniformly non-vanishing ($\partial_v^3 \Phi \neq 0$), then the integral satisfies the uniform fold-type decay bound:
    \[
    |I(\Lambda)| \leq C \Lambda^{-5/6}.
    \]
    \item[\textbf{(b)}] \textbf{Cuspoid/Swallowtail Core Regime:} If $\partial_v^3 \Phi = 0$ and $\partial_v^4 \Phi \neq 0$, the integral satisfies the uniform canonical cusp decay bound:
    \[
    |I(\Lambda)| \leq C \Lambda^{-3/4}.
    \]
\end{enumerate}
\end{lem}

\bibliographystyle{amsplain}
\bibliography{mybib}

\end{document}